\documentclass[twoside]{article}

\usepackage[accepted]{aistats2026}

\usepackage[T1]{fontenc} 
\usepackage{natbib}
\usepackage{enumitem}
\setlist[itemize]{leftmargin=1.25em}
\setlist[enumerate]{leftmargin=1.25em}
\usepackage{booktabs}
\usepackage{multirow}
\usepackage{color}
\usepackage{hyperref}
\hypersetup{
    colorlinks=true,
    linkcolor=blue,
    citecolor=blue,      
    urlcolor=cyan,
    pdfpagemode=FullScreen,
}
\usepackage{amsmath, amsfonts, amssymb, amsthm}
\usepackage{mathtools}
\usepackage{bm}
\usepackage{algorithm, algpseudocode}
\usepackage{chngcntr}
\usepackage{adjustbox}

\mathtoolsset{showonlyrefs}

\newcommand{\RR}{\mathbb{R}}
\newcommand{\ZZ}{\mathbb{Z}}
\newcommand{\ZZp}{\ZZ_{>0}}
\newcommand{\ZZnn}{\ZZ_{\geq 0}}
\newcommand{\bbE}{\mathbb{E}}

\newcommand{\Ind}{\mathbb{I}}

\newcommand{\bone}{\boldsymbol{1}}

\newcommand{\cA}{\mathcal{A}}

\newcommand{\cE}{\mathcal{E}}
\newcommand{\cF}{\mathcal{F}}

\newcommand{\cI}{\mathcal{I}}
\newcommand{\cJ}{\mathcal{J}}
\newcommand{\cK}{\mathcal{K}}

\newcommand{\cM}{\mathcal{M}}
\newcommand{\cN}{\mathcal{N}}

\newcommand{\cR}{\mathcal{R}}
\newcommand{\cS}{\mathcal{S}}
\newcommand{\cT}{\mathcal{T}}

\newcommand{\cV}{\mathcal{V}}

\newcommand{\res}{r}

\newcommand{\sfL}{\mathsf{L}}
\newcommand{\sfR}{\mathsf{R}}

\newcommand{\Cov}{\mathrm{Cov}}
\newcommand{\Var}{\mathrm{Var}}

\newcommand{\htaudim}{\hat{\tau}_{\mathrm{DiM}}}
\newcommand{\htauols}{\hat{\tau}_{\mathrm{OLS}}}
\newcommand{\hresdim}{\hat{\tau}^{\mathrm{res}}_{\mathrm{DiM}}}

\DeclareMathOperator*{\argmin}{\arg\min}

\DeclareMathOperator*{\Id}{\mathrm{Id}}

\newcommand{\arm}{\alpha}
\newcommand{\Sone}{\cS_1}
\newcommand{\Szero}{\cS_0}
\newcommand{\Sa}{\cS_{\arm}}

\newcommand{\na}{n_{\arm}}

\newcommand{\OLS}{\mathsf{OLS}}
\newcommand{\colsp}{\mathrm{col}}

\newcommand{\Vstar}{V^\star}
\newcommand{\Rstar}{R^\star}
\newcommand{\Bstar}{B^\star}

\newcommand{\ycenta}[1]{\tilde{y}^{(\arm)}_{#1}}

\newcommand{\ycent}[1]{y_{#1}^{\mathrm{c}}}

\newcommand{\rem}{m_t}
\newcommand{\nass}{\ell_t}
\newcommand{\Sprox}[1]{\cS^{\mathrm{prox}}_{#1}}

\newcommand{\Phidel}{\Phi^{\mathrm{del}}}
\newcommand{\Phiins}{\Phi^{\mathrm{ins}}}

\newcommand{\CR}{\widetilde{\mathrm{CR}}}

\newcommand{\EmpVar}{\mathrm{Var}_{\mathrm{emp}}}
\newcommand{\DIM}{\mathrm{DiM}}
\newcommand{\ols}{\mathrm{OLS}}

\DeclareMathOperator*{\esssup}{\mathrm{ess} \sup}

\newcommand{\Xa}{X_{\arm}}
\newcommand{\ya}{y_{\arm}}

\newcommand{\Ca}{C_{\arm}}

\newcommand{\xbara}{\bar{x}_{\arm}}

\newcommand{\ybara}{\bar{y}_{\arm}}
\newcommand{\onea}{\boldsymbol{1}_{\na}}

\newcommand{\Unif}{\mathrm{Unif}}

\theoremstyle{plain}
\newtheorem{theorem}{Theorem}
\newtheorem{proposition}[theorem]{Proposition}
\newtheorem{lemma}{Lemma}[section]
\newtheorem{corollary}[theorem]{Corollary}
\newtheorem{definition}{Definition}[section]
\newtheorem{example}{Example}
\newtheorem{assumption}{Assumption}
\theoremstyle{remark}
\newtheorem{remark}{Remark}[section]

\begin{document}
%

%

\twocolumn[

\aistatstitle{Design-Based Finite-Sample Analysis for Regression Adjustment}

\aistatsauthor{Dogyoon Song}

\aistatsaddress{University of California, Davis} ]

\begin{abstract}
    In randomized experiments, regression adjustment can improve the precision of average treatment effect (ATE) estimation using covariates without requiring a correctly specified outcome model. 
    Although well studied in low-dimensional settings, its behavior in high-dimensional regimes, where the number of covariates $p$ may exceed the number of observations $n$, remains underexplored. 
    Moreover, existing analyses are largely asymptotic, providing limited guidance for finite-sample inference. 
    We develop a design-based, non-asymptotic framework for analyzing the regression-adjusted ATE estimator under complete randomization. 
    This yields finite-sample-valid confidence intervals with explicit, instance-adaptive widths, even when $p > n$. 
    While these intervals rely on oracle (population-level) quantities, we also outline data-driven envelopes computable from observed data. 
    Our approach hinges on a refined swap sensitivity analysis of an estimator: stochastic fluctuation is controlled via a variance-adaptive Doob martingale and Freedman's inequality, and design bias is bounded by Stein's method of exchangeable pairs. 
    The analysis elucidates how covariate geometry governs concentration and bias of the adjusted estimator, suggesting when and how regression adjustment can be effective. 
\end{abstract}

\section{INTRODUCTION}

Randomized experiments are a central tool for causal inference across scientific domains; see, e.g., \citet{fisher1935design, neyman, kempthorne1952design} for foundational developments and \citet{imbens2015causal, rosenberger2015randomization, ding2024first} for modern textbook treatments.
In a treatment–control experiment with units $i\in[n]$ and potential outcomes $\{(y_i^{(1)},y_i^{(0)})\}_{i=1}^n$, the average treatment effect (ATE) $\tau = n^{-1} \sum_{i=1}^n ( y_i^{(1)} - y_i^{(0)} )$ is often the target of estimation. 
Two inferential viewpoints are common. 
Following \citet{neyman}, the \emph{design-based} (finite-population) perspective treats the potential outcomes as fixed, with randomness solely arising from the experimenter's treatment assignment \citep{aronow2014sharp, dasgupta2015causal,fogarty2018regression,mukerjee2018using, li2020rerandomization}.
By contrast, the superpopulation view models units as i.i.d. draws from a hypothetical infinite population \citep{tsiatis2008covariate, cattaneo2018inference}. 
We adopt the design-based, finite-population perspective to foreground and study the role of randomization without positing an outcome‑generating model.

A natural baseline for ATE estimation is the difference-in-means (DiM) estimator $\hat\tau_{\mathrm{DiM}}$, which is unbiased under complete randomization with a known variance formula \citep{neyman}.
When pre-treatment covariates $x_i\in\RR^p$ are predictive of outcomes, regression adjustment (RA) can improve precision by incorporating them.
Classical analysis of covariance (ANCOVA) often assumes a constant treatment effect, under which the ordinary least squares (OLS) of $Y_i$ on $(T_i,x_i)$ with a treatment indicator $T_i$ is a traditional approach \citep{fisher1935design,kempthorne1952design,hinkelmann2007design}.
\citet{freedman2008regression} highlighted three concerns: (i) RA can be less efficient than DiM under heterogeneity; (ii) usual OLS standard errors can be inconsistent; and (iii) the OLS estimator has $O_p(n^{-1})$ bias.
In fixed dimensions these issues are resolved by including all treatment–covariate interactions (equivalently, two arm‑wise OLS fits): 
\citet{lin_ols2013} shows the resulting estimator is consistent, asymptotically normal, and at least as efficient as DiM; \citet{chang2024exact} provide exact bias corrections with and without interaction.

Modern experiments often entail many covariates, and fixed‑$p$ asymptotics may fail to capture the behavior of RA. 
Recent work extends design‑based asymptotics to increasing‑$p$ regimes, e.g., asymptotic normality up to $p=o(n^{1/2})$, and then up to $p=o(n^{2/3})$ with a bias correction \citep{lei2021regression}, or further up to $p=o(n^{3/4})$ with cross‑fitting \citep{chiang2023regression}.
A complementary line of work considers high‑dimensional settings and provides guarantees under sparsity or semiparametric structure \citep{bloniarz2016lasso,wager2016high,chernozhukov2018double}. 
Nevertheless, a \emph{design-based}, \emph{finite-sample} theory for RA that remains valid in high‑dimensional and interpolating regimes---without additional ad hoc techniques or modeling assumptions---remains scarce. 
Moreover, practitioners typically face a single finite dataset, so asymptotic assurances alone may not offer actionable uncertainty quantification.

\paragraph{Contributions.}
In this paper, we develop a design-based, non‑asymptotic framework for analyzing ATE estimators under complete randomization. 
Our contributions in this paper are four‑fold.

\begin{enumerate}[label=(\roman*)]
    \item
    \emph{Design concentration and bias bounds} (Section \ref{sec:finite_sample_analysis}).  
    Viewing an ATE estimator as a function $f(\cS)$ of the treatment set $\cS$, we derive a \emph{variance‑adaptive} concentration inequality using swap‑sensitivity analysis coupled with a Doob martingale and Freedman’s inequality. 
    We also obtain a \emph{design‑bias} upper bound via Stein’s method of exchangeable pairs on the Johnson graph. 
    Both results apply to general set functions beyond ATE estimators.

    \item
    \emph{Finite-sample-valid inference} (Section \ref{sec:main-result-oracle}). 
    Combining these bounds yields an instance‑adaptive error radius $\varepsilon_{\cS}$ such that, for any $\delta \in (0,1)$,
    \[
        \Pr \bigl( |\hat\tau(\cS)-\tau| > \varepsilon_{\cS} \bigr) \leq \delta,
    \]
    producing finite‑sample‑valid $(1 - \delta)$ confidence intervals under complete randomization, without requiring asymptotic limits (Theorem~\ref{thm:oracle}).

    \item 
    \emph{Oracle swap sensitivity and covariate $Q$-geometry} (Section \ref{sec:main_OLS_RA}).
    Specializing to OLS adjustment, we derive a closed‑form \emph{paired deletion–insertion} identity (Theorem \ref{thm:perturbation_identity}) for the \emph{oracle} one‑swap change $\Delta_{ij}\hat\tau_{\mathrm{OLS}}(\cS)$ with respect to exchanging units $i\leftrightarrow j$ by applying rank‑one pseudoinverse updates to the ratio‑of‑quadratics representation of the intercepts (Proposition~\ref{prop:mu_ratio}). 
    The resulting formulas suggest that the covariate $Q$‑geometry governs sensitivity via leverages and standardized residual energy.

    \item 
    \emph{Toward data‑driven proxies} (Appendix \ref{sec:additional_discussion}). 
    We isolate the observable $Q$-geometric factors that enter these oracle envelopes and discuss how they may serve as practical proxies; establishing fully data-driven insertion envelopes is left as future work.
\end{enumerate}

\paragraph{Related Work.}
Regression adjustment, dating back to Fisher's ANCOVA, has long been used to improve precision in randomized experiments by leveraging pre-treatment covariates \citep{fisher1935design,kempthorne1952design,imbens2015causal}. 
\citet{freedman2008regression} criticized pitfalls of RA (efficiency loss under heterogeneity, inconsistent OLS standard errors, and $O_p(n^{-1})$ bias), while \citet{lin_ols2013} advocated fitting all treatment--covariate interactions (equivalently, arm-wise OLS with an intercept) to achieve asymptotic normality and non-inferiority to the difference-in-means (DiM) baseline. 
Recent refinements provide exact bias corrections and sharpen inference in complete randomization \citep[]{chang2024exact, fogarty2018regression}.

Beyond the classical fixed-$p$ regime, recent work establishes design-based asymptotics for RA as $p$ grows with $n$, including asymptotic normality up to $p=o(n^{1/2})$, extensions to $p=o(n^{2/3})$ with a degree-0 bias correction, and to $p=o(n^{3/4})$ via cross-fitting, under leverage regularity conditions \citep{lei2021regression,chiang2023regression}. 
This line of work is further extended to establish asymptotic normality up to $p=o(n^{\tfrac{d+2}{d+3}})$ with degree-$d$ corrections for all $d \geq 0$ \citep{song2025neumann}. 
High-dimensional methods under sparsity or semiparametric structure (e.g., lasso and double ML) also provide complementary guarantees, albeit typically in superpopulation or structured settings \citep{bloniarz2016lasso,chernozhukov2018double}. However, finite-sample, design-based guarantees that remain valid when $p \ge n$, without additional modeling assumptions, are still scarce; this paper aims to bridge this gap.

For the analysis, we use Freedman's inequality \citep{freedman1975tail} together with a Doob martingale tailored to sampling without replacement, to obtain a variance-adaptive concentration bound. 
We bound the design bias via Stein's method of exchangeable pairs on the Johnson graph of treatment assignment sets \citep{ross2011fundamentals, chen2010normal, delsarte1973algebraic,levin2017markov}. 
For OLS-RA, we derive exact swap sensitivity formulas from rank-one pseudoinverse updates applied to a ratio-of-quadratics representation of the OLS intercepts, exposing the associated covariate geometry \citep{greville1960some,ben2003generalized}.

\paragraph{Organization.} 
Section~\ref{sec:setup} introduces the setup and design-based preliminaries for ATE estimation and RA. 
Section~\ref{sec:finite_sample_analysis} decomposes the estimation error into stochastic fluctuation and design bias, yielding finite‑sample bounds for each (Propositions~\ref{prop:concentration} and \ref{prop:stein-bias}). 
Section~\ref{sec:main_results} establishes an oracle finite‑sample‑valid inferential guarantee (Theorem~\ref{thm:oracle}); Section~\ref{sec:main_OLS_RA} then specializes to OLS-RA and examines how covariate geometry affects RA. 
Section~\ref{sec:experiments} and Appendix~\ref{sec:deferred_experiments} report numerical experiments, and Section~\ref{sec:conclusion} concludes.

\section{SETUP AND BACKGROUND}\label{sec:setup}

\paragraph{Notation.} 
Let $\RR$ denote the set of real numbers, and $\ZZ$ the set of integers. 
Let $\ZZnn \coloneqq \{ z \in \ZZ: z \geq 0\}$ and $\ZZ_{> 0} \coloneqq \{ z \in \ZZ: z > 0\}$.
For $n \in \ZZp$, write $[n] \coloneqq \{1, \dots, n\}$. 
A vector $v \in \RR^n$ is identified with its column vector representation $v \in \RR^{n \times 1}$. 
For a matrix $M$, $M^{\top}$ denotes its transpose, and $M^{-1}$ its inverse (when it exists). 
For readability, we use calligraphic letters to denote sets (e.g., $\cS$). 
We write $\bone_n \coloneqq (1, \cdots, 1) \in \RR^n$ and $\Ind_E$ for the indicator that equals 1 if and only if $E$ is true (and 0 otherwise).

\subsection{Average Treatment Effect Estimation}
\paragraph{Finite Population and Potential Outcomes.}
We consider a finite population of $n$ units. 
Each unit $i\in[n]$ has two potential outcomes $y_i^{(\arm)}$ for $\arm\in\{0,1\}$, where $y_i^{(1)}$ is the treated outcome and $y_i^{(0)}$ is the control outcome. 
Collect $y^{(\arm)}=(y_1^{(\arm)},\dots,y_n^{(\arm)})\in\RR^n$.

\paragraph{Randomized Experiments.}
A completely randomized design selects a treatment set $\Sone \subset [n]$ of a fixed size $n_1$ uniformly at random; the control set is $\Szero \coloneqq [n] \setminus \Sone$, with $n_0 = |\Szero| = n - n_1$. 
Let $T_i \coloneqq \Ind\{ i \in \Sone \} \in \{0,1\}$ denote the treatment indicator. 
Then the observed outcome is
\begin{equation}\label{eqn:observed_model}
    Y_i = T_i y_i^{(1)} + (1 - T_i) y_i^{(0)}. 
\end{equation}
Randomness arises \emph{only} from the treatment assignment (that is, $y^{(1)}, y^{(0)}$ are deterministic). 

\medskip

\begin{remark}
    Note that in our setup, the so-called stable unit treatment value assumption (SUTVA) is implicitly imposed by \eqref{eqn:observed_model}, which posits  (i) ``no interference'' (each unit’s potential outcomes do not depend on other units’ assignments) and (ii) ``consistency'' (there are no other versions of treatment) conditions. 
\end{remark}

For each $\arm \in \{0,1\}$, let $S_{\arm} \in \{0,1\}^{n_{\arm} \times n}$ be the selection matrix such that $(S_{\arm})_{ij} = 1$ if and only if $j$ is the $i$-th smallest element in $\Sa$. 
Define $\ya=S_{\arm} y^{(\arm)}$, the \emph{observed} subvector of $y^{(\arm)}$ corresponding to $\Sa$.

\paragraph{Average Treatment Effect (ATE) Estimation.}
The finite-population average treatment effect is
\[
    \tau = \frac{1}{n} \sum_{i=1}^n (y^{(1)}_i - y^{(0)}_i).
\]
The difference-in-means (DiM) estimator of ATE is
\begin{equation}\label{eqn:DiM}
    \hat\tau_{\mathrm{DiM}} = \frac{1}{n_1}\sum_{i\in\Sone}y_i^{(1)} - \frac{1}{n_0}\sum_{i\in\Szero}y_i^{(0)}.
\end{equation}
Together with complete randomization, SUTVA makes the finite‑population ATE identifiable and the difference‑in‑means (DiM) estimator unbiased.

\subsection{Regression Adjustment}\label{sec:reg_adjustment}
\paragraph{Covariates, Design and Normalization.}
Each unit $i$ carries pre-treatment covariates $x_i\in\RR^p$. 
The design matrix refers to $X=\begin{bmatrix}x_1&\cdots&x_n\end{bmatrix}^\top\in\RR^{n\times p}$. 

\setcounter{assumption}{-1}
\begin{assumption}\label{assump:covariates}
    Suppose that
    \begin{enumerate}
        \item 
        $\bone_n^\top X=0$, i.e., columns of $X$ are centered, and
        \item 
        $X^\top X=nI_p$, i.e., columns of $X$ are orthogonal and have norm $\sqrt{n}$.
    \end{enumerate}
\end{assumption}

\begin{remark}
    Assumption \ref{assump:covariates} is only used for Section \ref{sec:setup}. 
\end{remark}

For $\arm \in \{0,1\}$, let $(\mu_{\arm},\beta_{\arm})$ denote the \emph{population} ordinary least squares (OLS) coefficients from regressing $y^{(\arm)}$ on $[\bone_n,X]$, and define the population residuals
\begin{equation}\label{eqn:population_OLS}
    \res^{(\arm)}  \coloneqq  y^{(\arm)} - \mu_{\arm} \bone_n - X\beta_{\arm},
        \qquad \arm\in\{0,1\}.
\end{equation}
By OLS orthogonality, $\bone_n^\top \res^{(\arm)}=0$ and $X^\top \res^{(\arm)}=0$. 
Computing the population OLS coefficients $(\mu_{\arm},\beta_{\arm})$ is only a notional device to aid in our analysis; it requires access to the potential outcome $y^{(\arm)}$ for the entire population, but we only observe $\ya = S_{\arm} y^{(\arm)}$.

\paragraph{Arm-wise OLS with Intercept.} 
For each $\arm \in \{0,1\}$, let $(\hat\mu_{\arm},\hat\beta_{\arm})$ be the arm-wise OLS coefficients obtained by fitting observed data:
\begin{equation*}
    (\hat\mu_{\arm},\hat\beta_{\arm}) \in \arg\min_{(\mu, \beta)} \left\{ \frac{1}{n_{\arm}} \sum_{i \in \Sa}\bigl( y^{(\arm)}_i - \mu - x_i^{\top} \beta  \bigr)^2 \right\}.
\end{equation*}
For each $\arm \in \{0,1\}$, define $\Xa=S_{\arm}X$ as the submatrix of $X$ that retains only the rows with indices in $\Sa$, and define arm-wise means 
\begin{equation}\label{eqn:armwise_means}
    \xbara\coloneqq \frac{1}{n_{\arm}} \Xa^\top\onea
    \qquad\text{and}\qquad
    \ybara\coloneqq \frac{1}{n_{\arm}} \onea^\top \ya.
\end{equation}
Define the arm-wise (row-) centering projection matrix $\Ca \coloneqq I_{n_{\arm}} - \frac{1}{n_{\arm}}\onea\onea^\top$ so $\Ca\onea = 0$.
Then $\Ca X_{\arm} = X_{\arm} - \onea \xbara^{\top}$, and the arm-wise sample covariance 
\begin{equation}\label{eqn:armwise_covariance}
    \Sigma_{\arm} 
        \coloneqq  \frac{1}{n_{\arm}}  \sum_{i \in \Sa} ( x_i - \xbara )( x_i - \xbara )^{\top}
        = \frac{1}{n_{\arm}} \Xa^\top \Ca \Xa .
\end{equation}
In the classical regime where \(\Sigma_{\arm}\) is invertible, standard OLS-with-intercept algebra gives
\begin{equation*}
    \hat\beta_{\arm} = \Sigma_{\arm}^{-1} \Bigl(\frac{1}{n_{\arm}}\Xa^\top \Ca \ya\Bigr),
    \qquad
    \hat\mu_{\arm} = \ybara - \xbara^\top \hat\beta_{\arm}.
\end{equation*}
Substituting $\ya=\mu_{\arm}\onea+\Xa\beta_{\arm}+S_{\arm} \res^{(\arm)}$ and using $\Ca\onea=0$ and $X_{\arm}^\top\Ca\onea=0$ yields
\begin{equation} \label{eq:beta-sample-minus-pop}
    \hat\beta_{\arm} = \beta_{\arm} + \Sigma_{\arm}^{-1} \Bigl(\frac{1}{n_{\arm}}\Xa^\top \Ca S_{\arm} \res^{(\arm)}\Bigr). 
\end{equation}
Finally,
\begin{equation}
    \hat\mu_{\arm} - \mu_{\arm} = \underbrace{\frac{1}{n_{\arm}}\onea^\top S_{\arm} \res^{(\arm)}}_{\text{residual mean in arm $\arm$}}
    - 
    \xbara^\top\bigl(\hat\beta_{\arm}-\beta_{\arm}\bigr).
    \label{eq:mu-deviation}
\end{equation}

\paragraph{Regression Adjustment.}
Recall the population residuals in \eqref{eqn:population_OLS}. 
Since $\bone_n^\top X=0$ and $\bone_n^\top \res^{(\arm)}=0$, 
\begin{equation}\label{eqn:ols_armwise}
    \tau
        =\frac{1}{n}\bone_n^{\top} y^{(1)} - \frac{1}{n}\bone_n^{\top} y^{(0)} 
        = \mu_1-\mu_0.
\end{equation}
The regression-adjusted ATE estimator via OLS (or, the OLS-adjusted ATE estimator) refers to a natural plug-in estimator
\begin{equation}\label{eqn:RA_OLS}
    \hat\tau_{\mathrm{OLS}}\coloneqq \hat\mu_1-\hat\mu_0.
\end{equation}

\subsection{Asymptotics of Regression Adjustment}
For fixed $p$, OLS adjustment with treatment-covariate interactions (equivalently, arm-wise OLS with intercept as in \eqref{eqn:RA_OLS}) is a classical method that is asymptotically non-inferior to DiM under standard regularity. 
In particular, \citet[Theorem 1]{lin_ols2013} show
\[
    \sqrt{n}( \hat\tau_{\mathrm{OLS}}-\tau ) \ \xRightarrow{d}\ \cN\bigl(0, \sigma^2_{\mathrm{res}}\bigr),
    \quad
    \sigma^2_{\mathrm{res}} \coloneqq \frac{v_1}{\rho}+\frac{v_0}{1-\rho}-v_\tau
\]
where
$v_{\arm} \!\coloneqq \lim_{n \to \infty} \!\frac{1}{n}\sum_{i=1}^n\bigl(r^{(\arm)}_i\bigr)^2$ for each $\arm \in \{0,1\}$, and $v_\tau\! \coloneqq \lim_{n \to \infty} \!\frac{1}{n}\sum_{i=1}^n\bigl(r^{(1)}_i \!-\! r^{(0)}_i\bigr)^2$. 
For $p$ growing with $n$, \citet{lei2021regression} establish asymptotic normality (AN) for OLS--RA up to $p=o(n^{1/2})$ (modulo logs), and extend AN to $p=o(n^{2/3})$ with de-biasing.

At the core of these classical fixed-\(p\) analyses is the decomposition obtained by combining \eqref{eq:beta-sample-minus-pop}--\eqref{eq:mu-deviation} and $\tau=\mu_1-\mu_0$:
\begin{equation}\label{eq:master-decomp}
    \hat\tau_{\mathrm{OLS}} - \tau
        = \underbrace{ \frac{1}{n_1} \bone_{n_1}^\top S_1 r^{(1)} - \frac{1}{n_0} \bone_{n_0}^\top S_0 r^{(0)}}_{\eqqcolon \hresdim;~~\text{Residual DiM}}
        - \underbrace{\Bigl(\sfR_1-\sfR_0\Bigr)}_{\text{Remainder}}
\end{equation}
where the arm-wise remainder for each $\arm \in \{0,1\}$ is
\begin{equation}\label{eqn:neumann_remainder}
    \sfR_{\arm} 
        \coloneqq \bar{x}_{\arm}^{\top} \bigl( \hat{\beta}_{\arm} - \beta_{\arm} \bigr) 
        = \bar x_{\arm}^\top \Sigma_{\arm}^{-1} \Bigl(\frac{1}{n_{\arm}}\Xa^\top \Ca S_{\arm} \res^{(\arm)}\Bigr).
\end{equation}
A sampling-without-replacement central limit theorem (e.g., \citet{hajek1960limiting}) yields $\sqrt{n} \cdot \hresdim \ \xRightarrow{d}\ \cN\bigl(0, \sigma^2_{\mathrm{res}}\bigr)$; if the remainder is $o_P(n^{-1/2})$, the AN of $\hat\tau_{\mathrm{OLS}}-\tau$ follows. 
Thus, ensuring $\sfR_1 - \sfR_0 = o_P(n^{-1/2})$ provides a sufficient condition for asymptotic normality.

\paragraph{Gaps and Motivation.}
Despite strong asymptotic results for regression adjustment, several gaps remain.

Existing guarantees based on the residual–DiM decomposition \eqref{eq:master-decomp} are fundamentally \emph{asymptotic}: they establish pointwise AN and consistency of variance estimators, but do not by themselves furnish \emph{finite‑sample, design‑based} confidence intervals with explicit, verifiable constants. Moreover, in high‑dimensional or interpolating regimes, OLS can nearly interpolate outcomes, making residuals small and the residual–DiM route numerically unstable (or infeasible); then the remainder $\sfR_1-\sfR_0$ need not be negligible without additional control.
Alternatives that rely on sparsity, cross‑fitting/double ML, or superpopulation modeling provide asymptotic guarantees but again lack finite‑sample, design‑based error control. 

To our knowledge, there remains a paucity of \emph{non‑asymptotic}, finite‑population results that (a) operate under complete randomization, (b) remain valid when $p$ is comparable to or exceeds the arm sizes, and (c) yield ready‑to‑use confidence intervals with explicit constants.
\section{FINITE-SAMPLE ANALYSIS OF CONCENTRATION AND BIAS}\label{sec:finite_sample_analysis}


We depart from the residual-DiM decomposition \eqref{eq:master-decomp} and work directly with the minimum-norm OLS estimator: least squares (LS) with an unpenalized intercept and the minimum-$\ell_2$-norm slope among all LS minimizers. 
This estimator admits a closed‐form intercept as a ratio of quadratic forms, valid in both classical and over‑parameterized regimes.

\begin{definition}\label{def:min_norm_OLS}
    Given $X \in \RR^{n \times p}$ and $y \in \RR^n$, the minimum-norm OLS estimator (with unpenalized intercept) of $(X,y)$, denoted by $\OLS(X,y)$, is the unique minimizer
    \begin{equation}\label{eqn:min_norm_OLS}
    \begin{aligned} 
        (\hat{\mu}, \hat{\beta}) &\in \argmin_{(\mu,\beta) \in \cS_{X,y}} \| \beta\|_2^2, \qquad\text{where}\\
        \cS_{X,y} &\coloneqq \argmin_{(\mu,\beta) \in \RR \times \RR^p} \| y - X\beta - \mu \bone_n \|_2^2. 
    \end{aligned}
    \end{equation}
\end{definition}

\begin{remark}
    When $\cS_{X,y}$ is a singleton (e.g., if $n \geq p$ and $\bone_n \not\in \colsp(X)$), the formulation \eqref{eqn:min_norm_OLS} reduces to the standard OLS with intercept.
\end{remark}

\begin{proposition}\label{prop:mu_ratio}
    Let $X \in \RR^{n \times p}$ and $y \in \RR^n$, and let $(\hat{\mu}, \hat{\beta}) = \OLS(X, y)$. 
    Then
    \begin{equation}\label{eq:mu-ratio}
        \hat{\mu} = \frac{ \bone_{n}^{\top} Q_{X}\, y }{ \bone_{n}^{\top} Q_{X}\, \bone_{n} },
    \end{equation}
    with
    \begin{equation}\label{eq:Qa-piecewise}
        Q_X = 
        \begin{cases}
            M_X \coloneqq I_n - X X^{\dagger},  & \text{if } \bone_n \not\in \colsp(X),\\
            K_X \coloneqq (X X^{\top})^{\dagger},  & \text{if } \bone_n \in \colsp(X).
        \end{cases}
    \end{equation}
\end{proposition}

\begin{remark}
    $X X^{\dagger}$ is the orthogonal projection matrix onto $\colsp(X)$, so-called the ``hat matrix.'' 
    $X X^{\top}$ is the Gram matrix of the row vectors $\{ x_1, \dots, x_n \}$. 
    The denominator in \eqref{eq:mu-ratio} is strictly positive, i.e., $\bone_n^\top Q_X \bone_n > 0$, because $Q_X$ is positive semidefinite and $\bone_n\notin\ker(Q_X)$ in both cases of \eqref{eq:Qa-piecewise}. 
    See Appendix \ref{sec:proof_proposition.quadratic} for a proof of Proposition \ref{prop:mu_ratio}.
\end{remark}

Recall that $\hat\tau_{\ols}\coloneqq \hat\mu_1-\hat\mu_0$ with arm-wise intercepts, cf. \eqref{eqn:RA_OLS}. 
Thus, $\hat\tau_{\ols}$ is a function of the random arm index sets $\Sa$ via $\hat{\mu}_{\arm}$, which depend on $\Sa$ through $Q_{\Xa}$ and $\ya$. 
Intuitively, if $\hat\tau_{\ols}$ is not very sensitive to a change in $\Sa$, then $\hat\tau_{\ols}$ concentrates around $\mathbb{E}[ \hat\tau_{\ols} ]$, which may still differ from the true ATE $\tau$ though.

\paragraph{Section Overview.} 

In what follows, we develop a design-based, non-asymptotic analysis for ATE estimators under complete randomization. 
Let $n \in \ZZnn$, $m \in [n]$, and $f: \binom{[n]}{m} \to \RR$ (e.g., $f = \hat\tau_{\ols}- \tau$). 
We seek a high-probability upper bound for $| f(\cS) |$. 
Note
\[
    | f(\cS) | \leq | f(\cS) - \bbE f(\cS) | + |\bbE f(\cS) |.
\]
Instantiating $f = \hat\tau_{\ols}- \tau$, the first term $f(\cS) - \bbE f(\cS) = \hat{\tau} - \bbE \hat{\tau}$ represents the stochastic fluctuation of the ATE estimator $\hat{\tau}$, and the second term $\bbE f(\cS) = \bbE \hat{\tau} - \tau$ is the bias of $\hat{\tau}$. 
In what follows, we describe our approach to control these two separately:
\begin{itemize}
    \item 
    Section \ref{sec:sensitivity_concentration}: 
    a high-probability bound on $\bigl| f(\cS) - \bbE f(\cS) \bigr|$ via concentration based on swap-sensitivity analysis with a variance-adaptive Doob martingale,
    \item 
    Section \ref{sec:bias_exchangeable}: 
    a deterministic upper bound on $\bigl| \bbE f(\cS) \bigr|$ via Stein's method of exchangeable pairs.
\end{itemize}
Together, these yield a high-probability upper bound on the estimation error $\bigl| \hat\tau_{\ols} - \tau \bigr|$, enabling finite‑sample‑valid confidence intervals.

\subsection{Concentration via Sensitivity Analysis}\label{sec:sensitivity_concentration}

We derive a high‑probability bound for $|f(\cS)-\bbE f(\cS)|$ under $\cS\sim\Unif\binom{[n]}{m}$.

\paragraph{Swap Sensitivity.} 
Recall the treatment set $\Sone \subset [n]$ is drawn uniformly at random from $\binom{[n]}{n_1}$, and the control set $\Szero = [n] \setminus \Sone$. 
The ATE estimators, such as $\hat\tau_{\DIM}$ in \eqref{eqn:DiM} and $\hat\tau_{\ols}$ in \eqref{eqn:RA_OLS}, are functions of the random subset $\Sone$. 
We quantify the sensitivity of such a random function to a swap operation $(i\leftrightarrow j)$.

\begin{definition}
    Let $n \in \ZZnn$, $m \in [n]$, and $f: \binom{[n]}{m} \to \RR$. 
    For any $\cS \in \binom{[n]}{m}$ and any $(i, j) \in \cS \times ([n]\setminus \cS)$, the \emph{(i,j)-swap sensitivity of f at $\cS$} is
    \begin{equation}\label{eqn:perturbation}
        \Delta_{ij} f(\cS) \coloneqq f \bigl( \cS^{(i\leftrightarrow j)}\bigr) - f(\cS)
    \end{equation}
    where
    \[
        \cS^{(i\leftrightarrow j)} \coloneqq \bigl( \cS \setminus \{i\} \bigr) \cup \{j\}.
    \]
\end{definition}

\begin{example}\label{example:DiM_sensitivity}
    For the difference-in-means estimator,
    \begin{align*}
        \Delta_{ij} \htaudim  = \frac{1}{n_1} \bigl( y^{(1)}_j - y^{(1)}_i \bigr) + \frac{1}{n_0} \bigl( y^{(0)}_j - y^{(0)}_i \bigr).
    \end{align*}
\end{example}

Computing the swap sensitivity of $\htauols$ is more involved; see Proposition \ref{prop:sensitivity_decomposition} and Theorem \ref{thm:perturbation_identity} for this.

\paragraph{Naive Bounded Differences Are Suboptimal.}
A (McDiarmid‑type) permutation bounded differences argument controls $|\Delta_{ij}f(\cS)|$ by a global Lipschitz envelope $c_{\mathrm{BD}}$, yielding tail bounds of the form $\exp\{-\Theta(t^2/(m\,c_{\mathrm{BD}}^2))\}$. 
This is typically suboptimal because $c_{\mathrm{BD}}$ reflects a worst‑case range (e.g., $\max_{i,j}|y^{(\arm)}_i-y^{(\arm)}_j|$) rather than the variance scale, and it ignores the correlations/cancellations across the swaps that connect realizable sets $\cS_1\leftrightarrow \cS_1'$. 
In contrast, in the asymptotic analysis of the DiM, the dispersion is governed by the variance proxy underlying $\sigma_{\mathrm{res}}^2=\frac{v_1}{\rho}+\frac{v_0}{1-\rho}-v_\tau$, often much smaller than worst‑case ranges. 
Therefore, a variance‑adaptive construction is needed for a useful concentration bound.

\paragraph{Assignment-exposure Martingale.} 
Encode complete randomization by a uniform permutation $\Pi$ of $[n]$, with $\Sone=\{\Pi(1),\dots,\Pi(n_1)\}$. 
Define the \emph{assignment-exposure filtration and martingale}
\begin{align}
    \cF_t &\coloneqq \sigma\big(\Pi(1),\ldots,\Pi(t)\big),\qquad t=0,1,\ldots,n_1,
        \label{eqn:filtration}\\
    M_t &\coloneqq \bbE[f(\Sone)\mid \cF_t].
        \label{eqn:martingale}
\end{align}
Then $M_t$ is a martingale with respect to $\cF_t$, and $D_t \coloneqq M_t - M_{t-1}$ is a martingale difference sequence. 
Furthermore, $\bbE[D_t\mid \cF_{t-1}]=0$ and
\[
    f(\Sone) - \bbE f(\Sone) = \sum_{t=1}^{n_1} D_t.
\]
Define the set of treated indices revealed up to step $t-1$, and the remaining pool at $t-1$: for $t \in [n_1]$,
\begin{equation}\label{eqn:past_remaining_sets}
    \cS^{\mathrm{past}}_{t-1} \coloneqq \{\Pi(1),\ldots,\Pi(t-1)\},
    \qquad
    \cR_{t-1} \coloneqq [n] \setminus \cS^{\mathrm{past}}_{t-1}.
\end{equation}
Let 
\[
    \rem \coloneqq |\cR_{t-1}| = n - (t-1), \qquad
    \nass \coloneqq n_1-(t-1),
\]
denote the number of total remaining units and the number of treated units not revealed yet, respectively.

\medskip

\begin{remark}
    Although $\cF_t$ reveals treated indices only, conditioning on $\cF_{t-1}$ also updates the law of (projected) terminal control set, which is uniform over the $n_0$-subsets of $\cR_{t-1}$.
    Thus both $\bbE[\hat\mu_1(\Sone)\mid\cF_t]$ and $\bbE[\hat\mu_0(\Szero)\mid\cF_t]$ evolve with $t$. 
    For example, for DiM, $\bbE[\hat\mu_0\mid\cF_{t-1}]$ is the average of $y^{(0)}$ over $\cR_{t-1}$ and shifts by a hypergeometric update when $I=\Pi(t)$ is revealed.
\end{remark}

Next, we express the increment $D_t$ as the average of a single paired swap around a feasible proxy set. 
Let $I \coloneqq \Pi(t)$. 
Let $J$ be uniform on $\cR_{t-1}\setminus\{I\}$ and $\cT$ uniform on $\binom{\cR_{t-1}\setminus\{I,J\}}{\nass-1}$, conditionally on $(\cF_{t-1},I)$. 
We define the proxy set of $\Sone$ based on information up to $t-1$
\begin{equation}\label{eqn:prox_set}
    \Sprox{t-1}(I, \cT) \coloneqq \cS^{\mathrm{past}}_{t-1} \cup \{I\} \cup \cT.
\end{equation}

\begin{proposition}
\label{prop:reveal-swap}
    Fix $t\in\{1,\ldots,n_1\}$ and let $I=\Pi(t)$.
    Conditional on $(\cF_{t-1},I)$, draw
    \begin{equation}\label{eqn:JT_dist}
        J\sim\Unif\big(\cR_{t-1}\setminus\{I\}\big),\quad
        \cT\sim\Unif\binom{\cR_{t-1}\setminus\{I,J\}}{\nass-1}.
    \end{equation}
    Then the martingale increment 
    \begin{align}
        D_t &=\bbE[f(\Sone)\mid\cF_t]-\bbE[f(\Sone)\mid\cF_{t-1}]   \notag\\
            &= - \alpha_t\cdot \bbE\Big[\,  \Delta_{IJ}\,f\big(\Sprox{t-1}(I,\cT)\big)  \bigm|  \cF_{t-1},\,I \,\Big],  \label{eq:rev-swap-avg}
    \end{align}
    where $\alpha_t = \frac{\rem-\nass}{\rem}=\frac{n_0}{n-t+1}\in(0,1)$. 
\end{proposition}

In \eqref{eq:rev-swap-avg}, the expectation is taken with respect to $(J, \cT)$. 
See Appendix \ref{sec:proof_lemma_swap} for the proof of Proposition \ref{prop:reveal-swap}. 

\medskip

\begin{remark}
    At step $t$, conditioned on the past ($\cS^{\mathrm{past}}_{t-1}$) and the current candidate ($I=\Pi(t)$), $D_t$ quantifies the average effect of the swap that “replace a uniformly drawn random competitor $J$ by $I$’’ inside a uniformly drawn random feasible completion set $\cT$.
\end{remark}

\paragraph{Concentration via Freedman's Inequality.}
For each $t = 1, \dots, n_1$ and each $i\in\cR_{t-1}$, define a $\cF_{t-1}$‑measurable function
\begin{equation}\label{eqn:zeta_t}
    \zeta_t(i) \coloneqq \bbE \left[\Delta_{iJ}\,f \left(\cS^{\mathrm{prox}}_{t-1}(i,\cT)\right)\,\Bigm|\,\cF_{t-1}\right],
\end{equation}
where the expectation is over $J$ and $\cT$ as in \eqref{eq:rev-swap-avg}. 
By Proposition~\ref{prop:reveal-swap}, $D_t=-\,\alpha_t\,\zeta_t(I)$ with $I\mid\cF_{t-1}\sim\Unif(\cR_{t-1})$.
Define (oracle) parameters using \eqref{eq:rev-swap-avg}:
\begin{align}
    v_t^\star &\coloneqq \Var(D_t\mid\cF_{t-1}) 
        = \alpha_t^2\,\Var \bigl(\zeta_t(I)\bigr),   \label{eq:var_star}\\
    r_t^\star &\coloneqq \sup|D_t|\ =\ \alpha_t\,\max_{i\in\cR_{t-1}}|\zeta_t(i)|.  \label{eq:range_star}
\end{align}
By construction, $\zeta_t(\cdot)$ and $\alpha_t$ are $\cF_{t-1}$-measurable, and thus, $v_t^\star$ and $r_t^\star$ are $\cF_{t-1}$‑measurable (predictable).

Applying Freedman’s martingale concentration inequality \citep[Thm.~1.6]{freedman1975tail} to the sum $\sum_{t=1}^{n_1}D_t$ yields the following probabilistic tail bound. 
\begin{proposition}\label{prop:concentration}
    Let $(M_t, \cF_t)$ be the assignment-exposure martingale in \eqref{eqn:filtration}, \eqref{eqn:martingale}. 
    Then for any $\varepsilon \geq 0$,
    \begin{equation}\label{eq:freedman_conc}
        \Pr \bigl( f(\Sone) - \bbE f(\Sone) \geq \varepsilon \bigr) \leq \exp \left( - \frac{ \varepsilon^2 / 2}{ \Vstar + \Rstar \varepsilon / 3 } \right),
    \end{equation}
    where $\Vstar \coloneqq \sum_{t=1}^{n_1} v_t^\star$ and $\Rstar \coloneqq \max_{1\leq t\leq n_1} r_t^\star$ with $(v_t^\star, r_t^\star)$ as in \eqref{eq:var_star}, \eqref{eq:range_star}.
\end{proposition}
The bound \eqref{eq:freedman_conc} is self‑normalized and uses predictable $(v_t^\star,r_t^\star)$ along the exposure filtration; see Appendix~\ref{sec:proof_concentration} for a detailed derivation. 

\medskip

\begin{remark}\label{rem:freedman_2sided}
    Applying \eqref{eq:freedman_conc} to $\pm(f(\Sone)-\bbE f(\Sone))$ and taking a union bound yields that
    for any $\delta\in(0,1)$,
    \begin{equation}\label{eq:freedman_conc.2}
        \Pr\left(\,|f(\Sone)-\bbE f(\Sone)| \leq \sqrt{2 \Vstar\,\sfL_\delta} + \frac{\Rstar}{3}\,\sfL_\delta\,\right) \geq 1-\delta,
    \end{equation}
    where $\sfL_\delta \coloneqq \log\frac{2}{\delta}$. 
\end{remark}

\paragraph{Discussion.} 
Relative to permutation bounded differences, the bound in \eqref{eq:freedman_conc} is \emph{variance‑adaptive}. 
$\Vstar$ aggregates \emph{local} conditional variances of paired swaps along the reveal process, while $\Rstar$ is a \emph{local} range. 
These parameters are \emph{oracle} because they depend on the true per‑swap effects $\Delta_{ij}f(\cdot)$. 
Given full potential‑outcome information, they are computable via \eqref{eq:var_star}–\eqref{eq:range_star}; for OLS‑RA, Section~\ref{sec:main_OLS_RA} and Appendix~\ref{sec:additional_discussion} describe the corresponding covariate geometry that governs swap sensitivities. 
The construction is symmetric in treatment/control exposures (ending at $n_1$ or $n_0$); both yield valid bounds for the same $f$, and one may take the smaller radius in practice.

\subsection{Bias Control via Exchangeable Pairs}\label{sec:bias_exchangeable}
Continuing with a generic function $f:\binom{[n]}{m} \to \RR$, we derive an upper bound for the design bias $\bigl| \bbE f(\Sone) \bigr|$ via Stein’s method of exchangeable pairs; see \citet{ross2011fundamentals} for an introduction, and \citet{chen2010normal,reinert2009multivariate} for comprehensive treatments of Stein's method and the exchangeable‑pairs framework. 

Throughout this section, expectations $\bbE[ \cdot ]$ and inner products $\langle f, g \rangle \coloneqq \bbE[ f(\cS) g(\cS) ]$ are taken with respect to the uniform law $\pi$ on $\binom{[n]}{m}$, unless stated otherwise.

\paragraph{Exchangeable Pair on Johnson Graph.}
For $n \in \ZZnn$ and $m \in [n]$, the Johnson graph $J(n, m) = \bigl( \cV_{n,m}, \cE_{n,m} \bigr)$ is an undirected graph with the vertex set $\cV_{n,m} = \binom{[n]}{m} = \bigl\{ \cS \subseteq [n]: |\cS| = m  \bigr\}$ and the edge set $\cE_{n,m} = \bigl\{ \{ v_1, v_2\} \in \cV_{n,m} \times \cV_{n,m}: |v_1 \cap v_2| = m-1  \bigr\}$; edges connect sets that differ by a single swap. 
For $\cS \in \cV_{n,m}$, let $\cN(\cS)$ denote the set of $m (n-m)$ neighbors of $\cS$. 
Let $P$ denote the Markov kernel (transition matrix) on $J(n,m)$ that jumps from $\cS$ to a neighbor $\cS'\in\cN(\cS)$ uniformly: $P(\cS,\cS')=\tfrac{1}{m(n-m)}$ if $\cS'\in\cN(\cS)$ and $0$ otherwise. 
The uniform distribution $\pi$ is stationary for $P$, and $P$ is reversible, i.e., $\pi(\cS) P(\cS, \cS') = \pi(\cS') P(\cS', \cS)$.

Let $\cS \sim\Unif\binom{[n]}{m}$, and let $\cS'$ be obtained by $\cS'\mid\cS\sim P(\cS,\cdot)$. 
Then $(\cS,\cS')$ is exchangeable.
Given a function $f:\binom{[n]}{m}\to\RR$, 
since $\cS'|\cS \sim \Unif \bigl( \cN(\cS) \bigr)$, 
\begin{equation}\label{eq:stein-op}
\begin{aligned}
    \bbE[ f(\cS') - f(\cS) \mid \cS] 
        &= \bbE[ f(\cS')\mid \cS] - f(\cS)\\
        &= (P f) (\cS) - f(\cS).
\end{aligned}
\end{equation} 
Below, we use \eqref{eq:stein-op} to derive a bias upper bound.

\paragraph{Markov Chain and Carr\'{e} du Champ.} 
Consider a Markov chain on $J(n,m)$ with the (discrete-time) generator $L = P - \Id$, so $(L f)(\cS) = (P f)(\cS) - f(\cS) = \bbE[f(\cS')-f(\cS)\mid\cS]$. 
The carr\'{e} du champ (Dirichlet form) is defined
\begin{equation}\label{eqn:carre_du_champ}
    \Gamma(f)(\cS) \coloneqq \frac{1}{2}\,\bbE\big[\,(f(\cS')-f(\cS))^2 \,\bigm|\, \cS\,\big].
\end{equation}
Note that $\Gamma(f)(\cS) \geq 0$ and $\bbE \Gamma(f) = 0$ if and only if $f$ is constant on $\binom{[n]}{m}$. 
By Jensen's inequality (equivalently, by Cauchy–Schwarz for reversible kernels), for all $\cS$,
\begin{equation}\label{eq:stein-reg}
\begin{aligned}
\big| (Lf)(\cS) \big| 
    &\leq \sqrt{ \bbE \big[(f(\cS')-f(\cS))^2 \mid \cS\big]} \\
    &= \sqrt{2\,\Gamma(f)(\cS)}.
\end{aligned}
\end{equation}
(See e.g., \citet[Chapter 1]{bakry2013analysis} for more background and details on $L$ and $\Gamma$.)

\paragraph{Stein Regression.}
Recall the left-hand side of \eqref{eq:stein-op}. 
The exchangeable‑pairs “regression with remainder” posits the conditional linearization 
\begin{equation}\label{eqn:stein_regression_with_rem}
    \bbE\big[f(\cS')-f(\cS)\mid \cS\big] = -\lambda\,f(\cS) + \psi(\cS)  
\end{equation}
or equivalently, $L f=-\lambda f + \psi$, where $\lambda \geq 0$ (the "drift") and the "remainder" $\psi$ are to be determined. 
Taking expectations and using $\bbE[f(\cS')-f(\cS)]=0$ (exchangeability) yields the Stein identity $\lambda\,\bbE f=\bbE \psi$.

Next, we choose the value of $\lambda$ to compute (or bound) $\bbE f$. 
If $\Var(f)=0$, then $f$ is constant and $|\bbE f|=|f(\cS)|$ for the realized $\cS$; this degenerate case is handled separately below. 
We henceforth assume $\Var(f)>0$.  
Let $g \coloneqq f-\bbE f$ and define
\begin{align}
    \lambda^\star 
        &\coloneqq \arg\min_{\lambda\ge 0}\ \bbE\big(L f+\lambda g\big)^2  
        = -\frac{\langle L f,\,g\rangle}{\langle g,\,g\rangle}, 
        \label{eqn:lambda_rayleigh}\\
    \psi_c^\star 
        &\coloneqq L f+\lambda^\star g,  
        \label{eqn:centered_stein_residual}\\
    \psi_u^\star
        &\coloneqq L f+\lambda^\star f.
        \notag
\end{align}
We use \(\psi_c^\star\) to choose \(\lambda\), and use \(\psi_u^\star\) to control \(\bbE f\). 

Observe that $\langle \psi_c^\star,g\rangle=0$. 
Moreover, since $\bbE Lf=0$,
\[
    \bbE\psi_u^\star
        = \bbE(Lf+\lambda^\star f)
        = \lambda^\star \bbE f.
\]

By reversibility and the Dirichlet identity, $\lambda^\star>0$; see \eqref{eq:lambda-star} below. 
Applying Cauchy–Schwarz inequality,  
\begin{equation}\label{eq:stein-global}
    |\bbE f|
        = \frac{|\bbE \psi_u^\star|}{\lambda^\star}
        \leq
        \frac{\sqrt{\bbE(\psi_u^\star)^2}}{\lambda^\star}.
\end{equation}

\paragraph{(Oracle) Identification of $\lambda^\star$ and Bias Bound.}
Reversibility implies the discrete Dirichlet identity
\[
    \bbE\,\Gamma(h) = -\langle Lh, h\rangle\qquad\text{for all }h,
\]
see, e.g., \citet[Ch.~1]{bakry2013analysis}. 
Since $L$ annihilates constants, $Lg = Lf$, and $\Var(f) = \langle g, g \rangle$. 
Moreover, $\bbE \Gamma(f) = - \langle Lf, f \rangle = - \langle Lg, f \rangle = - \langle Lg, f - \bbE f \rangle = - \langle Lg, g \rangle$. 
Using the Rayleigh‑quotient form in \eqref{eqn:lambda_rayleigh},
\begin{equation}\label{eq:lambda-star}
    \lambda^\star 
        = -\frac{\langle Lf,\,g\rangle}{\langle g,\,g\rangle}
        = -\frac{\langle Lg,\,g\rangle}{\langle g,\,g\rangle}
        = \frac{\bbE\Gamma(f)}{\Var(f)} 
        \stackrel{(*)}{>} 0.
\end{equation}
Here, (*) holds because $\Var(f) > 0$ implies $\bbE \Gamma(f) \neq 0$.

By instantiating $f(\Sone)=\hat\tau(\Sone)-\tau$, we obtain the following bias bound.

\begin{proposition}[Oracle Stein bias bound]\label{prop:stein-bias}
    Let $f:\binom{[n]}{m} \to\RR$ and let $(\cS,\cS')$ be the one‑swap exchangeable pair on $J(n,m)$. 
    If $\Var(f) > 0$, then, under complete randomization and the kernel $P$ above,
    \begin{equation}\label{eq:bias-prop}
        |\bbE f(\cS)|
            \leq
            \frac{\sqrt{\bbE\big(Lf+\lambda^\star f\big)^2}}{\lambda^\star},
            \qquad
            \lambda^\star=\frac{\bbE\Gamma(f)}{\Var(f)}.
    \end{equation}
\end{proposition}

The right‑hand side is an oracle quantity: it requires the full distributional knowledge about \(f\) over all treatment assignments. 
In particular, for \(f(\Sone)=\hat\tau(\Sone)-\tau\) instantiated as the deviation of an ATE estimator, Proposition~\ref{prop:stein-bias} yields an oracle bias bound for \(\hat\tau\). 
Although Proposition~\ref{prop:stein-bias} follows from the discussions in the main text so far, we provide a formal proof in Appendix~\ref{sec:proof_Stein_bias} for completeness.

\medskip

\begin{remark}
    By Poincar\'{e} (spectral‑gap) inequality, $\bbE\,\Gamma(h) \geq \mathrm{gap}_{n,m}\,\Var(h)$ for all $h$, where $\mathrm{gap}_{n,m}$ is the spectral gap of the random walk by $P$ on $J(n,m)$. 
    Consequently, for $\lambda^\star=\bbE\Gamma(f)/\Var(f)$ chosen above, we have $\lambda^\star\geq \mathrm{gap}_{n,m}$ \citep[Chapter 13]{levin2017markov}. 
    It is well known that $\mathrm{gap}_{n,m} = 1-\lambda_2(P) = \tfrac{1}{m}+\frac{1}{n-m} = \tfrac{n}{m(n-m)}$; 
    see, e.g., \citet{delsarte1973algebraic}.
\end{remark}

\section{MAIN RESULTS}\label{sec:main_results}
In this section we apply the machinery from Section~\ref{sec:finite_sample_analysis} to $f(\cS)=\hat\tau(\cS)-\tau$, where $\tau$ is the ATE and $\hat{\tau}$ an estimator. 
Section~\ref{sec:main-result-oracle} gives an \emph{oracle} finite-sample bound in terms of concentration and Stein parameters. Section~\ref{sec:main_OLS_RA} specializes these quantities to OLS--RA via a closed-form swap-sensitivity identity.

\subsection{Oracle Finite-sample Analysis}\label{sec:main-result-oracle}
Let $f(\cS)=\hat\tau(\cS)-\tau$. 
With $(M_t,\cF_t)$ as in \eqref{eqn:martingale}, recall from \eqref{eq:var_star}–\eqref{eq:range_star} that for $t=1,\dots,n_1$, 
\begin{equation}\label{eqn:oracle_conc_components}
    v_t^\star \coloneqq \Var(M_t - M_{t-1}\mid\cF_{t-1}), \qquad
    r_t^\star \coloneqq \sup |M_t - M_{t-1}|.
\end{equation}
Set the (oracle) concentration parameters
\begin{equation}\label{eqn:oracle_concentration_param}
    \Vstar \coloneqq \sum_{t=1}^{n_1} v_t^\star
    \qquad\text{and}\qquad
    \Rstar \coloneqq \max_{1\leq t\leq n_1} r_t^\star. 
\end{equation}
With the one-swap kernel $P$ on $J(n,n_1)$, 
\begin{equation}\label{eq:Gamma-expansion}
    \begin{aligned}
    \Gamma(f)(\cS)
        &=\frac{1}{2n_1 n_0}\sum_{i\in \cS}\sum_{j\notin \cS} \big(\Delta_{ij} f(\cS)\big)^2,
    \end{aligned}
\end{equation}
where $\cS'|\cS\sim P(\cS,\cdot)$ and $\Delta_{ij} f(\cS)\coloneqq f(\cS^{(i\leftrightarrow j)})-f(\cS)$, cf. \eqref{eqn:perturbation}.
Averaging \eqref{eq:Gamma-expansion} over $\cS\sim\Unif\binom{[n]}{n_1}$ gives the population quantity $\bbE\Gamma(f)$.

Following \eqref{eq:bias-prop}, define the oracle bias parameter
\begin{equation}\label{eqn:oracle_bias_param}
    \Bstar \coloneqq \frac{\sqrt{\bbE\big(Lf+\lambda^\star f\big)^2}}{\lambda^\star},
        \quad\text{where}\quad
    \lambda^\star=\frac{\bbE\Gamma(f)}{\Var(f)},
\end{equation}
with the convention that if $\Var(f)=0$ then $\Bstar\coloneqq|\bbE f|$ (since $\Gamma(f)\equiv 0$ in this case).

\begin{theorem}[Oracle finite-sample CI]\label{thm:oracle}
    Let $\cS_1 \sim \Unif \binom{[n]}{n_1}$ and $f(\cS)=\hat\tau(\cS)-\tau$. 
    For any $\delta\in(0,1)$,
    \begin{equation}\label{eqn:main_oracle}
        \Pr \left(\,| f(\cS_1) | \leq \sqrt{2\Vstar \, \sfL_{\delta}}+\frac{\Rstar}{3} \, \sfL_{\delta} + \Bstar\,\right) \geq 1-\delta,
    \end{equation}
    where $\sfL_{\delta} \coloneqq \log\frac{2}{\delta}$. $\Vstar, \Rstar, \Bstar$ are as in \eqref{eqn:oracle_concentration_param} and \eqref{eqn:oracle_bias_param}.
\end{theorem}

\begin{proof}[Proof sketch of Theorem \ref{thm:oracle}]
    Observe $| f(\cS) | \leq | f(\cS) - \bbE f(\cS) | + |\bbE f(\cS) |$ and combine Propositions \ref{prop:concentration} and \ref{prop:stein-bias}.
    See Appendix \ref{sec:proof_theorem_oracle} for a complete proof.
\end{proof}

\paragraph{Remarks on Bias Bound $\mathbf{\Bstar}$.} 
The quantities in \eqref{eqn:oracle_bias_param} are population-level oracle quantities computable from the full potential-outcome information over the randomization distribution, not from the instantiated data alone. 
While \(\bbE\Gamma(f)\) determines \(\lambda^\star\), the bias bound is governed by the uncentered Stein residual \(Lf+\lambda^\star f\). 
See Example~\ref{example:DiM_bias} in Appendix~\ref{sec:example_stein} for the DiM illustration. 
For DiM, \(f=\hat\tau_{\DIM}-\tau\) is an exact Johnson-graph eigenfunction, \(Lf=-\lambda_{\mathrm{lin}}f\) with \(\lambda_{\mathrm{lin}}=\tfrac{n}{n_1n_0}\), and \(\lambda^\star=\lambda_{\mathrm{lin}}\). 
Hence \(Lf+\lambda^\star f\equiv0\), so \(\Bstar=0\), capturing the design unbiasedness. 
For estimators beyond DiM, such as OLS--RA, \(\Bstar\) measures the size of the corresponding uncentered Stein residual, and provides an oracle bound on design bias.

\subsection{Swap Sensitivity Analysis for OLS-RA}\label{sec:main_OLS_RA}
Since \(D_t\) is a conditional average of one-swap sensitivities (Proposition~\ref{prop:reveal-swap}),
the parameters \(\Vstar,\Rstar\) in \eqref{eqn:oracle_concentration_param} reduce to population-level
swap-sensitivity analysis. For OLS--RA, the change \(\Delta_{ij}f(\cS)=f(\cS^{(i\leftrightarrow j)})-f(\cS)\)
admits a convenient algebraic form via paired deletion–insertion and rank-one pseudoinverse
updates.

\begin{definition}
    Let $X \in \RR^{n \times p}$ and $y^{(1)}, y^{(0)} \in \RR^n$.  
    For $\arm \in \{0,1\}$, $\cS \in \binom{[n]}{m}$, $i\in \cS$, and $j \notin \cS$, define the armwise atomic changes due to deletion/insertion as 
    \begin{equation}\label{eqn:perturbations}
    \begin{aligned}
        \Delta^{\mathrm{del}}_\arm (i; \cS) &\coloneqq \hat{\mu}_\arm \big(\cS\setminus\{i\}\big) - \hat{\mu}_\arm (\cS),\\
        \Delta^{\mathrm{ins}}_\arm (j; \cS) &\coloneqq \hat{\mu}_\arm \big(\cS\cup\{j\}\big) - \hat{\mu}_\arm (\cS),
    \end{aligned}
    \end{equation}
    where $\hat{\mu}_\arm(\cS) = \tfrac{ \bone_{m}^{\top} Q_{X_{\cS}}\, y^{(\arm)}_{\cS} }{ \bone_{m}^{\top} Q_{X_{\cS}}\, \bone_{m} }$, cf. \eqref{eq:mu-ratio}.
\end{definition}

For OLS--RA, the swap sensitivity of $f(\cS_1)=\htauols(\cS_1)-\tau$ equals that of $\htauols(\cS_1) = \hat{\mu}_1(\cS_1) - \hat{\mu}_0(\cS_0)$ because $\tau$ is constant. 
The one‑swap change admits an exact paired deletion-insertion decomposition. 
\begin{proposition}\label{prop:sensitivity_decomposition}
    For any $\Sone \in \binom{[n]}{n_1}$ and $(i,j) \in \Sone \!\times\! \Szero$,
    \begin{equation}\label{eq:exact-swap-coupled}
    \begin{aligned}
        \Delta_{ij} \htauols(\cS_1)
            &= \Bigl[ \Delta^{\mathrm{del}}_1 \bigl(i; \cS_1 \bigr) + \Delta^{\mathrm{ins}}_1 \bigl(j; \cS_1\setminus \{i\} \bigr) \Bigr] \\
                &\quad - \Bigl[ \Delta^{\mathrm{del}}_0 \bigl(j; \cS_0 \bigr) + \Delta^{\mathrm{ins}}_0 \bigl(i; \cS_0 \setminus \{j\} \bigr)\Bigr].
    \end{aligned}
    \end{equation}
\end{proposition}

\eqref{eq:exact-swap-coupled} is an algebraic telescoping identity, and no properties of the OLS estimator are needed beyond that $\hat\mu_{\arm}(\cdot)$ is well-defined. 
See Appendix \ref{sec:proof_proposition_decomposition} for a proof.

By Proposition~\ref{prop:mu_ratio}, the atomic deletion and insertion changes admit explicit, computable formulas through one-row updates of the corresponding arm-wise \(Q\)-matrices. 
To state the result, we introduce notation. 
For a positive semidefinite matrix $Q$, define the $Q$-inner product and $Q$-norm as follows: for vectors $u, v$,
\[
    \langle u,v\rangle_{Q} \coloneqq u^\top Q v,
    \qquad 
    \|u\|_{Q} \coloneqq \sqrt{u^\top Q u}.
\]
With $X \in \RR^{n \times p}$ and $y^{(1)}, y^{(0)} \in \RR^n$ fixed, we let $X_{\cS} \in \RR^{|\cS| \times p}$ and $y^{(\arm)}_{\cS} \in \RR^{|\cS|}$ denote the row-submatrix/sub-vector of $X$, $y^{(\arm)}$ with indices restricted to $\cS$. 
To avoid clutter, we write $\langle \cdot, \cdot \rangle_{\cS} \coloneqq \langle \cdot, \cdot \rangle_{Q_{X_{\cS}}}$ and $\| \cdot \|_{\cS} \coloneqq \| \cdot \|_{Q_{X_{\cS}}}$. 
With $s \coloneqq |\cS|$, denote by $e_i\in\RR^{s}$ the $i$-th standard basis vector indexed on $\cS$. 
Define 
\begin{align*}
    \Phidel_i(\cS) &\coloneqq \|\bone\|_{\cS}^2\,\|e_i\|_{\cS}^2-\langle e_i,\bone\rangle_{\cS}^2,\\
    \Phiins_j(\cS) &\coloneqq \|\bone\|_{\cS}^2\,\sigma_j^2(\cS)\ +\ \big(1-\alpha_j(\cS)\big)^2,
\end{align*}
where $\sigma_j^2(\cS) \coloneqq \|x_j\|_2^2-\| X_\cS x_j \|_{\cS}^2$ and $\alpha_j(\cS)\coloneqq \langle \bone, X_\cS x_j \rangle_{\cS}$.

\begin{theorem}[Atomic OLS updates: regular cases]\label{thm:perturbation_identity}
    Let $X \in \RR^{n \times p}$ and $y^{(1)}, y^{(0)} \in \RR^n$. 
    For any $\arm \in \{0,1\}$, $\cS \in \binom{[n]}{m}$, and $(i,j) \in \cS \times \cS^c$, the atomic changes in \eqref{eqn:perturbations} are exactly computable from the ratio representation \eqref{eq:mu-ratio} by one-row Moore--Penrose update formulas.

    \begin{enumerate}
        \item 
        In the regular deletion case, if \(\Phidel_i(\cS)>0\), then
        \[
            \Delta^{\mathrm{del}}_\arm(i;\cS)
            =
            -\frac{
                \langle e_i,\bone\rangle_{\cS}
                \langle e_i,\ycenta{\cS}\rangle_{\cS}
            }{
                \Phidel_i(\cS)
            },
        \]
        where \(\ycenta{\cS}\coloneqq y^{(\arm)}_\cS-\hat\mu_\arm(\cS)\bone \). 

        \item 
        In the regular \(K\)-branch insertion case, where both \(\cS\) and \(\cS\cup\{j\}\) are in the \(K\)-branch and \(\sigma_j^2(\cS)>0\), 
        \begin{equation}\label{eqn:swap_sensitivity_insertion_regular}
            \Delta^{\mathrm{ins}}_\arm(j;\!\cS)
            =
            \frac{1\!-\!\alpha_j(\cS)}{\Phiins_j(\cS)}
            \Big(
                y^{(\arm)}_j-\hat\mu_\arm(\cS)
                -\langle X_\cS x_j,\ycenta{\cS}\rangle_{\cS}
            \Big).
        \end{equation}
    \end{enumerate}
\end{theorem}

Appendix~\ref{sec:proof_theorem_perturbation} gives the full branch-specific identities and proof of Theorem \ref{thm:perturbation_identity}. 
Together with the paired telescope \eqref{eq:exact-swap-coupled}, these identities yield oracle one-swap effects \(\Delta_{ij}\hat\tau_{\OLS}(\cS)\). 
Consequently, \(\Gamma(f)\), \(Lf\), and the predictable increments \(D_t\)---and hence \((\Vstar,\Rstar,\Bstar)\)---are oracle-computable from \((X,y^{(0)},y^{(1)})\); see Appendix~\ref{sec:swap_OLS} for more details.

\paragraph{Discussion.} 
Theorem~\ref{thm:oracle} yields a design-based, finite-sample CI for $\tau$ centered at $\hat\tau$. 
The pathwise fluctuation terms $(\Vstar,\Rstar)$ are \emph{instance-adaptive}, through the realized assignment $\cS_1$ and reveal order via the predictable increments $D_t$, while $\Bstar$ is a population (oracle) term determined by the Stein residual in \eqref{eqn:oracle_bias_param}. 
Given $(X,y^{(0)},y^{(1)})$, the exact swap identity \eqref{eq:exact-swap-coupled} and Theorem~\ref{thm:perturbation_identity} permit oracle evaluation of $\Vstar,\Rstar$ and $\Bstar$.

For OLS--RA, Theorem~\ref{thm:perturbation_identity} shows that swap sensitivity is governed by the geometry of the arm-wise \(Q\)-matrices, cf. Appendix~\ref{sec:geometric_envelope}.
The symmetry between deletion and insertion becomes clear when insertion is viewed as reverse deletion. 
Deleting \(i\) from \(\cS\) is governed by the \(Q\)-leverage of \(i\) in the current set \(\cS\), whereas inserting \(j\) into \(\cS\) is governed by the \(Q\)-leverage of \(j\) in the augmented set \(\cS\cup\{j\}\). 
In regular cases, both atomic changes obey the same leverage--residual template:
\[
    |\Delta|
    =
    \text{standardized residual}
    \times
    \frac{\sqrt{\text{leverage}}}{1-\text{leverage}}.
\]

For deletion,
\[
    \big|\Delta^{\mathrm{del}}_\arm(i;\cS)\big|
    =
    R^{\mathrm{del}}_{\arm,i}(\cS)\,
    \frac{\sqrt{\ell_i(\cS)}}{1-\ell_i(\cS)},
\]
where
\[
    \ell_i(\cS)
    \coloneqq
    \frac{\langle e_i,\bone\rangle_{Q_\cS}^2}
    {\|e_i\|_{Q_\cS}^2\|\bone\|_{Q_\cS}^2},
    \qquad
    R^{\mathrm{del}}_{\arm,i}(\cS)
    \coloneqq
    \frac{|\langle e_i,\ycenta{\cS}\rangle_{\cS}|}
    {\|e_i\|_{\cS}\|\bone\|_{\cS}}.
\]

For insertion, let \(\ell_j^+(\cS)\coloneqq \ell_j(\cS\cup\{j\})\) denote the leverage of the incoming unit in the augmented set, and let \(R^+_{\arm,j}(\cS) \) be the residual defined similarly; see \eqref{eqn:residual_augmented} in Appendix~\ref{sec:geometric_envelope}. 
Whenever the corresponding regular deletion identity applies in \(\cS\cup\{j\}\),
\[
    \big|\Delta^{\mathrm{ins}}_\arm(j;\cS)\big|
    =
    R^+_{\arm,j}(\cS)\,
    \frac{\sqrt{\ell_j^+(\cS)}}{1-\ell_j^+(\cS)}.
\]

Thus, moderate deletion/insertion leverage and small standardized residuals imply reduced per-swap effects. 
Compared to DiM, whose swap sensitivity scales with raw potential-outcome contrasts (cf. Example~\ref{example:DiM_sensitivity}), OLS--RA can benefit when covariates reduce residuals without inflating deletion or insertion \(Q\)-leverages.

\medskip

\begin{remark}[Data-driven proxy]
    $\Delta^{\mathrm{ins}}_\arm(j;\cS)$ depends on the incoming potential outcome \(y_j^{(\arm)}\), which is counterfactual when \(j\) lies in the opposite arm. 
    Appendix~\ref{sec:geometric_envelope} isolates the observable geometric factors. 
\end{remark}

\section{EXPERIMENTS}\label{sec:experiments}

Although this paper is primarily theoretical, Appendix~\ref{sec:deferred_experiments} reports oracle-setting Monte Carlo experiments that quantify the finite-sample behavior of the confidence intervals in Theorem~\ref{thm:oracle}. 
For OLS--RA, the oracle quantities \((V^\star,R^\star,B^\star)\) are approximated by Monte Carlo estimators \((\widehat V^\star,\widehat R^\star,\widehat B^\star)\); for DiM, \(V^\star\) and \(R^\star\) are available in closed form under the remaining-pool reveal law.
Details about the MC procedure, results and discussion are deferred to Appendix~\ref{sec:deferred_experiments}.

Overall, the proposed finite-sample (FS) CIs are valid but conservative. 
For DiM, the exact remaining-pool oracle quantities yield over-coverage at small \(n\), so Experiment~1 should be read primarily as a validity and conservatism sanity check for the general design-based finite-sample bound. 
For OLS--RA, the direct MC estimates \((\widehat V^\star,\widehat R^\star,\widehat B^\star)\) produce empirical coverage essentially equal to \(1.000\) across the reported cells, while the FS widths remain about \(2.2\)–\(2.6\) times the empirical inter-percentile range (IPR), which itself peaks near the arm-wise interpolation threshold and then decreases as \(p\) grows, consistent with the \(Q\)-geometric sensitivity picture in Section~\ref{sec:main_OLS_RA}. 
Diagnostics suggest nontrivial slack in the current FS bound, motivating future work on tighter assignment-adaptive range and bias proxies that preserve finite-sample validity. 
\section{CONCLUSION}\label{sec:conclusion}

We developed a design-based, non-asymptotic framework for analyzing regression adjustment under complete randomization, yielding finite-sample-valid confidence intervals for the ATE. 
Our analysis controls stochastic fluctuation through a variance-adaptive Doob martingale and Freedman’s inequality, bounds design bias via Stein’s method of exchangeable pairs, and yields an oracle guarantee (Theorem~\ref{thm:oracle}) that applies in both classical and high-dimensional/interpolating regimes. 
For OLS--RA, paired deletion--insertion identities and rank-one pseudoinverse updates reveal how deletion leverage, insertion leverage, and standardized residual energy govern swap sensitivity, clarifying when RA can improve on the DiM baseline and how covariate \(Q\)-geometry shapes its precision gains.

Several limitations motivate future work. 
First, the current guarantees are oracle: computing \((\Vstar,\Rstar,\Bstar)\) requires full potential outcomes. 
Developing data-driven surrogates based only on \((X,T,Y)\), with finite-sample guarantees, is an important next step. 
Second, the current Freedman-type concentration and bias bounds can be conservative; sharper geometry-aware proxies, or alternatives to Freedman-type bounds, may tighten the intervals while retaining validity. 
Third, extending the martingale and exchangeable-pairs constructions beyond complete randomization---for example to stratified, clustered, rerandomized, or covariate-adaptive designs---would broaden applicability.



\clearpage
\bibliographystyle{plainnat}
\bibliography{bibliography}

@article{lei2021regression,
  title={Regression adjustment in completely randomized experiments with a diverging number of covariates},
  author={Lei, Lihua and Ding, Peng},
  journal={Biometrika},
  volume={108},
  number={4},
  pages={815--828},
  year={2021},
  publisher={Oxford University Press}
}

@article{lin_ols2013,
author = {Winston Lin},
title = {{Agnostic notes on regression adjustments to experimental data: Reexamining Freedman’s critique}},
volume = {7},
journal = {The Annals of Applied Statistics},
number = {1},
publisher = {Institute of Mathematical Statistics},
pages = {295 -- 318},
year = {2013},
doi = {10.1214/12-AOAS583},
URL = {https://doi.org/10.1214/12-AOAS583}
}

@article{freedman1975tail,
  title={On tail probabilities for martingales},
  author={Freedman, David A},
  journal={the Annals of Probability},
  pages={100--118},
  year={1975},
  publisher={JSTOR}
}

@article{freedman2008regression,
  title={On regression adjustments to experimental data},
  author={Freedman, David A},
  journal={Advances in Applied Mathematics},
  volume={40},
  number={2},
  pages={180--193},
  year={2008},
  publisher={Elsevier}
}

@article{hajek1960limiting,
  title={Limiting distributions in simple random sampling from a finite population},
  author={H{\'a}jek, Jaroslav},
  journal={A Magyar Tudom{\'a}nyos Akad{\'e}mia Matematikai Kutat{\'o} Int{\'e}zet{\'e}nek k{\"o}zlemenyei},
  volume={5},
  number={3},
  pages={361--374},
  year={1960},
  publisher={Akad{\'e}miai Kiad{\'o}}
}

@book{ben2003generalized,
  title={Generalized Inverses: Theory and Applications},
  author={Ben-Israel, Adi and Greville, Thomas NE},
  year={2003},
  publisher={Springer}
}

@article{greville1960some,
  title={Some applications of the pseudoinverse of a matrix},
  author={Greville, Thomas NE},
  journal={SIAM Review},
  volume={2},
  number={1},
  pages={15--22},
  year={1960},
  publisher={SIAM}
}

@book{bakry2013analysis,
  title={Analysis and geometry of Markov diffusion operators},
  author={Bakry, Dominique and Gentil, Ivan and Ledoux, Michel},
  volume={348},
  year={2013},
  publisher={Springer Science \& Business Media}
}

@book{levin2017markov,
  title={Markov Chains and Mixing Times},
  author={Levin, David A and Peres, Yuval},
  volume={107},
  year={2017},
  publisher={American Mathematical Society}
}

@article{delsarte1973algebraic,
  title={An algebraic approach to the association schemes of coding theory},
  author={Delsarte, Philippe},
  journal={Philips Res. Rep. Suppl.},
  volume={10},
  pages={vi+--97},
  year={1973}
}

@article{ross2011fundamentals,
  title={Fundamentals of {S}tein’s method},
  author={Ross, Nathan},
  journal={Probability Surveys},
  volume={88},
  pages={210--293},
  year={2011}
}

@article{reinert2009multivariate,
  title={Multivariate normal approximation with Stein’s method of exchangeable pairs under a general linearity condition},
  author={Reinert, Gesine and R{\"o}llin, Adrian},
  journal={Annals of Probability},
  volume={37},
  number={6},
  pages={2150--2173},
  year={2009}
}

@book{chen2010normal,
  title={Normal Approximation by Stein’s Method},
  author={Chen, Louis HY and Goldstein, Larry and Shao, Qi-Man},
  year={2010},
  publisher={Springer Science \& Business Media}
}

@book{fisher1935design,
  title={The Design of Experiments},
  author={Fisher, R.A.},
  lccn={36004973},
  series={The Design of Experiments},
  url={https://books.google.com/books?id=-EsNAQAAIAAJ},
  year={1935},
  publisher={Oliver and Boyd}
}

@book{kempthorne1952design,
  author = {Kempthorne, Oscar},
  title = {The Design and Analysis of Experiments},
  publisher = {John Wiley \& Sons},
  year = {1952},
  address = {New York},
  series = {A Wiley Publication in Mathematical Statistics}
}

@book{imbens2015causal,
  title={Causal Inference in Statistics, Social, and Biomedical Sciences},
  author={Imbens, Guido W and Rubin, Donald B},
  year={2015},
  publisher={Cambridge University Press}
}

@book{rosenberger2015randomization,
  title={Randomization in Clinical Trials: Theory and Practice},
  author={Rosenberger, William F and Lachin, John M},
  year={2015},
  publisher={John Wiley \& Sons}
}

@article{neyman,
  author       = {Jerzy Neyman},
  title        = {Sur les applications de la théorie des probabilités aux expériences agricoles: Essai des principes},
  journal      = {Roczniki Nauk Rolniczych},
  volume       = {10},
  pages        = {1--51},
  year         = {1923},
  note         = {Excerpts translated and reprinted in English in \emph{Statistical Science}, Vol. 5, No. 4 (1990), pp. 465--472, by D. M. Dabrowska and T. P. Speed},
  language     = {French}
}

@book{hinkelmann2007design,
  title={Design and Analysis of Experiments, Volume 1: Introduction to Experimental Design},
  author={Hinkelmann, Klaus and Kempthorne, Oscar},
  volume={1},
  year={2007},
  publisher={John Wiley \& Sons}
}

@article{chang2024exact,
  title={Exact bias correction for linear adjustment of randomized controlled trials},
  author={Chang, Haoge and Middleton, Joel A and Aronow, PM},
  journal={Econometrica},
  volume={92},
  number={5},
  pages={1503--1519},
  year={2024},
  publisher={Wiley Online Library}
}

@article{chiang2023regression,
  title={Regression adjustment in randomized controlled trials with many covariates},
  author={Chiang, Harold D and Matsushita, Yukitoshi and Otsu, Taisuke},
  journal={arXiv preprint arXiv:2302.00469},
  year={2023}
}

@article{bloniarz2016lasso,
  title={Lasso adjustments of treatment effect estimates in randomized experiments},
  author={Bloniarz, Adam and Liu, Hanzhong and Zhang, Cun-Hui and Sekhon, Jasjeet S and Yu, Bin},
  journal={Proceedings of the National Academy of Sciences},
  volume={113},
  number={27},
  pages={7383--7390},
  year={2016},
  publisher={National Academy of Sciences}
}

@article{wager2016high,
  title={High-dimensional regression adjustments in randomized experiments},
  author={Wager, Stefan and Du, Wenfei and Taylor, Jonathan and Tibshirani, Robert J},
  journal={Proceedings of the National Academy of Sciences},
  volume={113},
  number={45},
  pages={12673--12678},
  year={2016},
  publisher={National Academy of Sciences}
}

@book{ding2024first,
  title={A First Course in Causal Inference},
  author={Ding, Peng},
  year={2024},
  publisher={Chapman and Hall/CRC}
}

@article{tsiatis2008covariate,
  title={Covariate adjustment for two-sample treatment comparisons in randomized clinical trials: a principled yet flexible approach},
  author={Tsiatis, Anastasios A and Davidian, Marie and Zhang, Min and Lu, Xiaomin},
  journal={Statistics in Medicine},
  volume={27},
  number={23},
  pages={4658--4677},
  year={2008},
  publisher={Wiley Online Library}
}

@article{cattaneo2018inference,
  title={Inference in linear regression models with many covariates and heteroscedasticity},
  author={Cattaneo, Matias D and Jansson, Michael and Newey, Whitney K},
  journal={Journal of the American Statistical Association},
  volume={113},
  number={523},
  pages={1350--1361},
  year={2018},
  publisher={Taylor \& Francis}
}

@article{aronow2014sharp,
  title={Sharp bounds on the variance in randomized experiments},
  author={Aronow, Peter M and Green, Donald P and Lee, Donald KK},
  journal={Annals of Statistics},
  volume={42},
  number={3},
  pages={850--871},
  year={2014}
}

@article{dasgupta2015causal,
  title={Causal inference from 2K factorial designs by using potential outcomes},
  author={Dasgupta, Tirthankar and Pillai, Natesh S and Rubin, Donald B},
  journal={Journal of the Royal Statistical Society Series B: Statistical Methodology},
  volume={77},
  number={4},
  pages={727--753},
  year={2015},
  publisher={Oxford University Press}
}

@article{fogarty2018regression,
  title={Regression-assisted inference for the average treatment effect in paired experiments},
  author={Fogarty, Colin B},
  journal={Biometrika},
  volume={105},
  number={4},
  pages={994--1000},
  year={2018},
  publisher={Oxford University Press}
}

@article{mukerjee2018using,
  title={Using standard tools from finite population sampling to improve causal inference for complex experiments},
  author={Mukerjee, Rahul and Dasgupta, Tirthankar and Rubin, Donald B},
  journal={Journal of the American Statistical Association},
  volume={113},
  number={522},
  pages={868--881},
  year={2018},
  publisher={Taylor \& Francis}
}

@article{li2020rerandomization,
  title={Rerandomization and regression adjustment},
  author={Li, Xinran and Ding, Peng},
  journal={Journal of the Royal Statistical Society Series B: Statistical Methodology},
  volume={82},
  number={1},
  pages={241--268},
  year={2020},
  publisher={Oxford University Press}
}

@article{chernozhukov2018double,
  author = {Chernozhukov, Victor and Chetverikov, Denis and Demirer, Mert and Duflo, Esther and Hansen, Christian and Newey, Whitney and Robins, James},
  title = {Double/debiased machine learning for treatment and structural parameters},
  journal = {The Econometrics Journal},
  volume = {21},
  number = {1},
  pages = {C1-C68},
  year = {2018},
  month = {01},
  publisher={Oxford University Press Oxford, UK}
}

@article{song2025neumann,
  title={Neumann-series corrections for regression adjustment in randomized experiments},
  author={Song, Dogyoon},
  journal={arXiv preprint arXiv:2511.08539},
  year={2025}
}

@article{shen2023algebraic,
  title={Algebraic and statistical properties of the ordinary least squares interpolator},
  author={Shen, Dennis and Song, Dogyoon and Ding, Peng and Sekhon, Jasjeet S},
  journal={arXiv preprint arXiv:2309.15769},
  year={2023}
}

\clearpage
\section*{Checklist}

\begin{enumerate}

  \item For all models and algorithms presented, check if you include:
  \begin{enumerate}
    \item A clear description of the mathematical setting, assumptions, algorithm, and/or model. [\textbf{Yes}]
    \item An analysis of the properties and complexity (time, space, sample size) of any algorithm. [\textbf{Yes}]
    \item (Optional) Anonymized source code, with specification of all dependencies, including external libraries. [\textbf{Not Applicable}]
  \end{enumerate}

  \item For any theoretical claim, check if you include:
  \begin{enumerate}
    \item Statements of the full set of assumptions of all theoretical results. [\textbf{Yes}]
    \item Complete proofs of all theoretical results. [\textbf{Yes}]
    \item Clear explanations of any assumptions. [\textbf{Yes}]     
  \end{enumerate}

  \item For all figures and tables that present empirical results, check if you include:
  \begin{enumerate}
    \item The code, data, and instructions needed to reproduce the main experimental results (either in the supplemental material or as a URL). [\textbf{Yes}]
    \item All the training details (e.g., data splits, hyperparameters, how they were chosen). [\textbf{Not Applicable}]
    \item A clear definition of the specific measure or statistics and error bars (e.g., with respect to the random seed after running experiments multiple times). [\textbf{Yes}]
    \item A description of the computing infrastructure used. (e.g., type of GPUs, internal cluster, or cloud provider). [\textbf{Yes}]
  \end{enumerate}

  \item If you are using existing assets (e.g., code, data, models) or curating/releasing new assets, check if you include:
  \begin{enumerate}
    \item Citations of the creator If your work uses existing assets. [\textbf{Not Applicable}]
    \item The license information of the assets, if applicable. [\textbf{Not Applicable}]
    \item New assets either in the supplemental material or as a URL, if applicable. [\textbf{Not Applicable}]
    \item Information about consent from data providers/curators. [\textbf{Not Applicable}]
    \item Discussion of sensible content if applicable, e.g., personally identifiable information or offensive content. [\textbf{Not Applicable}]
  \end{enumerate}

  \item If you used crowdsourcing or conducted research with human subjects, check if you include:
  \begin{enumerate}
    \item The full text of instructions given to participants and screenshots. [\textbf{Not Applicable}]
    \item Descriptions of potential participant risks, with links to Institutional Review Board (IRB) approvals if applicable. [\textbf{Not Applicable}]
    \item The estimated hourly wage paid to participants and the total amount spent on participant compensation. [\textbf{Not Applicable}]
  \end{enumerate}

\end{enumerate}


\clearpage
\appendix
\counterwithin{theorem}{section}
\thispagestyle{empty}

\onecolumn
\aistatstitle{
    \textsc{
    Supplementary Material for\\
    ``Design-Based Finite-Sample Analysis for\\
    Regression Adjustment''
    }
}

The supplementary material is organized as follows. 
Section~\ref{sec:additional_discussion} offers additional discussion, including an oracle Stein bias example for DiM (Section~\ref{sec:example_stein}), further details of the swap–sensitivity analysis for OLS–RA (Section~\ref{sec:swap_OLS}), and geometry‑based envelopes computable without counterfactuals (Sections~\ref{sec:geometric_envelope}--\ref{sec:leverage_bounds}).
Section~\ref{sec:proof_section3} contains deferred proofs for the results in Section~\ref{sec:finite_sample_analysis}.
Section~\ref{sec:proof_section4} presents the proofs of the main results from Section~\ref{sec:main_results}.
Finally, Section~\ref{sec:deferred_experiments} reports numerical experiments that complement the theory in the main text.

\section{ADDITIONAL DISCUSSION}\label{sec:additional_discussion}

\subsection{Example: Oracle Bias Bound for DiM}\label{sec:example_stein}

\begin{example}\label{example:DiM_bias}
    As a sanity check, we compute the oracle bias parameter \(\Bstar\) in \eqref{eqn:oracle_bias_param} for $\hat\tau_{\DIM}$. 

    Define
    \[
        a_i \coloneqq \frac{y_i^{(1)}}{n_1}+\frac{y_i^{(0)}}{n_0},
        \qquad
        \bar a \coloneqq \frac{1}{n}\sum_{i=1}^n a_i .
    \]
    Since
    \[
        \hat\tau_{\DIM}(\cS)-\tau
        =
        \sum_{i\in\cS}(a_i-\bar a),
    \]
    the DiM error \(f(\cS)=\hat\tau_{\DIM}(\cS)-\tau\) is a linear function on the Johnson graph. 
    Hence, averaging one-swap changes gives
    \begin{align}
        (Lf)(\cS)
        &= \frac{1}{n_1n_0}\sum_{i\in\cS}\sum_{j\notin\cS}
            \bigl\{f(\cS^{(i\leftrightarrow j)})-f(\cS)\bigr\} \notag\\
        &= \frac{1}{n_1n_0}\sum_{i\in\cS}\sum_{j\notin\cS}(a_j-a_i) \notag\\
        &= -\,\frac{n}{n_1n_0}\,f(\cS).
        \label{eq:dim-exact-eigen}
    \end{align}
    Thus \(f\) is an exact eigenfunction of the Johnson-walk generator with
    \[
        \lambda_{\mathrm{lin}}=\frac{n}{n_1n_0}.
    \]
    Moreover,
    \begin{equation}\label{eqn:variance_dim_example}
        \Var(f)
        =
        \frac{n_1n_0}{n-1}
        \left(
            \frac{\Var_1}{n_1^2}
            +
            \frac{\Var_0}{n_0^2}
            +
            \frac{2\,\Cov_{\mathrm{diff}}}{n_1n_0}
        \right),
    \end{equation}
    where
    \[
        \Var_\alpha \coloneqq \frac{1}{n}\sum_{i=1}^n \bigl(y^{(\alpha)}_i-\bar y^{(\alpha)}\bigr)^2\quad(\alpha\in\{0,1\}),
        \qquad
        \Cov_{\mathrm{diff}} \coloneqq \frac{1}{n}\sum_{i=1}^n \bigl(y^{(1)}_i-\bar y^{(1)}\bigr)\bigl(y^{(0)}_i-\bar y^{(0)}\bigr).
    \]

    Next, recall that the carr\'{e} du champ for Johnson walk (simple random walk generated by uniform transition kernel on the Johnson graph), cf. \eqref{eq:Gamma-expansion}, is 
    \[
        \Gamma(f)(\cS) = \frac{1}{2\,n_1n_0}\sum_{i\in \cS}\ \sum_{j\notin \cS}\bigl(\Delta_{ij} f(\cS)\bigr)^2.
    \]
    For $f=\hat\tau_{\DIM}-\tau$, observe (cf. Example \ref{example:DiM_sensitivity}) that
    \begin{equation}\label{eq:dim_sensitivity}
        \Delta_{ij}\hat\tau_{\DIM} = \frac{1}{n_1}\bigl(y^{(1)}_{j}-y^{(1)}_{i}\bigr) + \frac{1}{n_0}\bigl(y^{(0)}_{j}-y^{(0)}_{i}\bigr).
    \end{equation}
    Averaging $\Gamma(f)(\cS)$ over $\cS\sim\Unif\binom{[n]}{n_1}$ (complete randomization) equals averaging over all ordered pairs $(i,j)$ with $i\neq j$, i.e., for any function $g: [n]\times[n] \to \RR$, 
    \begin{equation}\label{eq:ordered-pair-average}
        \bbE_S \left[\frac{1}{n_1n_0}\sum_{i\in S}\sum_{j\notin S} g(i,j)\right]
            = \frac{1}{n(n-1)}\sum_{i\neq j} g(i,j).
    \end{equation}
    Lastly, observe the finite-population identities (straightforward expansions):
    \begin{align}
        \frac{1}{n(n-1)}\sum_{i\neq j} \bigl(a_j-a_i\bigr)^2
            &= \frac{2n}{n-1}\,\Var_n(a), \label{eq:fp-var}\\
        \frac{1}{n(n-1)}\sum_{i\neq j} \bigl(a_j-a_i\bigr)\bigl(b_j-b_i\bigr)
            &= \frac{2n}{n-1}\,\Cov_n(a,b), \label{eq:fp-cov}
    \end{align}
    where $\Var_n(a)=\frac{1}{n}\sum_{i=1}^n(a_i-\bar a)^2$ and $\Cov_n(a,b)=\frac{1}{n}\sum_{i=1}^n(a_i-\bar a)(b_i-\bar b)$.
    
    Therefore, 
    \begin{align}
        \bbE_{\cS} [ \Gamma( \hat\tau_{\DIM} - \tau )(\cS) ]
            &= \bbE_{\cS} \left[ \frac{1}{2\,n_1n_0}\sum_{i\in \cS}\ \sum_{j\notin \cS}\bigl(\Delta_{ij} \hat\tau_{\DIM} (\cS)\bigr)^2 \right] 
                \qquad\because \Delta_{ij} (\htaudim - \tau) = \Delta_{ij} \htaudim ~\&~\eqref{eq:dim_sensitivity}  \notag\\
            &= \frac{1}{2 n(n-1)} \sum_{i, j \in [n]: i \neq j}  \bigl(\Delta_{ij} \hat\tau_{\DIM} (\cS)\bigr)^2
                \qquad\quad \because\eqref{eq:ordered-pair-average} \notag\\
            &= \frac{1}{2}\cdot \frac{1}{n(n-1)}
                \sum_{i\neq j} \left[\frac{\bigl(y_j^{(1)}-y_i^{(1)}\bigr)^2}{n_1^2}
                + \frac{\bigl(y_j^{(0)}-y_i^{(0)}\bigr)^2}{n_0^2}
                + \frac{2\bigl(y_j^{(1)}-y_i^{(1)}\bigr)\bigl(y_j^{(0)}-y_i^{(0)}\bigr)}{n_1n_0}\right] \notag\\
            &= \frac{n}{n-1}\left(\frac{\Var_1}{n_1^2}+\frac{\Var_0}{n_0^2}+\frac{2\,\Cov_{\mathrm{diff}}}{n_1 n_0}\right).
                \label{eq:Gamma-DiM}
    \end{align}

    Combining \eqref{eqn:variance_dim_example} with \eqref{eq:Gamma-DiM} yields
    \[
        \lambda^\star
        =
        \frac{\bbE\Gamma(f)}{\Var(f)}
        =
        \frac{n}{n_1n_0}
        =
        \lambda_{\mathrm{lin}}.
    \]
    Therefore, under the corrected definition \eqref{eqn:oracle_bias_param},
    \[
        \Bstar
        =
        \frac{\sqrt{\bbE\big(Lf+\lambda^\star f\big)^2}}{\lambda^\star}
        =
        0.
    \]
    Thus, for DiM estimator, the Stein bias parameter \(\Bstar\) captures its design unbiasedness. 
\end{example}

\subsection{Swap Sensitivity Analysis of OLS-RA}\label{sec:swap_OLS}

Here we assemble the exact swap sensitivity for OLS--RA from the atomic deletion/insertion identities (Theorem \ref{thm:perturbation_identity}) using the paired deletion--insertion decomposition (Proposition \ref{prop:sensitivity_decomposition}). 
Throughout, $\cS_1$ is the realized treatment set, $\cS_0=[n]\setminus\cS_1$, and we consider a paired swap $(i\leftrightarrow j)$ with $i\in\cS_1$, $j\in\cS_0$.

\paragraph{Recollecting Notation.}
For a set $\cS$, write $Q_{\cS}\equiv Q_{X_{\cS}}$ and define the $Q_{\cS}$–inner product and norm by $\langle u,v\rangle_{\cS} \coloneqq u^\top Q_{\cS} v$ and $\|u\|_{\cS}\coloneqq\sqrt{u^\top Q_{\cS} u}$.
Let $\bone$ be the all–ones vector of length $|\cS|$. For arm $\arm\in\{0,1\}$,
\[
    \hat\mu_\arm(\cS)=\frac{\langle \bone, y^{(\arm)}_{\cS}\rangle_{\cS}}{\|\bone\|_{\cS}^2},
    \qquad\text{and}\qquad
    \ycenta{\cS}\coloneqq y^{(\arm)}_{\cS}-\hat\mu_\arm(\cS)\bone.
\]
For $j\notin\cS$, define $\sigma_j^2(\cS)\coloneqq \|x_j\|_2^2-\|X_{\cS}x_j\|_{\cS}^2$ and $\alpha_j(\cS)\coloneqq\langle \bone,X_{\cS}x_j\rangle_{\cS}$. 
We define
\begin{align*}
    \Phidel_i(\cS) &\coloneqq \|\bone\|_{\cS}^2\,\|e_i\|_{\cS}^2-\langle e_i,\bone\rangle_{\cS}^2,
    \qquad\text{and}\qquad
    \Phiins_j(\cS) \coloneqq \|\bone\|_{\cS}^2\,\sigma_j^2(\cS)\ +\ \big(1-\alpha_j(\cS)\big)^2,
\end{align*}

\paragraph{Paired Deletion–Insertion Decomposition.}
For $i\in\cS_1$, $j\in\cS_0$, let $\cS_1^-=\cS_1\setminus\{i\}$ and $\cS_0^-=\cS_0\setminus\{j\}$. 
Then the paired swap $(i\leftrightarrow j)$ can be decomposed into \emph{two} row operations in each arm:
\begin{align*}
    &\cS_1\ \xrightarrow{\ \mathrm{delete}\ i\ }\ \cS_1^-:=\cS_1\setminus\{i\}
    \ \xrightarrow{\ \mathrm{insert}\ j\ }\ \cS_1':=\cS_1^-\cup\{j\},
    \\
    &\cS_0\ \xrightarrow{\ \mathrm{delete}\ j\ }\ \cS_0^-:=\cS_0\setminus\{j\}
    \ \xrightarrow{\ \mathrm{insert}\ i\ }\ \cS_0':=\cS_0^-\cup\{i\},
\end{align*}
and
\begin{equation}\label{eq:paired-decomp}
    \Delta_{ij}\hat\tau_{\ols}(\cS_1)
        =
        \underbrace{\Delta^{\mathrm{del}}_1(i;\cS_1)+\Delta^{\mathrm{ins}}_1(j;\cS_1^-)}_{\text{treatment arm}}
        -
        \underbrace{\Bigl(\Delta^{\mathrm{del}}_0(j;\cS_0)+\Delta^{\mathrm{ins}}_0(i;\cS_0^-)\Bigr)}_{\text{control arm}}.
\end{equation}
This is a purely algebraic telescope (see Proposition \ref{prop:sensitivity_decomposition} and its proof in Appendix \ref{sec:proof_proposition_decomposition}).

\paragraph{Atomic Identities and Unit-wise Pairing.}
By Theorem~\ref{thm:perturbation_identity}, for any \(\arm\in\{0,1\}\), \(\cS\in\binom{[n]}{m}\), \(i\in\cS\), and \(j\notin\cS\), the regular deletion formula is
\begin{equation}\label{eq:atomic-del}
    \Delta^{\mathrm{del}}_\arm(i;\cS)
    =
    -\frac{
        \langle e_i,\bone\rangle_{\cS}
        \langle e_i,\ycenta{\cS}\rangle_{\cS}
    }{
        \Phidel_i(\cS)
    },
\end{equation}
whenever the regular deletion update applies and \(\Phidel_i(\cS)>0\). 
For insertion, Theorem~\ref{thm:perturbation_identity} gives the exact block-ratio formula in general, with the simplified \(K\)- and \(M\)-branch regular cases stated in \eqref{eq:exact-ins-K} and \eqref{eq:exact-ins-M}.

Define the unit‑wise contributions
\begin{equation}\label{eq:ci-dj}
    c_i \coloneqq \Delta^{\mathrm{del}}_1(i;\cS_1)\;-\;\Delta^{\mathrm{ins}}_0(i;\cS_0^-),
    \qquad
    d_j \coloneqq \Delta^{\mathrm{ins}}_1(j;\cS_1^-)\;-\;\Delta^{\mathrm{del}}_0(j;\cS_0).
\end{equation}
Substituting the deletion identity \eqref{eq:atomic-del} together with the appropriate branch-specific insertion identity from Lemma~\ref{lem:insertion} into \eqref{eq:paired-decomp} yields the symmetric exact form
\begin{equation}\label{eq:paired-exact}
    \Delta_{ij}\hat\tau_{\ols}(\cS_1)=c_i+d_j.
\end{equation}

\paragraph{A Remark on Oracle vs.\ Data‑driven Analyses.}
Because insertion terms involve incoming counterfactual outcomes \(y^{(1)}_j\) and \(y^{(0)}_i\), the exact \(c_i,d_j\) (and hence \(\Delta_{ij}\hat\tau_{\ols}\)) are oracle quantities. 
Appendix~\ref{sec:geometric_envelope} identifies the \(Q\)-geometric factors---deletion leverage, insertion leverage, and standardized residual scales---that govern these terms. 
Fully data-driven insertion envelopes additionally require observable bounds for the incoming residuals.

\subsection{Geometric Envelopes for Deletion and Insertion}\label{sec:geometric_envelope}

The exact swap sensitivity in \eqref{eq:paired-exact} combines atomic deletion/insertion changes in closed form, cf.\ \eqref{eq:atomic-del}, \eqref{eq:exact-ins-block}, \eqref{eq:exact-ins-K}, and \eqref{eq:exact-ins-M}. 
This subsection records geometry-driven envelopes for these terms and clarifies which parts are oracle and which would require additional data-driven proxies.

\paragraph{Deletion and Insertion Leverages.}
For \(\cA\subset[n]\) and \(u\in\cA\), define the \(Q\)-leverage
\[
    \ell_u(\cA)
    \coloneqq
    \frac{\langle e_u,\bone\rangle_{Q_{X_\cA}}^2}
    {\|e_u\|_{Q_{X_\cA}}^2\|\bone\|_{Q_{X_\cA}}^2}
    \in[0,1].
\]
Deletion uses the current leverage \(\ell_i(\cS)\). 
Insertion uses the augmented leverage
\[
    \ell_j^+(\cS)
    \coloneqq
    \ell_j(\cS\cup\{j\}),
\]
i.e., the leverage of the incoming unit after it has been inserted.

In the regular \(K\)-branch, this augmented leverage has the pre-insertion expression
\[
    \ell_j^{K,+}(\cS)
    =
    \frac{(1-\alpha_j(\cS))^2}
    {\|\bone\|_{Q_\cS}^2\sigma_j^2(\cS)+(1-\alpha_j(\cS))^2},
\]
where
\[
    \alpha_j(\cS)=\langle \bone, X_\cS x_j\rangle_{Q_\cS},
    \qquad
    \sigma_j^2(\cS)=\|x_j\|_2^2-\|X_\cS x_j\|_{Q_\cS}^2.
\]
In the \(M\)-branch row-span case \(c_j=0\), the augmented insertion leverage is
\[
    \ell_j^{M,+}(\cS)
    =
    \frac{(1-\eta_j)^2}
    {\|\bone\|_{Q_\cS}^2(1+h_j)+(1-\eta_j)^2},
\]
with \(h_j,\eta_j\) as in \eqref{eq:exact-ins-M}. 
If \(c_j\ne0\), then the \(M\)-branch insertion effect is zero.

\smallskip

\begin{remark}[Relation to classical hat leverages]
    The quantity \(\ell_u(\cA)\), which we refer to as the ``\(Q\)-leverage'' (or simply ``leverage''), is a directional cosine in \(Q\)-geometry and is not equal to the diagonal of the OLS hat matrix with intercept, which is commonly referred to as the ``leverage score'' in the regression analysis.   
    In the special case \(Q_\cA=M_{X_\cA}\) and \(1^\top X_{\cA}=0\),
    \[
        \ell_{u}(\cA)\ =\ \frac{1}{\,|\cA|\,(1-h_{uu}(X_\cA))\,},
        \qquad\text{where}\qquad
        h_{uu}(X_{\cA}) = \bigl[ X_{\cA} (X_{\cA}^{\top} X_{\cA})^{-1} X_{\cA}^{\top}  \bigr]_{uu},
    \]
    linking our $Q$-leverage to the familiar hat diagonals $h_{ii}$. 
    We use the term “leverage” to denote this \(Q\)-cosine and retain the classical term “hat leverage” when referring to \(h_{uu}(X_{\cA})\).
\end{remark}



\paragraph{Geometric Envelopes.}
From the regular deletion identity \eqref{eq:atomic-del},
\[
    \Delta^{\mathrm{del}}_\arm(i;\cS)
        =
        -\frac{\langle e_i,\bone\rangle_{\cS}\langle e_i,\ycenta{\cS}\rangle_{\cS}}{\|\bone\|_{\cS}^2\,\|e_i\|_{\cS}^2-\langle e_i,\bone\rangle_{\cS}^2}.
\]
Taking absolute values gives the envelope:
\begin{equation}\label{eqn:deletion_envelope}
    \big|\Delta^{\mathrm{del}}_\arm(i;\cS)\big|
    =
    R^{\mathrm{del}}_{\arm,i}(\cS)
    \frac{\sqrt{\ell_i(\cS)}}{1-\ell_i(\cS)},
    \qquad
    R^{\mathrm{del}}_{\arm,i}(\cS)
    \coloneqq
    \frac{|\langle e_i,\ycenta{\cS}\rangle_{\cS}|}
    {\|e_i\|_{\cS}\|\bone\|_{\cS}}.
\end{equation}

Insertion has the same structure when viewed as reverse deletion. 
Let \(\cS^+=\cS\cup\{j\}\). Since
\[
    \Delta^{\mathrm{ins}}_\arm(j;\cS)
    =
    -\Delta^{\mathrm{del}}_\arm(j;\cS^+),
\]
the augmented-set leverage gives, in regular cases,
\begin{equation}\label{eq:insertion-augmented-leverage}
    \big|\Delta^{\mathrm{ins}}_\arm(j;\cS)\big|
    =
    R^+_{\arm,j}(\cS)
    \frac{\sqrt{\ell_j^+(\cS)}}{1-\ell_j^+(\cS)},
\end{equation}
where
\begin{equation}\label{eqn:residual_augmented}
    R^+_{\arm,j}(\cS)
    \coloneqq
    \frac{
    |\langle e_j,\widetilde y^{(\arm)}_{\cS^+}\rangle_{Q_{X_{\cS^+}}}|
    }{
    \|e_j\|_{Q_{X_{\cS^+}}}\|\bone\|_{Q_{X_{\cS^+}}}
    }.
\end{equation}
Thus deletion and insertion obey the same leverage--residual template; only the set defining the geometry changes from \(\cS\) to \(\cS^+\).

The augmented-set representation \eqref{eq:insertion-augmented-leverage} provides a symmetrical viewpoint that treats insertion as deletion in \(\cS^+\). 
For computation and interpretation, it is also useful to rewrite this same quantity in pre-insertion coordinates. 
In the regular \(K\)-branch insertion case, \eqref{eq:exact-ins-K} to be presented in Lemma~\ref{lem:insertion} can be equivalently written as
\begin{equation}\label{eq:K-insertion-leverage-form}
    \big|\Delta^{\mathrm{ins}}_\arm(j;\cS)\big|
    =
    R^{K,+}_{\arm,j}(\cS)
    \sqrt{\ell_j^{K,+}(\cS)\bigl(1-\ell_j^{K,+}(\cS)\bigr)},
\end{equation}
where
\[
    R^{K,+}_{\arm,j}(\cS)
    \coloneqq
    \frac{
    \left|y_j^{(\arm)}-\hat\mu_\arm(\cS)
    -
    \langle X_\cS x_j,\ycenta{\cS}\rangle_{Q_\cS}\right|
    }{
    \sigma_j(\cS)\|\bone\|_{Q_\cS}
    }.
\]
Indeed, this follows from
\[
    \sqrt{\ell_j^{K,+}(\cS)\bigl(1-\ell_j^{K,+}(\cS)\bigr)}
    =
    \frac{|1-\alpha_j(\cS)|\,\sigma_j(\cS)\|\bone\|_{Q_\cS}}
    {\|\bone\|_{Q_\cS}^2\sigma_j^2(\cS)+(1-\alpha_j(\cS))^2}.
\]
Similarly, in the \(M\)-branch row-span case \(c_j=0\), \eqref{eq:exact-ins-M} to be presented in Lemma~\ref{lem:insertion} gives
\begin{equation}\label{eq:M-insertion-leverage-form}
    \big|\Delta^{\mathrm{ins}}_\arm(j;\cS)\big|
    =
    R^{M,+}_{\arm,j}(\cS)
    \sqrt{\ell_j^{M,+}(\cS)\bigl(1-\ell_j^{M,+}(\cS)\bigr)},
\end{equation}
where
\[
    R^{M,+}_{\arm,j}(\cS)
    \coloneqq
    \frac{|y_j^{(\arm)}-\hat\mu_\arm(\cS)-\beta_j|}
    {\sqrt{1+h_j}\,\|\bone\|_{Q_\cS}}.
\]
If \(c_j\ne0\), then \(\Delta^{\mathrm{ins}}_\arm(j;\cS)=0\). 
Degenerate or branch-switching cases are handled by the exact ratio and block identities in Appendix~\ref{sec:proof_theorem_perturbation}.

\paragraph{Remarks on the Envelopes.} 
The deletion envelope is amplified when \(i\) is highly aligned with the intercept direction in the current \(Q_\cS\)-geometry, i.e., when \(\ell_i(\cS)\) is close to one. 
Insertion is amplified when the incoming unit has high leverage in the augmented geometry \(Q_{X_{\cS\cup\{j\}}}\), i.e., when \(\ell_j^+(\cS)\) is close to one. 
Thus, OLS--RA swap effects are small when deletion leverage, insertion leverage, and the corresponding standardized residual scales are controlled.

The insertion residuals in \(R^+_{\arm,j}\), \(R^{K,+}_{\arm,j}\), and \(R^{M,+}_{\arm,j}\) involve the incoming potential outcome \(y_j^{(\arm)}\), which is counterfactual when \(j\) lies in the opposite arm. 
A fully data-driven insertion proxy therefore requires an additional observable envelope for these incoming residuals. 
Developing such calibrated envelopes is left to future work.

\subsection{Envelope-controlled Concentration and Bias}\label{sec:leverage_bounds}

The exact formulas above give a geometry-driven swap envelope for OLS--RA. 
For \(\arm\in\{0,1\}\), define the local deletion envelope
\[
    \mathcal D_{\arm,i}(\cA)
    \coloneqq
    \begin{cases}
    R^{\mathrm{del}}_{\arm,i}(\cA)
    \dfrac{\sqrt{\ell_i(\cA)}}{1-\ell_i(\cA)},
        &\text{if the regular deletion identity applies},\\[0.8em]
    \big|\Delta^{\mathrm{del}}_\arm(i;\cA)\big|,
        &\text{otherwise, using the exact ratio identity}.
    \end{cases}
\]
Similarly, define the local insertion envelope
\[
    \mathcal I_{\arm,j}(\cA)
    \coloneqq
    \begin{cases}
    R^+_{\arm,j}(\cA)
    \dfrac{\sqrt{\ell_j^+(\cA)}}{1-\ell_j^+(\cA)},
        &\text{if the regular augmented deletion identity applies},\\[0.8em]
    \big|\Delta^{\mathrm{ins}}_\arm(j;\cA)\big|,
        &\text{otherwise, using the exact block identity}.
    \end{cases}
\]
Then we set
\[
    \mathfrak D_\arm
    \coloneqq
    \sup_{\substack{\cA\subset[n],\,|\cA|=n_\arm\\ i\in\cA}}
    \mathcal D_{\arm,i}(\cA),
    \qquad\text{and}\qquad
    \mathfrak I_\arm
    \coloneqq
    \sup_{\substack{\cA\subset[n],\,|\cA|=n_\arm-1\\ j\notin\cA}}
    \mathcal I_{\arm,j}(\cA).
\]
In regular \(K\)- and \(M\)-branch cases, \(\mathcal I_{\arm,j}\) can equivalently be computed from the pre-insertion forms \eqref{eq:K-insertion-leverage-form} and \eqref{eq:M-insertion-leverage-form}. 
Finally, set
\[
    \Delta_{\mathrm{geo}}
    \coloneqq
    \mathfrak D_1+\mathfrak I_1+\mathfrak D_0+\mathfrak I_0.
\]

\begin{lemma}[Geometry-controlled oracle parameters]\label{lem:leverage-oracle-bounds}
    For OLS--RA with \(f(\cS)=\hat\tau_{\OLS}(\cS)-\tau\), the oracle parameters satisfy
    \begin{equation}\label{eq:oracle-geom-bounds}
        \Rstar \leq \Delta_{\mathrm{geo}},\qquad
        \Vstar \leq A_n\,\Delta_{\mathrm{geo}}^2,\qquad
        \bbE\Gamma(f) \leq \tfrac{1}{2}\Delta_{\mathrm{geo}}^2,
    \end{equation}
    where \(A_n\coloneqq \sum_{t=1}^{n_1}\alpha_t^2\). 
    Moreover, 
    \begin{equation}\label{eq:Bstar-geom-with-bias}
        (\Bstar)^2
        \le
        |\bbE f|^2
        +
        \frac{\Delta_{\mathrm{geo}}^2}{(\lambda^\star)^2}
        \le
        |\bbE f|^2
        +
        \left(
            \frac{n_1n_0}{n}\Delta_{\mathrm{geo}}
        \right)^2.
    \end{equation}
\end{lemma}

\begin{proof}[Proof of Lemma~\ref{lem:leverage-oracle-bounds}]

By construction of \(\mathfrak D_\arm\) and \(\mathfrak I_\arm\), together with the paired telescope \eqref{eq:paired-decomp},
\[
    |\Delta_{ij}\hat\tau_{\OLS}(\cS_1)|
    \le
    \mathfrak D_1+\mathfrak I_1+\mathfrak D_0+\mathfrak I_0
    =
    \Delta_{\mathrm{geo}}
\]
for every admissible swap \((i,j)\in\cS_1\times\cS_0\). 
The rest follows by propagating this geometric swap envelope through the definitions of \(\Rstar,\Vstar,\Gamma(f)\), and \(\Bstar\).

By Proposition~\ref{prop:reveal-swap}, \(D_t=-\alpha_t g_t(I)\), where
\[
    g_t(i)
    =
    \bbE\left[
        \Delta_{iJ}f(\Sprox{t-1}(i,\cT))
        \mid \cF_{t-1}
    \right],
    \qquad
    I\mid\cF_{t-1}\sim\Unif(\cR_{t-1}).
\]
The envelope gives \(|g_t(i)|\le \Delta_{\mathrm{geo}}\). Hence
\[
    r_t^\star\le \alpha_t\Delta_{\mathrm{geo}}\le \Delta_{\mathrm{geo}},
    \qquad
    v_t^\star
    =
    \alpha_t^2\Var(g_t(I)\mid\cF_{t-1})
    \le
    \alpha_t^2\Delta_{\mathrm{geo}}^2.
\]
Summing over \(t\) gives
\[
    \Rstar\le \Delta_{\mathrm{geo}},
    \qquad
    \Vstar\le A_n\Delta_{\mathrm{geo}}^2.
\]
Moreover, by \eqref{eq:Gamma-expansion},
\[
    \Gamma(f)(\cS)
    =
    \frac{1}{2n_1n_0}
    \sum_{i\in\cS}\sum_{j\notin\cS}
    \big(\Delta_{ij}f(\cS)\big)^2
    \le
    \frac12\Delta_{\mathrm{geo}}^2,
\]
and hence \(\bbE\Gamma(f)\le \tfrac12\Delta_{\mathrm{geo}}^2\).

Finally, since \(Lf(\cS)\) is the average one-swap change of \(f\), \(|Lf(\cS)|\le\Delta_{\mathrm{geo}}\). 
Let
\(\psi_c^\star\coloneqq Lf+\lambda^\star(f-\bbE f)\) 
as in \eqref{eqn:centered_stein_residual}. 
By the definition of \(\lambda^\star\) as the minimizer of \(\bbE\{Lf+\lambda(f-\bbE f)\}^2\),
\[
    \bbE(\psi_c^\star)^2
    \le
    \bbE(Lf)^2
    \le
    \Delta_{\mathrm{geo}}^2.
\]
Since
\(Lf+\lambda^\star f = \psi_c^\star+\lambda^\star\bbE f\) and \(\bbE\psi_c^\star=0\),
we obtain
\[
    (\Bstar)^2
    =
    \frac{\bbE(Lf+\lambda^\star f)^2}{(\lambda^\star)^2}
    =
    |\bbE f|^2
    +
    \frac{\bbE(\psi_c^\star)^2}{(\lambda^\star)^2}
    \le
    |\bbE f|^2+\frac{\Delta_{\mathrm{geo}}^2}{(\lambda^\star)^2}.
\]
Using \(\lambda^\star\ge \mathrm{gap}_{n,n_1}=n/(n_1n_0)\) gives \eqref{eq:Bstar-geom-with-bias}.
\end{proof}

\begin{corollary}[Typical scaling]
    If the geometry-driven envelope above satisfies \(\Delta_{\mathrm{geo}}=o(1)\), then
    \[
        \Rstar=O(\Delta_{\mathrm{geo}}),\qquad
        \Vstar=O(A_n\Delta_{\mathrm{geo}}^2),\qquad
        \bbE\Gamma(f)=O(\Delta_{\mathrm{geo}}^2).
    \]
    In particular, if \(A_n=O(n)\) and \(\Delta_{\mathrm{geo}}=O(n^{-1/2})\), then \(\Vstar=O(1)\). 
    For the corrected Stein bias parameter,
    \[
        \Bstar
        \le
        |\bbE f|+\frac{\Delta_{\mathrm{geo}}}{\lambda^\star}
        \le
        |\bbE f|+\frac{n_1n_0}{n}\Delta_{\mathrm{geo}}.
    \]
    The final inequality is a worst-case Poincar\'e fallback and can be loose in balanced designs; sharper control depends on the actual Rayleigh quotient \(\lambda^\star\) and the centered Stein residual.
\end{corollary}

\paragraph{Tightness of the Bounds in Lemma \ref{lem:leverage-oracle-bounds}.}
The constants in \eqref{eq:oracle-geom-bounds} are sharp up to the Cauchy–Schwarz slack in the standardized residual scales. 
The bounds are informative when deletion leverage and insertion leverage are moderate and the standardized deletion/insertion residuals are small. 
Large leverage, large incoming residuals, or nearly singular innovation scales can inflate the local range and variance, possibly reflecting genuine sensitivity rather than merely looseness of the current analysis.
\section{DEFERRED PROOFS FROM SECTION \ref{sec:finite_sample_analysis}}\label{sec:proof_section3}

\subsection{Proof of Proposition \ref{prop:mu_ratio}}\label{sec:proof_proposition.quadratic}

\begin{proof}[Proof of Proposition~\ref{prop:mu_ratio}]
    Let $L(\mu,\beta)=\|y-\mu\bone_n-X\beta\|_2^2$. 
    The normal equations for least-squares are
    \begin{align}
        \nabla_{\beta} L(\mu, \beta) &= -2 \cdot X^\top\bigl(y-\mu\bone_n-X\beta\bigr)=0,
            \label{eq:normal-eqs.1}\\
        \nabla_{\mu} L(\mu, \beta) &= -2 \cdot \bone_n^\top\bigl(y-\mu\bone_n-X\beta\bigr)=0.
            \label{eq:normal-eqs.2}
    \end{align}
    We prove this proposition in three steps.
    
    \begin{itemize}
        \item 
        \textbf{Step 1: Characterization of $\cS_{X,y}$ (the set of LS minimizers).} 
        Observe that the normal equations \eqref{eq:normal-eqs.1}--\eqref{eq:normal-eqs.2} form first-order (necessary) conditions for optimality for the least squares; that is, if $(\mu, \beta) \in \cS_{X,y}$, then these equations must be satisfied. 
        
        For each $\mu\in\RR$, the set of $\beta$ satisfying \eqref{eq:normal-eqs.1} is the affine set
        \begin{equation}\label{eqn:set_LS}
            \cS(\mu) \coloneqq \bigl\{\ \beta(\mu)+(I_p-X^\dagger X)w \ :\ w\in\RR^p\ \bigr\},
            \qquad\text{where}\qquad
            \beta(\mu)\coloneqq X^\dagger(y-\mu\bone_n).
        \end{equation}
        For any $\beta \in \cS(\mu)$, the regression residual is
        \begin{align*}
            r(\mu) &= y-\mu\bone_n-X\beta\\
                   &=y-\mu\bone_n-XX^\dagger(y-\mu\bone_n)  &   \because X (I_p - X^{\dagger}X ) = 0\\
                   &=M_X\,(y-\mu\bone_n).
        \end{align*}
        Therefore, \eqref{eq:normal-eqs.2} is equivalent to the \emph{scalar} constraint
        \begin{equation}\label{eq:mu-feasibility}
            \bone_n^\top M_X\,(y-\mu\bone_n)=0.
        \end{equation}
        \begin{itemize}
            \item 
            If $\bone_n\notin\colsp(X)$, then $M_X\bone_n\neq 0$ and \eqref{eq:mu-feasibility} determines a \emph{unique} intercept
            \begin{equation}\label{eq:mu-FOC-M}
                \mu^\star \;=\; \frac{\bone_n^\top M_X y}{\bone_n^\top M_X \bone_n}.
            \end{equation}
            \item 
            If $\bone_n\in\colsp(X)$, then $M_X\bone_n=0$, and \eqref{eq:mu-feasibility} holds automatically; \emph{every} $\mu\in\RR$ is feasible in this case. 
        \end{itemize}
        
        In summary,
        \begin{equation}\label{eq:Sxy-character}
            \cS_{X,y}
                =\Bigl\{\,(\mu,\,\beta(\mu)+(I_p-X^\dagger X)w)\ :\ \mu\in\cM,\; w\in\RR^p\,\Bigr\},
                \qquad\text{where}\qquad
                \cM=
                    \begin{cases}
                        \{\mu^\star\}, & \bone_n\notin\colsp(X),\\
                        \RR, & \bone_n\in\colsp(X).
                    \end{cases}
        \end{equation}
    
        \item 
        \textbf{Step 2: Uniqueness of the minimum-norm OLS solution.}
        We minimize $\|\beta\|_2^2$ over $(\mu,\beta)\in\cS_{X,y}$. 
        Let $\mathsf{N}(X) = \{ v \in \RR^p: X v = 0 \}$ denote the null space of $X$. 
        Observe that
        \[
            \beta(\mu)\in\mathrm{range}(X^\top),\qquad 
            (I_p-X^\dagger X)w\in\mathsf{N}(X),\qquad 
            \mathrm{range}(X^\top)\perp\mathsf{N}(X).
        \]
        For any feasible $\mu$ and $w$,
        \begin{align*}
            \|\beta(\mu)+(I_p-X^\dagger X)w\|_2^2
                &=\|\beta(\mu)\|_2^2+\|(I_p-X^\dagger X)w\|_2^2\\
                &\geq \|\beta(\mu)\|_2^2,
        \end{align*}
        with equality iff $w=0$, cf. \eqref{eqn:set_LS}. 
        Thus \emph{for each feasible $\mu$} the minimum-norm slope is unique and equals $\beta(\mu)$. It remains to minimize $\|\beta(\mu)\|_2^2$ over $\mu\in\cM$.
        \begin{itemize}
            \item 
            If $\bone_n\notin\colsp(X)$, then $\cM=\{\mu^\star\}$ and the unique minimizer is $(\hat\mu,\hat\beta)=(\mu^\star,\beta(\mu^\star))$.
            \item 
            If $\bone_n\in\colsp(X)$, then for all $\mu\in\RR$,
            \begin{equation}\label{eqn:mu_quadratic}
            \begin{aligned}
                \|\beta(\mu)\|_2^2
                    &= (y-\mu\bone_n)^\top (X^\dagger)^\top X^\dagger (y-\mu\bone_n)\\
                    &= (y-\mu\bone_n)^\top (XX^\top)^\dagger (y-\mu\bone_n),
            \end{aligned}
            \end{equation}
            a strictly convex quadratic in $\mu$ because $\bone_n\in\mathrm{range}(XX^\top)$ and $(XX^\top)^\dagger$ is positive definite on that range. Hence there is a \emph{unique} minimizer $\hat\mu$ and thus a \emph{unique} pair $(\hat\mu,\hat\beta)$ with $\hat\beta=\beta(\hat\mu)$.
        \end{itemize}
        Therefore, the minimum-norm OLS pair $(\hat\mu,\hat\beta)$ is \emph{unique} among all LS minimizers.

        \item
        \textbf{Step 3: Identification and ratio-of-quadratics formulas.}
        Lastly, we identify the unique $\hat{\mu}$.

        \begin{itemize}
            \item 
            In the first case ($\bone_n\notin\colsp(X)$), the unique $\hat\mu$ is exactly \eqref{eq:mu-FOC-M}, i.e.
            \[
                \hat\mu = \frac{\bone_n^\top M_X y}{\bone_n^\top M_X \bone_n},
            \]
            with $\hat\beta=X^\dagger(y-\hat\mu\,\bone_n)$. Note $\bone_n^\top M_X\bone_n=\|P_{\mathsf{N}(X^\top)}\bone_n\|_2^2>0$.

            \item 
            In the second case ($\bone_n\in\colsp(X)$), minimizing the strictly convex quadratic objective \eqref{eqn:mu_quadratic} in $\mu$ yields
            \[
                \hat\mu = \frac{\bone_n^\top (XX^\top)^\dagger y}{\bone_n^\top (XX^\top)^\dagger \bone_n},
                \qquad
                \hat\beta=X^\dagger(y-\hat\mu\,\bone_n),
            \]
            with $\bone_n^\top (XX^\top)^\dagger \bone_n=\|(XX^\top)^{\dagger 1/2}\bone_n\|_2^2>0$.
        \end{itemize}
        
        Combining the two cases gives \eqref{eq:mu-ratio}–\eqref{eq:Qa-piecewise}. 
    \end{itemize}
\end{proof}

\subsection{Proof of Proposition \ref{prop:reveal-swap}}\label{sec:proof_lemma_swap}

\begin{proof}[Proof of Proposition \ref{prop:reveal-swap}]
    Given $\cF_{t-1}$, let $\cK$ --- the feasible ``completion set'' that complements $\cS^{\mathrm{past}}_{t-1}$ to form a feasible treatment set $\Sone$ --- be distributed as $\cK \sim \Unif\binom{\cR_{t-1}}{\nass}$. 
    Then $\bbE[f\mid \cF_{t-1}] = \bbE\big[f(\cS^{\mathrm{past}}_{t-1}\cup \cK)\big]$.

    Conditioning on $\cF_t = \sigma( \cF_{t-1}, I)$ forces $I \in \cK$, i.e., $\cK=\{I\}\cup \cT$ with $\cT\sim \Unif\binom{\cR_{t-1}\setminus\{I\}}{\nass-1}$, so $\bbE[f\mid \cF_t] = \bbE_{\cT}\big[f(\Sprox{t-1}(I,\cT))\big]$. 

    Write $g(\cK) \coloneqq f\big(\cS^{\mathrm{past}}_{t-1}\cup \cK\big)$. 
    Then
    \begin{align*}
        \bbE[f\mid \cF_{t-1}]&=\bbE\big[ g(\cK)\mid \cF_{t-1}\big],\\
        \bbE[f\mid \cF_t]&=\bbE\big[ g(\cK)\mid \cF_{t-1},\, I\in \cK\big].
    \end{align*}
    Let $p_t \coloneqq \Pr(I\in \cK\mid \cF_{t-1})=\nass/\rem$. 
    By the law of total expectation,
    \[
        \bbE[g(\cK)\mid \cF_{t-1}]
            = p_t\,\bbE[g(\cK)\mid \cF_{t-1},\, I\in \cK]
                + (1-p_t)\,\bbE[g(\cK)\mid \cF_{t-1},\, I\notin \cK].
    \]
    Hence,
    \begin{align*}
        D_t
            &= \bbE[g(\cK)\mid \cF_{t-1},\, I\in \cK]-\bbE[g(\cK)\mid \cF_{t-1}]\\
            &= (1-p_t)\Big(\bbE[g(\cK)\mid \cF_{t-1},\, I\in \cK]
            - \bbE[g(\cK)\mid \cF_{t-1},\, I\notin \cK]\Big)\\
            &= \frac{\rem-\nass}{\rem}\Big(\bbE[f\mid \cF_{t-1},\, I\in \cK] - \bbE[f\mid \cF_{t-1},\, I\notin \cK]\Big).
    \end{align*}
   
    To evaluate the difference $\bbE[f\mid \cF_{t-1},\, I\in \cK]-\bbE[f\mid \cF_{t-1},\, I\notin \cK]$ appearing above, we now construct a coupling under the $I\notin \cK$ branch. 
    Under the condition $I\notin \cK$, write $\cK=\{J\}\cup \cT$ with $J \sim \Unif\bigl( \cR_{t-1}\setminus\{I\} \bigr)$ and $\cT \sim \Unif \bigl( \binom{\cR_{t-1}\setminus\{I,J\}}{\nass-1} \bigr)$. 
    Pair the configurations $\cK_I \coloneqq \{I\}\cup \cT$ and $\cK_J \coloneqq \{J\}\cup \cT$ and note 
    \[
        f\big(\cS^{\mathrm{past}}_{t-1}\cup \cK_I\big)-f\big(\cS^{\mathrm{past}}_{t-1}\cup \cK_J\big)
            = -\,\Delta_{I\,J}\,f\big( \Sprox{t-1}(I,\cT)\big),
    \]
    Averaging completes the proof and yields \eqref{eq:rev-swap-avg}.
\end{proof}

\subsection{Proof of Proposition \ref{prop:concentration}}\label{sec:proof_concentration}

Let $(\cF_t)_{t=0}^{n_1}$ be the assignment–exposure filtration from \eqref{eqn:filtration}, and for a measurable function $f$, define for $t=1,\ldots,n_1$,
\[
    M_t \coloneqq \bbE\big[f(\Sone)\mid\cF_t\big],
    \qquad\text{and}\qquad
    D_t \coloneqq M_t - M_{t-1}.
\]
Then $M_0=\bbE f(\Sone)$ and $M_{n_1}=f(\Sone)$, so $f(\Sone)-\bbE f(\Sone)=\sum_{t=1}^{n_1} D_t$, with
$\bbE[D_t\mid\cF_{t-1}]=0$.
As in \eqref{eq:var_star}–\eqref{eq:range_star}, define the predictable one–step variance and range
\[
    v_t^\star \coloneqq \Var(D_t\mid \cF_{t-1}),
    \qquad\text{and}\qquad
    r_t^\star \coloneqq \sup |D_t|,
\]
and their aggregates $V^\star \coloneqq \sum_{t=1}^{n_1} v_t^\star$ and $R^\star \coloneqq \max_{1\le t\le n_1} r_t^\star$ ($\sup$ in the definition of $r_t$ is taken over the law of $D_t\mid \cF_{t-1}$). 
By construction, $v_t^\star,r_t^\star$ are $\cF_{t-1}$–measurable and $|D_t|\le r_t^\star$ almost surely.

\subsubsection{Background: Freedman's Inequality}
We use the following standard form of Freedman’s inequality \citep[Theorem~1.6]{freedman1975tail}.
\begin{lemma}[Freedman]\label{lem:freedman}
    Let $(M_t,\cF_t)$ be a martingale with differences $D_t=M_t-M_{t-1}$ satisfying $\bbE[D_t\mid\cF_{t-1}]=0$
    and $|D_t|\le b$ almost surely. Set $V_t\coloneqq\sum_{s=1}^t \bbE[D_s^2\mid\cF_{s-1}]$. Then for any $\varepsilon\ge0$ and $\sigma^2>0$,
    \[
        \Pr \left(\,\exists t\ge 0:\ M_t\ge \varepsilon\ \text{ and }\ V_t\le \sigma^2\,\right)
            \leq \exp \left(-\frac{\varepsilon^2/2}{\sigma^2 + b\,\varepsilon/3}\right).
    \]
\end{lemma}
A convenient self‑normalized (empirical‑Bernstein) corollary follows by plugging $\sigma^2=V_t$ inside the
event and optimizing $\varepsilon$.

\begin{corollary}[Self‑normalized one‑sided form]\label{cor:freedman-self}
    Under Lemma~\ref{lem:freedman}, for any $L>0$,
    \[
        \Pr \left(\,\exists t\ge 0:\ M_t\ \ge\ \sqrt{2 V_t\,L}\ +\ \frac{b}{3}\,L\,\right)\ \le\ e^{-L}.
    \]
    In particular, for any fixed horizon $T$,
    \[
        \Pr \left(\ M_T\ \ge\ \sqrt{2 V_T\,L}\ +\ \frac{b}{3}\,L\ \right)\ \le\ e^{-L}.
    \]
\end{corollary}

\begin{proof}[Proof (sketch)]
    Fix $L>0$ and consider $\cE=\{\exists t:\ M_t\ge \sqrt{2V_tL}+(b/3)L\}$. For each $t\in\cE$, set
    $\varepsilon=\sqrt{2V_tL}+(b/3)L$ and $\sigma^2=V_t$ in Lemma~\ref{lem:freedman} to get $\Pr(\cE)\le e^{-L}$. 
    The fixed‑time form is the corresponding sub‑event.
\end{proof}

\subsubsection{Completing Proof of Proposition \ref{prop:concentration}}
\begin{proof}[Proof of Proposition \ref{prop:concentration}]
    Apply Corollary~\ref{cor:freedman-self} to the martingale $(M_t)$ above. Since $|D_t|\le r_t^\star$ a.s. and
    $r_t^\star$ is $\cF_{t-1}$–measurable, we may take $b=R^\star=\max_t r_t^\star$ (conditioning on the realized predictable sequence).
    Likewise, $V_{n_1}=V^\star$ because $\bbE[D_t^2\mid\cF_{t-1}]=\Var(D_t\mid\cF_{t-1})=v_t^\star$. Therefore, for any $L>0$,
    \begin{equation}\label{eq:one-sided-freedman}
        \Pr \left(\, M_{n_1}-M_0 \ \ge\ \sqrt{2 V^\star\,L}\ +\ \frac{R^\star}{3}\,L\,\right) \ \le\ e^{-L}.
    \end{equation}
    Recalling $M_{n_1}-M_0=f(\Sone)-\bbE f(\Sone)$, \eqref{eq:one-sided-freedman} yields the one‑sided tail form
    \[
        \Pr \left(\, f(\Sone)-\bbE f(\Sone) \ \ge\ \sqrt{2 V^\star\,L}\ +\ \frac{R^\star}{3}\,L\,\right) \ \le\ e^{-L}.
    \]
    Applying the same bound to $-f$ and taking a union bound gives, for any $\delta\in(0,1)$ with
    $L_\delta\coloneqq\log\frac{2}{\delta}$,
    \[
        \Pr \left(\,|f(\Sone)-\bbE f(\Sone)| \ \le\ \sqrt{2 V^\star\,L_\delta}\ +\ \frac{R^\star}{3}\,L_\delta\,\right)\ \ge\ 1-\delta,
    \]
    which is the stated concentration inequality in Proposition~\ref{prop:concentration}.  
\end{proof}

\smallskip

\begin{remark}[Predictability and sharpness]
    The parameters $(V^\star,R^\star)$ are \emph{predictable} (i.e., $\cF_{t-1}$–measurable stepwise) and enter the bound
    \emph{inside the event}, which is legitimate because the proof conditions on the realized predictable sequence and
    then applies Freedman with those bounds (a standard empirical‑Bernstein device). This retains pathwise sharpness
    relative to range‑based, non‑adaptive envelopes. 
\end{remark}

\subsection{Proof of Proposition \ref{prop:stein-bias}}\label{sec:proof_Stein_bias}

Although Proposition \ref{prop:stein-bias} follows from the discussions in the main text (Section \ref{sec:bias_exchangeable}), we provide a formal proof here for completeness.

\begin{proof}[Proof of Proposition~\ref{prop:stein-bias}]
Throughout, expectations $\bbE[\cdot]$ and $L^2$ inner products $\langle f,g\rangle\coloneqq \bbE[f(\cS)g(\cS)]$ are with respect to the uniform law $\pi$ on $\binom{[n]}{m}$, for which the one-swap kernel $P$ on $J(n,m)$ is reversible.
Recall the generator $L=P-\Id$ and the carr\'e du champ
\[
\Gamma(f)(\cS) \;=\; \frac12\,\bbE \big[(f(\cS')-f(\cS))^2\mid \cS\big],
\]
as defined in \eqref{eq:stein-op}–\eqref{eq:stein-reg}.

Assume \(\Var(f)>0\). 
Let \(g\coloneqq f-\bbE f\), and let \(\lambda^\star\) be defined according to \eqref{eqn:lambda_rayleigh}. 
By the Dirichlet identity and the computation in \eqref{eq:lambda-star},
\[
    \lambda^\star=\frac{\bbE\Gamma(f)}{\Var(f)}>0.
\]
Define the uncentered Stein remainder
\[
    \psi_u^\star \coloneqq Lf+\lambda^\star f.
\]
Since \(\bbE Lf=0\), we have
\[
    \bbE\psi_u^\star
    =
    \lambda^\star \bbE f.
\]
Therefore, by Cauchy--Schwarz inequality, we obtain \eqref{eq:bias-prop} as follows, completing the proof:
\[
    |\bbE f|
    =
    \frac{|\bbE\psi_u^\star|}{\lambda^\star}
    \le
    \frac{\sqrt{\bbE(\psi_u^\star)^2}}{\lambda^\star}
    =
    \frac{\sqrt{\bbE\big(Lf+\lambda^\star f\big)^2}}{\lambda^\star}.
\]
\end{proof}

\section{DEFERRED PROOFS FROM SECTION \ref{sec:main_results}}\label{sec:proof_section4}

\subsection{Proof of Theorem \ref{thm:oracle}}\label{sec:proof_theorem_oracle}

\begin{proof}[Proof of Theorem~\ref{thm:oracle}]
    Let $f(\cS)=\hat\tau(\cS)-\tau$ and recall the assignment-exposure martingale $(M_t,\cF_t)$ from \eqref{eqn:filtration}–\eqref{eqn:martingale} with increments $D_t = M_t-M_{t-1}$. 
    By the triangle inequality,
    \[
        |f(\cS_1)| \leq |f(\cS_1)-\bbE f(\cS_1)| + |\bbE f(\cS_1)|.
    \]
    
    For the fluctuation term $|f(\cS_1)-\bbE f(\cS_1)|$, apply Proposition~\ref{prop:concentration} in its two-sided form \eqref{eq:freedman_conc.2}, cf. Remark \ref{rem:freedman_2sided}. 
    With $\Vstar, \Rstar$ defined in \eqref{eqn:oracle_concentration_param} as the aggregates of the predictable parameters $v_t^\star=\Var(D_t\mid\cF_{t-1})$ and $r_t^\star = \sup|D_t|$ from \eqref{eq:var_star}–\eqref{eq:range_star}, we obtain:
    \begin{equation}\label{eqn:thm_concentration}
        \Pr \left(\,|f(\cS_1)-\bbE f(\cS_1)| \leq \sqrt{2\,\Vstar\,\sfL_\delta}+\frac{\Rstar}{3}\,\sfL_\delta\,\right) \geq 1-\delta,
        \qquad\text{where}\qquad  \sfL_\delta=\log\frac{2}{\delta}.
    \end{equation}
    
    For the bias term $|\bbE f(\cS_1)|$, recall the Stein bias bound from Proposition~\ref{prop:stein-bias}. 
    \begin{itemize}
        \item 
        If \(\Var(f)=0\), then \(f\) is constant under the design, so \(|f(\cS_1)|=|\bbE f|\). 
        In this degenerate case we interpret \(\Bstar\) as \(|\bbE f|\). 
        
        \item 
        Otherwise, if $\Var(f)>0$, then Proposition~\ref{prop:stein-bias} yields
        \begin{equation}\label{eqn:thm_bias}
            |\bbE f(\cS_1)| \leq \Bstar =
            \frac{\sqrt{\bbE\big(Lf+\lambda^\star f\big)^2}}{\lambda^\star},
            \qquad
            \lambda^\star=\frac{\bbE\Gamma(f)}{\Var(f)}.
        \end{equation}
    \end{itemize}
    
    Combining the two displays \eqref{eqn:thm_concentration} and \eqref{eqn:thm_bias} yields \eqref{eqn:main_oracle}. 
    Note the bias bound is deterministic (no probability), so the overall success probability remains $1-\delta$ from \eqref{eqn:thm_concentration}.
\end{proof}

\subsection{Proof of Proposition \ref{prop:sensitivity_decomposition}}\label{sec:proof_proposition_decomposition}

\begin{proof}[Proof of Proposition~\ref{prop:sensitivity_decomposition}]
    Fix $\cS_1\in\binom{[n]}{n_1}$ and $(i,j)\in \cS_1\times \cS_0$. 
    Let $\cS_1'=\cS_1^{(i\leftrightarrow j)}=(\cS_1\setminus\{i\})\cup\{j\}$ and $\cS_0'=\cS_0^{(j\leftrightarrow i)}=(\cS_0\setminus\{j\})\cup\{i\}$ denote the swapped treatment and control sets, respectively. 
    By definition of the regression-adjusted estimator,
    \begin{align*}
        \Delta_{ij}\,\htauols(\cS_1)
            &= \htauols(\cS_1')-\htauols(\cS_1)\\
            &= \Bigl(\, \hat\mu_1(\cS_1')-\hat\mu_1(\cS_1) \, \Bigr) - \Bigl(\, \hat\mu_0(\cS_0')-\hat\mu_0(\cS_0) \, \Bigr).
    \end{align*}
    Insert the intermediate sets $\cS_1\setminus\{i\}$ and $\cS_0\setminus\{j\}$ and telescope:
    \[
        \hat\mu_1(\cS_1')-\hat\mu_1(\cS_1)
            = \underbrace{\big(\hat\mu_1(\cS_1\setminus\{i\})-\hat\mu_1(\cS_1)\big)}_{\Delta^{\mathrm{del}}_{1}(i;\,\cS_1)}
                + 
            \underbrace{\big(\hat\mu_1\big((\cS_1\setminus\{i\})\cup\{j\}\big)-\hat\mu_1(\cS_1\setminus\{i\})\big)}_{\Delta^{\mathrm{ins}}_{1}(j;\,\cS_1\setminus\{i\})},
    \]
    and likewise,
    \[
        \hat\mu_0(\cS_0')-\hat\mu_0(\cS_0)
            = \underbrace{\big(\hat\mu_0(\cS_0\setminus\{j\})-\hat\mu_0(\cS_0)\big)}_{\Delta^{\mathrm{del}}_{0}(j;\,\cS_0)}
                + 
            \underbrace{\big(\hat\mu_0\big((\cS_0\setminus\{j\})\cup\{i\}\big)-\hat\mu_0(\cS_0\setminus\{j\})\big)}_{\Delta^{\mathrm{ins}}_{0}(i;\,\cS_0\setminus\{j\})}.
    \]
    Substituting these identities into the expression for $\Delta_{ij}\,\htauols(\cS_1)$ yields
    \[
        \Delta_{ij}\,\htauols(\cS_1)
            =\Big[\Delta^{\mathrm{del}}_{1}(i;\,\cS_1) + \Delta^{\mathrm{ins}}_{1}(j;\,\cS_1\setminus\{i\})\Big]
                -\Big[\Delta^{\mathrm{del}}_{0}(j;\,\cS_0) + \Delta^{\mathrm{ins}}_{0}(i;\,\cS_0\setminus\{j\})\Big],
    \]
    which is the equation display \eqref{eq:exact-swap-coupled}. 
\end{proof}

\subsection{Proof of Theorem \ref{thm:perturbation_identity}}\label{sec:proof_theorem_perturbation}

\subsubsection{Background: Column Append and Deletion Identities for Pseudoinverse}\label{sec:LA_primer}

Here we record the standard rank-one column‑append/deletion formulas of \citet{greville1960some}.  
We refer interested readers to a textbook by \citet[Chapters 5 and 7]{ben2003generalized} for more background and detailed expositions. 
The statements below are quoted with minor rephrasing for clarity.

\begin{lemma}[Column append identity; {\citet[Section 4]{greville1960some}, summarized and rephrased}]\label{lem:greville}
    Let $A \in \RR^{n \times m}$ and $a \in \RR^n$. 
    Set $d \coloneqq A^{\dagger} a$ and $c = a - A d = (I_n - A A^{\dagger}) a$. 
    Then
    \begin{equation}\label{eqn:column_append}
        \begin{bmatrix} A & a \end{bmatrix}^{\dagger} 
            = \begin{bmatrix} A^{\dagger} - db \\ b \end{bmatrix}
        \qquad\text{where}\qquad
        b = 
        \begin{cases}
            c^{\dagger},                                     & \text{if } c \neq 0,\\
            (1 + d^{\top} d)^{-1} d^{\top} A^{\dagger},      & \text{if } c = 0,
        \end{cases}
    \end{equation}
    where for a nonzero column vector $v$, $v^{\dagger}\equiv v^\top/(v^\top v)$ denotes its Moore--Penrose pseudoinverse.
\end{lemma}

\begin{proof}[Proof of Lemma \ref{lem:greville}]
    This statement is a summary of \citet[Section 4]{greville1960some} as a lemma, with minor rephrasing. 
    The recursive expression in \eqref{eqn:column_append} is obtained by combining Eqs. (8), (9), (11), and (16) or (22), depending on whether $c = 0$ or not.
\end{proof}

The next column deletion identity, while not explicitly stated in \citet{greville1960some}, is an immediate algebraic reverse of Lemma \ref{lem:greville}.

\begin{lemma}[Column deletion identity]\label{lem:greville_delete}
    Let $\widetilde A \coloneqq \begin{bmatrix} A & a \end{bmatrix} \in \RR^{n \times (m+1)}$ and partition $\widetilde A^{\dagger} = \begin{bmatrix} B \\ b \end{bmatrix}$ with $B \in \RR^{m \times n}$ and $b \in \RR^{1 \times n}$. 
    Then
    \begin{equation}\label{eqn:column_delete}
        A^{\dagger} = B + d b
            \qquad\text{where}\qquad
            d =
            \begin{cases}
                B \left(a - b^{\dagger}\right), & \text{if } b A = 0,\\
                (1 - ba)^{-1}\,  B a, & \text{if } b A \neq 0.
            \end{cases}
    \end{equation}
\end{lemma}

\begin{proof}[Proof of Lemma \ref{lem:greville_delete}]
    By Lemma \ref{lem:greville}, writing $d \coloneqq A^{\dagger} a$ and $c \coloneqq (I_n - A A^{\dagger}) a$, we have
    \[
        \widetilde A^{\dagger} 
        = \begin{bmatrix} A^{\dagger} - d\, b \\ b \end{bmatrix},
        \qquad\text{with}\quad
        b =
        \begin{cases}
            c^{\dagger}, & c \neq 0,\\
            (1 + d^{\top} d)^{-1} d^{\top} A^{\dagger}, & c = 0.
        \end{cases}
    \]
    Hence $B = A^{\dagger} - d\, b$, i.e., $A^{\dagger} = B + d\, b$. 
    It remains to express $d = A^{\dagger} a$ in terms of $B$, $a$, and $b$.

    \begin{itemize}
        \item
        \emph{Case 1: $bA=0$.} 
        This corresponds to the $c\neq 0$ branch in Lemma~\ref{lem:greville}, where $b=c^{\dagger}$ and thus $c=b^{\dagger}$.
        Therefore, $bA = c^{\dagger} A = 0$. 
        Using $a = A d + c$,
        \begin{align*}
            d 
                &= A^{\dagger} a
                = A^{\dagger} A d 
                = A^{\dagger} ( a - c )\\
                &\stackrel{(*)}{=} (A^{\dagger} - d b)\,(a - c)
                = B\,(a - c)\\
                &= B\,(a - b^{\dagger}).
        \end{align*}
        Here, (*) follows from the observation that since $a = Ad + c$ and $b= c^{\dagger}$, $ba = b(Ad+c) = 0+ bc = 1$, and therefore, $b(a-c) = ba - bc = 1 - 1 = 0$.

        \item 
        \emph{Case 2: $bA \neq 0$.} Multiplying $B = A^{\dagger} - d b$ by $a$ from the right:
        \[
            B a = A^{\dagger} a - d\, b a = d\,(1 - b a),
        \]
        hence $d = (1 - b a)^{-1} B a$. 
        Here, $c=0$, and Lemma \ref{lem:greville} yields $b = (1 + d^{\top} d)^{-1} d^{\top} A^{\dagger}$; so $b a = (1 + d^{\top} d)^{-1} d^{\top} d \in [0,1)$ and thus, $1 - b a = (1 + d^{\top} d)^{-1} > 0$ and the inverse $(1 - b a)^{-1}$ is well defined.        
    \end{itemize}
    
    Substituting the corresponding expression for $d$ into $A^{\dagger} = B + d b$ yields \eqref{eqn:column_delete}. 
    We note that the identities invoked are exactly those consolidated in Lemma \ref{lem:greville} from \citet[Section~4, Eqs.~(8), (9), (11), (16), (22)]{greville1960some}.
\end{proof}

Row insertion/deletion counterparts follow by transposition. 
In Section \ref{sec:lemma_deletion_insertion}, we apply these row-wise versions with $A = \Xa^{\top}$ and $a = x_i^{\top}$ (or $a = x_j^{\top}$) when deleting (or appending) a row (=unit) in $\Xa$.

\subsubsection{Row Deletion and Row Insertion Lemmas}\label{sec:lemma_deletion_insertion}

Throughout this section, we use the $Q$‑formulation of the OLS intercept $\hat{\mu}$ from Proposition~\ref{prop:mu_ratio}.
For $X \in \RR^{n \times p}$ and $y \in \RR^n$, the minimum‑norm OLS intercept is
\[
    \hat{\mu} = \frac{ \bone_{n}^{\top} Q_{X}\, y }{ \bone_{n}^{\top} Q_{X}\, \bone_{n} },
        \qquad\text{where}\qquad
        Q_X = 
            \begin{cases}
                M_X \coloneqq I_n - X X^{\dagger},  & \text{if } \bone_n \not\in \colsp(X),\\
                K_X \coloneqq (X X^{\top})^{\dagger},  & \text{if } \bone_n \in \colsp(X).
            \end{cases}
\]
Here, the denominator $\bone^\top Q_X \bone>0$ in both cases.

\paragraph{Notation.} 
For a symmetric positive semidefinite matrix $Q$, we write $\langle u,v\rangle_Q \coloneqq u^\top Q v$ and $\|u\|_Q \coloneqq \sqrt{u^\top Q u}$. 
Given $\cS \subset [n]$ and arm $\arm\in\{0,1\}$, we abbreviate
\[
    Q_\cS \coloneqq Q_{X_\cS},
    \qquad
    \hat\mu_\arm(\cS) \coloneqq \frac{\langle \bone,\; y^{(\arm)}_\cS \rangle_{Q_\cS}}{\|\bone\|_{Q_\cS}^2},
    \qquad
    \ycenta{\cS} \coloneqq y^{(\arm)}_\cS-\hat\mu_\arm(\cS)\,\cdot\, \bone,
\]
so that $\langle \bone,\,\ycenta{\cS}\rangle_{Q_\cS}=0$. 
When the dimension is clear from context, we write $\bone=\bone_n$ (or $\bone_{|\cS|}$) and drop the subscript.

\paragraph{Rank-$1$ Ratio Update for Bookkeeping.} 
We isolate a purely algebraic identity for how the intercept ratio changes under a rank–one perturbation of $Q$.
This will be applied with the row-deletion/row-insertion updates.

\begin{lemma}\label{lem:rank1_update}
    For any symmetric matrix $Q$ and vector $y$, define
    \[
        \mu_y(Q) \coloneqq 
            \begin{cases}
                \frac{\bone^\top Q y}{\bone^\top Q \bone},  & \text{if }\bone^\top Q \bone > 0,\\
                0,  &\text{if }\bone^\top Q \bone = 0.
            \end{cases},
        \qquad\text{and}\qquad
        \ycent{Q} \coloneqq y - \mu_y(Q) \cdot \bone.
    \]
    For any scalar $\kappa$ and vector $v$, let $Q'= Q + \kappa (Q v) (Q v)^{\top}$. 
    Then
    \begin{equation}\label{eq:ratio-rank1}
        \mu_y(Q')-\mu_y(Q)
            = \frac{\ \kappa \langle \bone,v\rangle_{Q}\ \langle v,\, \ycent{Q}\rangle_{Q}\ }{\ \bone^\top Q' \bone\ }.
    \end{equation}
\end{lemma}

\begin{proof}[Proof of Lemma \ref{lem:rank1_update}]
    Observe that
    \[
        \bone^\top Q' y
            = \bone^\top Q y + \kappa \langle \bone,v\rangle_{Q}\,\langle v,y\rangle_{Q},
        \qquad\text{and}\qquad
        \bone^\top Q' \bone  
            = \bone^\top Q \bone + \kappa \langle \bone,v\rangle_{Q}^2.
    \]
    Therefore,
    \begin{align*}
        \mu_y(Q')-\mu_y(Q)
            &= \frac{\bone^\top Q y + \kappa \langle \bone, v\rangle_{Q}\langle v,y\rangle_{Q}}{\bone^\top Q \bone + \kappa \langle \bone,v\rangle_{Q}^{2}} - \frac{\bone^\top Q y}{\bone^\top Q \bone} \\
            &= \frac{\ \kappa \langle \bone,v\rangle_{Q}\ \bigl(\langle v,y\rangle_{Q} - \mu_y(Q) \langle v,\bone\rangle_{Q}\bigr)\ }{\ \bone^\top Q' \bone\ } \\
            &= \frac{\ \kappa \langle \bone,v\rangle_{Q}\ \bigl\langle v, y- \mu_y(Q) \cdot \bone\bigr\rangle_{Q}\ }{\ \bone^\top Q' \bone\ } \\
            &= \frac{\ \kappa \langle \bone,v\rangle_{Q}\ \langle v,\, \ycent{Q} \rangle_{Q}\ }{\ \bone^\top Q' \bone\ }. 
    \end{align*}
\end{proof}

\begin{remark}[Useful variant without $Q$-geometry]
    The proof of Lemma \ref{lem:rank1_update} goes through verbatim if $Q'$ is updated by a general rank‑one form $Q' = Q + \kappa\,rr^\top$ (not necessarily $r=Qv$), but with $\langle \bone,v\rangle_Q$ and $\langle v,\cdot\rangle_Q$ replaced by $\,\bone^\top r$ and $r^\top(\cdot)\,$.
\end{remark}


\paragraph{Row Deletion: Exact Identity.} 
We apply Lemma~\ref{lem:rank1_update} to the Greville deletion (Lemma \ref{lem:greville_delete} after transpose) to obtain a closed form for the atomic deletion change of the intercept. 

Given $\cS$ with $s = |\cS|$, let $e_i\in\RR^s$ denote the coordinate vector of unit $i$ within $\cS$ (i.e., the indicator of $i$’s position in $\cS$). 
For $i \in \cS$, define
\begin{equation}\label{eqn:deletion_denominator}
    \Phidel_i(\cS) \coloneqq \|\bone\|_{Q_\cS}^2\,\|e_i\|_{Q_\cS}^2-\langle e_i,\bone\rangle_{Q_\cS}^2 \geq 0.
\end{equation}

\begin{lemma}[Row deletion identity]\label{lem:deletion}
    Let \(X\in\RR^{n\times p}\), \(y^{(1)},y^{(0)}\in\RR^n\), \(\arm\in\{0,1\}\), \(\cS\in\binom{[n]}{m}\), and \(i\in\cS\). 
    Let \(R_i\in\{0,1\}^{(m-1)\times m}\) be the row-deletion matrix that removes the coordinate of \(i\) from vectors indexed by \(\cS\). 
    Let \(Q_\cS=Q_{X_\cS}\), \(\hat\mu_\arm=\hat\mu_\arm(\cS)\), and
    \[
        \ycenta{\cS}=y_\cS^{(\arm)}-\hat\mu_\arm\bone .
    \]
    Then the deletion change satisfies the exact ratio identity
    \begin{equation}\label{eq:exact-del-ratio}
        \Delta^{\mathrm{del}}_\arm(i;\cS)
        =
        \frac{
            \bone^\top Q_{X_{\cS\setminus\{i\}}}
            R_i\ycenta{\cS}
        }{
            \bone^\top Q_{X_{\cS\setminus\{i\}}}\bone
        }.
    \end{equation}
    Moreover, in the regular deletion case where \(q_i\coloneqq e_i^\top Q_\cS e_i>0\) and the zero-padded update satisfies
    \[
        R_i^\top Q_{X_{\cS\setminus\{i\}}}R_i
        =
        Q_\cS-\frac{Q_\cS e_ie_i^\top Q_\cS}{q_i},
    \]
    if \(\Phidel_i(\cS)>0\), then
    \begin{equation}\label{eq:exact-del-base}
        \Delta^{\mathrm{del}}_\arm(i;\cS)
        =
        -\frac{
            \langle e_i,\bone\rangle_{Q_\cS}
            \langle e_i,\ycenta{\cS}\rangle_{Q_\cS}
        }{
            \Phidel_i(\cS)
        }.
    \end{equation}
\end{lemma}

\begin{proof}[Proof of Lemma~\ref{lem:deletion}]
    The identity \eqref{eq:exact-del-ratio} follows directly from the definition:
    \[
        \hat\mu_\arm(\cS\setminus\{i\})
        =
        \frac{
            \bone^\top Q_{X_{\cS\setminus\{i\}}}y_{\cS\setminus\{i\}}^{(\arm)}
        }{
            \bone^\top Q_{X_{\cS\setminus\{i\}}}\bone
        }.
    \]
    Subtracting \(\hat\mu_\arm(\cS)\) from both sides gives \eqref{eq:exact-del-ratio}.

    Next, we prove the regular closed form \eqref{eq:exact-del-base}. 
    Let \(Q=Q_\cS\), \(q_i=e_i^\top Qe_i\), \(a_i=e_i^\top Q\bone\), and \(b_i=e_i^\top Q\ycenta{\cS}\). 
    By the assumed regular deletion update,
    \[
        Q_-^\uparrow
        :=
        R_i^\top Q_{X_{\cS\setminus\{i\}}}R_i
        =
        Q-\frac{Qe_ie_i^\top Q}{q_i}.
    \]
    Since \(\bone^\top Q\ycenta{\cS}=0\), the numerator in \eqref{eq:exact-del-ratio}, written in the original \(s\)-dimensional coordinates, equals
    \[
        \bone^\top Q_-^\uparrow \ycenta{\cS}
        =
        -\frac{a_i b_i}{q_i}.
    \]
    The denominator equals
    \[
        \bone^\top Q_-^\uparrow \bone
        =
        \bone^\top Q\bone-\frac{a_i^2}{q_i}
        =
        \frac{
            \|\bone\|_{Q}^2\|e_i\|_{Q}^2-\langle e_i,\bone\rangle_Q^2
        }{
            q_i
        }
        =
        \frac{\Phidel_i(\cS)}{q_i}.
    \]
    Dividing the two displays yields \eqref{eq:exact-del-base}, which completes the proof.
\end{proof}

\begin{remark}[Explicit appearance of $y^{(\arm)}_i$]
    Intuitively, we expect $\Delta^{\mathrm{del}}_\arm(i; \cS)$ to depend on the covariate $x_i$ and the outcome $y^{(\arm)}_i$ at the deleted unit $i \in \cS$. 
    The factor 
    \[
        \langle e_i,\;y^{(\arm),c}_\cS\rangle_{Q_\cS} 
            = e_i^\top Q_\cS\big(y^{(\arm)}_\cS-\hat\mu_\arm(\cS)\bone\big)
            = (Q_\cS y^{(\arm)}_\cS)_i-\hat\mu_\arm(\cS)\,(Q_\cS\bone)_i
    \]
    enters \eqref{eq:exact-del-base} and depends \emph{explicitly} on $y^{(\arm)}_i$ (together with the neighboring entries weighted by $Q_\cS$). 
    Thus the deletion sensitivity reflects both the covariate geometry (via $Q_\cS$ and $\langle e_i,\bone\rangle_{Q_\cS}$) and the outcome at unit $i$.
\end{remark}

\paragraph{Row Insertion: Exact Identity.}
Next, we combine the Greville column-append identity (Lemma \ref{lem:greville}, transposed to rows) with the bookkeeping lemma (Lemma \ref{lem:rank1_update}) to derive the insertion change.
This makes the dependence on $(x_j, y^{(\arm)}_j)$ explicit and holds uniformly across the two $Q$-branches in Proposition~\ref{prop:mu_ratio}. 

\begin{lemma}[Row insertion identity]\label{lem:insertion}
    Let \(X\in\RR^{n\times p}\), \(y^{(1)},y^{(0)}\in\RR^n\), \(\cS\in\binom{[n]}{m}\), \(\arm\in\{0,1\}\), and \(j\notin\cS\). 
    Let \(\cS^+=\cS\cup\{j\}\), and order \(\cS^+\) so that \(j\) is the last row. Partition
    \[
        Q_{X_{\cS^+}}
        =
        \begin{bmatrix}
            Q_{00}^{+} & q_{0j}^{+}\\
            (q_{0j}^{+})^\top & q_{jj}^{+}
        \end{bmatrix}.
    \]
    With
    \[
        r_j^{(\arm)}\coloneqq y_j^{(\arm)}-\hat\mu_\arm(\cS),
        \qquad
        \ycenta{\cS}=y_\cS^{(\arm)}-\hat\mu_\arm(\cS)\bone,
    \]
    the atomic insertion change satisfies the exact block-ratio formula
    \begin{equation}\label{eq:exact-ins-block}
        \Delta^{\mathrm{ins}}_\arm(j;\cS)
        =
        \frac{
            \bone^\top Q_{00}^{+}\ycenta{\cS}
            +(q_{0j}^{+})^\top\ycenta{\cS}
            +\bigl(\bone^\top q_{0j}^{+}+q_{jj}^{+}\bigr)r_j^{(\arm)}
        }{
            \bone^\top Q_{00}^{+}\bone
            +2\,\bone^\top q_{0j}^{+}
            +q_{jj}^{+}
        }.
    \end{equation}

    In addition, the following simplifications hold.

    \begin{enumerate}
        \item 
        (Regular \(K\)-branch case) Suppose \(Q_{X_\cS}=K_{X_\cS}\), and define
        \[
            \sigma_j^2(\cS)
            =
            \|x_j\|_2^2-\|X_\cS x_j\|_{Q_\cS}^2
            >0.
        \]
        If \(\sigma_j^2(\cS) > 0\), then
        \begin{equation}\label{eq:exact-ins-K}
            \Delta^{\mathrm{ins}}_\arm(j;\cS)
            =
            \frac{
                \bigl(1-\alpha_j(\cS)\bigr)
                \Bigl(
                y_j^{(\arm)}-\hat\mu_\arm(\cS)
                -
                \langle X_\cS x_j,\ycenta{\cS}\rangle_{Q_\cS}
                \Bigr)
            }{
                \|\bone\|_{Q_\cS}^2\sigma_j^2(\cS)
                +
                \bigl(1-\alpha_j(\cS)\bigr)^2
            },
        \end{equation}
        where \(\alpha_j(\cS)=\langle \bone,X_\cS x_j\rangle_{Q_\cS}\).

        \item 
        (Regular \(M\)-branch case) Suppose \(Q_{X_\cS}=M_{X_\cS}\), and let
        \[
            A_\cS:=X_\cS^\top,\qquad
            d_j:=A_\cS^\dagger x_j,\qquad
            c_j:=x_j-A_\cS d_j.
        \]
        \begin{enumerate}
            \item 
            If \(c_j\ne0\), then \(\Delta^{\mathrm{ins}}_\arm(j;\cS)=0\). 

            \item 
            If \(c_j=0\), then
            \begin{equation}\label{eq:exact-ins-M}
                \Delta^{\mathrm{ins}}_\arm(j;\cS)
                =
                \frac{
                    (1-\eta_j)
                    \bigl(
                    y_j^{(\arm)}-\hat\mu_\arm(\cS)-\beta_j
                    \bigr)
                }{
                    \|\bone\|_{Q_\cS}^2(1+h_j)
                    +
                    (1-\eta_j)^2
                },
            \end{equation}
            where
            \[
                h_j=\|d_j\|_2^2,\qquad
                \eta_j=\bone^\top d_j,\qquad
                \beta_j=d_j^\top\ycenta{\cS}.
            \]
        \end{enumerate}
    \end{enumerate}
\end{lemma}

\begin{proof}[Proof of Lemma~\ref{lem:insertion}]
    The block formula \eqref{eq:exact-ins-block} follows directly from the ratio representation: if we write \(Q^+=Q_{X_{\cS^+}}\) and \(y^+=(y_\cS^{(\arm)},y_j^{(\arm)})\), then
    \[
        \hat\mu_\arm(\cS^+)-\hat\mu_\arm(\cS)
        =
        \frac{[\bone;1]^\top Q^+\bigl[y_\cS^{(\arm)}-\hat\mu_\arm(\cS)\bone;\,
        y_j^{(\arm)}-\hat\mu_\arm(\cS)\bigr]}
        {[\bone;1]^\top Q^+[\bone;1]}.
    \]
    Expanding this expression under the displayed block partition of \(Q^+\) gives \eqref{eq:exact-ins-block}.

    Next, we derive the regular \(K\)-branch formula. 
    Let \(u_j=X_\cS x_j\), \(Q=K_{X_\cS}\), and \(\sigma^2=\|x_j\|_2^2-u_j^\top Q u_j>0\). 
    By the row-append Moore--Penrose update, equivalently the transposed \(c\neq0\) branch of Greville's identity in Lemma~\ref{lem:greville}, the appended Gram pseudoinverse satisfies
    \[
        Q_{X_{\cS^+}}
        =
        \begin{bmatrix}
            Q+\dfrac{Q u_j u_j^\top Q}{\sigma^2}
            &
            -\dfrac{Q u_j}{\sigma^2}\\
            -\dfrac{u_j^\top Q}{\sigma^2}
            &
            \dfrac{1}{\sigma^2}
        \end{bmatrix}.
    \]
    When \(X_\cS X_\cS^\top\) is invertible this reduces to the usual Schur-complement inverse update.
    Substituting this block matrix into \eqref{eq:exact-ins-block}, and using \(\bone^\top Q\ycenta{\cS}=0\), yields \eqref{eq:exact-ins-K}.

    Lastly, consider the regular \(M\)-branch. Let \(Q=M_{X_\cS}\), \(A_\cS=X_\cS^\top\), \(d_j=A_\cS^\dagger x_j\), and \(c_j=x_j-A_\cS d_j\). 
    Greville's row-append formula for the residual-maker matrix gives two cases. 
    If \(c_j\ne0\), then
    \[
        Q_{X_{\cS^+}}
        =
        \begin{bmatrix}
            Q&0\\
            0&0
        \end{bmatrix},
    \]
    so \eqref{eq:exact-ins-block} gives \(\Delta^{\mathrm{ins}}_\arm(j;\cS)=0\). 
    If \(c_j=0\), then
    \[
        Q_{X_{\cS^+}}
        =
        \begin{bmatrix}
            Q+\dfrac{d_jd_j^\top}{1+\|d_j\|_2^2}
            &
            -\dfrac{d_j}{1+\|d_j\|_2^2}\\
            -\dfrac{d_j^\top}{1+\|d_j\|_2^2}
            &
            \dfrac{1}{1+\|d_j\|_2^2}
        \end{bmatrix}.
    \]
    Substituting this expression into \eqref{eq:exact-ins-block} yields \eqref{eq:exact-ins-M}. 
    In particular, when \(X_\cS\) has no covariates, \(d_j=0\), and \eqref{eq:exact-ins-M} reduces to
    \[
        \Delta^{\mathrm{ins}}_\arm(j;\cS)
        =
        \frac{y_j^{(\arm)}-\hat\mu_\arm(\cS)}
        {|\cS|+1},
    \]
    as required for insertion into a sample mean.
\end{proof}

\subsubsection{Completing Proof of Theorem \ref{thm:perturbation_identity}}\label{sec:completing_proof_perturbation}

\begin{proof}[Proof of Theorem \ref{thm:perturbation_identity}]
    The regular deletion identity in Theorem~\ref{thm:perturbation_identity} is exactly \eqref{eq:exact-del-base} from Lemma~\ref{lem:deletion}. 
    The regular \(K\)-branch insertion identity is exactly \eqref{eq:exact-ins-K} from Lemma~\ref{lem:insertion}. 
    The full exact branch-specific identities are given by \eqref{eq:exact-del-ratio}, \eqref{eq:exact-ins-block}, \eqref{eq:exact-ins-K}, and \eqref{eq:exact-ins-M}. 
    Combining these atomic identities with the paired deletion--insertion telescope in Proposition~\ref{prop:sensitivity_decomposition} gives the oracle one-swap sensitivity of OLS--RA.
\end{proof}

\section{NUMERICAL EXPERIMENTS}\label{sec:deferred_experiments}

This appendix complements the theory by describing \emph{(i)} Monte Carlo (MC) approximation of the oracle quantities \((\Vstar,\Rstar,\Bstar)\) (Section~\ref{sec:experiments_monte_carlo}) and \emph{(ii)} small‑scale simulations that probe the sharpness of the finite‑sample CIs in Theorem~\ref{thm:oracle} across classical and over‑parameterized regimes (Section~\ref{sec:experiments_main}). 
The Python code used to generate the Monte Carlo results reported here is available at \url{https://github.com/dogyoons/regadj_finitesample}.

\subsection{Monte Carlo Approximation of Oracle Parameters}\label{sec:experiments_monte_carlo}

This section explains how the oracle quantities \((V^\star,R^\star,B^\star)\) entering Theorem~\ref{thm:oracle} are computed in the experiments. 
Throughout, we work in the \emph{oracle} setting, meaning that the full finite-population objects \((X,y^{(0)},y^{(1)})\) are treated as known. 
The data-driven proxies discussed in Appendices~\ref{sec:geometric_envelope}--\ref{sec:leverage_bounds} are not used here.

Given an ATE estimator \(\hat\tau\) and a treatment set \(\cS\), write \(f(\cS)=\hat\tau(\cS)-\tau\). 
The finite-sample radius in Theorem~\ref{thm:oracle} depends on three population quantities:
\[
    \Vstar=\sum_{t=1}^{n_1} v_t^\star,
    \qquad
    \Rstar=\max_t r_t^\star,
    \qquad
    \Bstar=\frac{\sqrt{\bbE(Lf+\lambda^\star f)^2}}{\lambda^\star},
\]
where \(v_t^\star=\Var(D_t\mid\cF_{t-1})\), \(r_t^\star=\sup|D_t|\), and \(D_t=M_t-M_{t-1}\) is the Doob-martingale increment along the assignment-exposure filtration. 
By Proposition~\ref{prop:reveal-swap}, these quantities are determined by one-swap sensitivities \(\Delta_{ij}f\) and their conditional averages.

For OLS--RA, exact evaluation of these expectations is computationally infeasible in general, because it would require summing over many treatment assignments, reveal states, and ordered swap pairs. 
Accordingly, in Experiment~2 we approximate \((V^\star,R^\star,B^\star)\) by Monte Carlo estimators \((\widehat V^\star,\widehat R^\star,\widehat B^\star)\). 
For DiM, by contrast, \(V^\star\) and \(R^\star\) admit closed forms under the remaining-pool reveal law, so we use these exact oracle quantities in Experiment~1. 
Throughout, OLS--RA swap effects are evaluated using the paired deletion--insertion identities from Proposition~\ref{prop:sensitivity_decomposition} and Theorem~\ref{thm:perturbation_identity}, so no OLS re-fitting is required.

\subsubsection{Monte Carlo Estimation of Concentration Parameters}\label{sec:MC_concentration}

\paragraph{Recollecting Setup and Notation.}
Fix a realized assignment \((\cS_1,\cS_0)\) and an independent reveal order \(\Pi\) of the treated units. 
Recall from \eqref{eqn:past_remaining_sets} that at step \(t\), we let
\[
    \cS^{\mathrm{past}}_{t-1}=\{\Pi(1),\ldots,\Pi(t-1)\},
    \qquad
    \cR_{t-1}=[n]\setminus \cS^{\mathrm{past}}_{t-1},
\]
and let \(\nass=n_1-(t-1)\). 
For a candidate unit \(i\in\cR_{t-1}\), the proxy treatment set is defined in \eqref{eqn:prox_set} as
\[
    \Sprox{t-1}(i,\cT)=\cS^{\mathrm{past}}_{t-1}\cup\{i\}\cup\cT,
    \qquad
    \cT\sim\Unif\binom{\cR_{t-1}\setminus\{i\}}{\nass-1}.
\]
In the implementation, \(\cI_t\subseteq \cR_{t-1}\) is either the full remaining pool or a uniform subsample of size \(B_i\). 
For each \(i\in\cI_t\), we draw \(B_{\mathrm{cond}}\) completion sets \(\cT_{b,i}\), and for each completion either enumerate all admissible controls or sample \(B_J\) admissible controls uniformly without replacement. 
Here \(B_i\), \(B_{\mathrm{cond}}\), and \(B_J\) are the MC budgets controlling the candidate set, completion draws, and control draws, respectively.

\paragraph{Stepwise Monte Carlo Estimate.}
For each \(i\in\cI_t\), define
\begin{equation}\label{eqn:zeta_hat_i}
    \widehat\zeta_t(i)
    =
    \frac{1}{B_{\mathrm{cond}}}
    \sum_{b=1}^{B_{\mathrm{cond}}}
    \left[
        \frac{1}{|\cJ_{b,i}|}
        \sum_{J\in\cJ_{b,i}}
        \Delta_{iJ} f\big(\Sprox{t-1}(i,\cT_{b,i})\big)
    \right],
\end{equation}
where \(\cJ_{b,i}\subseteq \cR_{t-1}\setminus(\{i\}\cup\cT_{b,i})\) is either the full admissible control set or a uniform subsample of size \(B_J\). 
We also define the within-\(i\) completion variance
\[
    \widehat s_t^2(i)
    :=
    \widehat{\Var}_{b=1,\ldots,B_{\mathrm{cond}}}
    \left(
        \frac{1}{|\cJ_{b,i}|}
        \sum_{J\in\cJ_{b,i}}
        \Delta_{iJ} f\big(\Sprox{t-1}(i,\cT_{b,i})\big)
    \right).
\]
The reported experiments use the direct plug-in estimators
\[
    \widehat{\Vstar}
    =
    \sum_{t=1}^{n_1}\alpha_t^2
    \left[
        \widehat{\Var}_{i\in\cI_t}\big(\widehat\zeta_t(i)\big)
        -
        \frac{1}{|\cI_t|}\sum_{i\in\cI_t}\frac{\widehat s_t^2(i)}{B_{\mathrm{cond}}}
    \right]_+,
    \qquad
    \widehat{\Rstar}
    =
    \max_{1\le t\le n_1}\alpha_t\max_{i\in\cI_t}|\widehat\zeta_t(i)|.
\]
Here, \([z]_+ = \max\{z, 0\}\). 
The optional one-sided upper-confidence bound correction for \(R^\star\) is implemented in the code but is not used in the reported tables.

\paragraph{Algorithmic Pseudocode.} 
The entire procedure is summarized in Algorithm \ref{alg:mc-varrange} in a pseudocode format.

\begin{algorithm}[t]
    \caption{\textsc{MCVarRange}$(X,y^{(0)},y^{(1)},\Sone; \Pi, B_i, B_{\mathrm{cond}}, B_J)$ --- compute \((\widehat V^\star,\widehat R^\star)\)}
    \label{alg:mc-varrange}
    \begin{algorithmic}[1]
        \Require $X$, $y^{(0)}$, $y^{(1)}$, realized $\Sone$; reveal permutation $\Pi$; budgets $B_i,B_{\mathrm{cond}},B_J$
        \Ensure $\widehat \Vstar$, $\widehat \Rstar$
        \State $\widehat{\Vstar} \gets 0$;\ $\widehat{\Rstar} \gets 0$
        \For{$t=1$ \textbf{to} $n_1$}
            \State $\cS^{\mathrm{past}}_{t-1}\gets\{\Pi(1),\ldots,\Pi(t-1)\}$;\quad 
                   $\cR_{t-1}\gets [n]\setminus\cS^{\mathrm{past}}_{t-1}$
            \If{$|\cR_{t-1}|\le B_i$}
                \State $\cI_t\gets \cR_{t-1}$
            \Else
                \State $\cI_t\gets $ a uniform subsample of size $B_i$ from $\cR_{t-1}$
            \EndIf
            \State Initialize $\widehat\zeta_t(i)\gets 0$ for all $i\in \cI_t$
            \For{$i \in \cI_t$}
              \For{$b=1$ \textbf{to} $B_{\mathrm{cond}}$}
                \State Draw $\cT_{b,i}\sim \Unif\binom{\cR_{t-1}\setminus\{i\}}{\nass-1}$
                \If{$B_J=0$ \textbf{or} $|\cR_{t-1}\setminus(\{i\}\cup\cT_{b,i})|\le B_J$}
                    \State \(\cJ_{b,i} \gets \cR_{t-1}\setminus(\{i\}\cup\cT_{b,i})\)
                \Else
                    \State \(\cJ_{b,i} \gets \) a uniform subsample of size $B_J$ from $\cR_{t-1}\setminus(\{i\}\cup\cT_{b,i})$
                \EndIf
                \State Compute $\Delta_{iJ} f\big(\Sprox{t-1}(i,\cT_{b,i})\big)$ for all \(J\in\cJ_{b,i}\)
                \State $\widehat\zeta_t(i)\gets \widehat\zeta_t(i) + \frac{1}{B_{\mathrm{cond}}}\cdot
                   \frac{1}{|\cJ_{b,i}|}\sum_{J\in \cJ_{b,i}}
                   \Delta_{iJ} f\big(\Sprox{t-1}(i,\cT_{b,i})\big)$  
              \EndFor
            \EndFor
            \State Compute \(\widehat s_t^2(i)\) from the \(B_{\mathrm{cond}}\) completion draws
            \State 
            $\widehat \Vstar\ \gets\ \widehat \Vstar
                + \alpha_t^2\left[
                \EmpVar\{\,\widehat\zeta_t(i): i\in \cI_t\,\}
                -
                \frac{1}{|\cI_t|}\sum_{i\in\cI_t}\frac{\widehat s_t^2(i)}{B_{\mathrm{cond}}}
                \right]_+$
            \State $\widehat \Rstar \gets \max\big(\widehat \Rstar,\ \alpha_t \max_{i\in\cI_t} |\widehat\zeta_t(i)|\big)$
        \EndFor
        \State \Return $(\widehat \Vstar,\ \widehat \Rstar)$
    \end{algorithmic}
\end{algorithm}

\subsubsection{Monte Carlo Estimation of Bias Parameter} \label{sec:MC_bias}
Choose budget parameters $B_S \in \ZZp$. 
\begin{enumerate}
    \item 
    Draw i.i.d. treatment sets $\cS_1^{(b)} \sim \Unif\binom{[n]}{n_1}$ for $b \in [B_S]$, and set $\cS_0^{(b)} = [n] \setminus \cS_1^{(b)}$.

    \item 
    For each $b \in [B_S]$, compute
    \begin{align}
        \widehat{\Gamma}_b
            &=
            \frac{1}{2n_1n_0}
            \sum_{i\in\cS_1^{(b)}}\sum_{j\in\cS_0^{(b)}}
            \bigl(\Delta_{ij}f(\cS_1^{(b)})\bigr)^2,
            \label{eqn:Gamma_hat_b}\\
        \widehat{Lf}_b
            &=
            \frac{1}{n_1n_0}
            \sum_{i\in\cS_1^{(b)}}\sum_{j\in\cS_0^{(b)}}
            \Delta_{ij}f(\cS_1^{(b)}),
            \label{eqn:Lf_hat_b}\\
        f_b
            &=
            f(\cS_1^{(b)})
            =
            \hat\tau(\cS_1^{(b)})-\tau .
            \notag
    \end{align}

    \item 
    Then set
    \[
        \widehat{\bbE\Gamma(f)}
            =
            \frac{1}{B_S}\sum_{b=1}^{B_S}\widehat{\Gamma}_b,
        \qquad
        \widehat{\Var}(f)
            =
            \EmpVar\bigl(\{f_b\}_{b\le B_S}\bigr),
    \]
    and
    \[
        \widehat{\lambda^\star}
            =
            \max\left\{
                \mathrm{gap}_{n,n_1},
                \frac{\widehat{\bbE\Gamma(f)}}{\widehat{\Var}(f)}
            \right\}.
    \]
    Here, \(\widehat{\Var}(f)\) is the sample variance of \(f(\cS_1^{(b)})=\hat\tau(\cS_1^{(b)})-\tau\) over \(b\le B_S\), and \(\mathrm{gap}_{n,n_1}= \tfrac{n}{n_1n_0}\).

    \item 
    Finally, estimate the corrected oracle bias parameter by
    \[
        \widehat\Bstar
            =
            \frac{
                \left[
                B_S^{-1}\sum_{b=1}^{B_S}
                \bigl(\widehat{Lf}_b+\widehat{\lambda^\star} f_b\bigr)^2
                \right]^{1/2}
            }{\widehat{\lambda^\star}}.
    \]
\end{enumerate}

If \(|\cS_1^{(b)}\times \cS_0^{(b)}|\) is too large, we subsample \(B_{\mathrm{pair}}\) ordered pairs \((i,j)\) uniformly without replacement to approximate \(\widehat{\Gamma}_b\) and \(\widehat{Lf}_b\).
See Algorithm \ref{alg:mc-bias} for a summary of the procedure in pseudocode format.

\begin{algorithm}[t]
    \caption{\textsc{MCBias}$(X,y^{(0)},y^{(1)},n_1; B_S, B_{\mathrm{pair}})$ — estimate \(B^\star\) for \(f(\cS)=\hat\tau_{\OLS}(\cS)-\tau\)}
    \label{alg:mc-bias}
    \begin{algorithmic}[1]
        \Require $X$, $y^{(0)}$, $y^{(1)}$, treated size $n_1$; budgets $B_S, B_{\mathrm{pair}}$
        \Ensure $\widehat{\bbE\Gamma(f)}$, $\widehat{\lambda^\star}$, $\widehat \Bstar$
        \For{$b=1$ \textbf{to} $B_S$}
          \State Draw $\cS^{(b)}_1\sim\Unif\binom{[n]}{n_1}$; set $\cS^{(b)}_0=[n]\setminus \cS^{(b)}_1$
          \State $f_b \gets \hat\tau(\cS^{(b)}_1)-\tau$
          \If{$B_{\mathrm{pair}} \ge n_1 n_0$}
             \State $\widehat\Gamma_b \gets \frac{1}{2 n_1 n_0}
                    \sum_{i\in \cS^{(b)}_1}\sum_{j\in \cS^{(b)}_0}
                    \big(\Delta_{ij} f(\cS^{(b)}_1)\big)^2$
             \State $\widehat{Lf}_b \gets \frac{1}{n_1 n_0}
                    \sum_{i\in \cS^{(b)}_1}\sum_{j\in \cS^{(b)}_0}
                    \Delta_{ij} f(\cS^{(b)}_1)$
          \Else
             \State Draw \(B_{\mathrm{pair}}\) ordered pairs \((i_k,j_k)\) uniformly at random from \(\cS^{(b)}_1\times \cS^{(b)}_0\) without replacement
             \State $\widehat\Gamma_b \gets \frac{1}{2\,B_{\mathrm{pair}}}\sum_{k=1}^{B_{\mathrm{pair}}}
                    \big(\Delta_{i_k j_k} f(\cS^{(b)}_1)\big)^2$
             \State $\widehat{Lf}_b \gets \frac{1}{B_{\mathrm{pair}}}\sum_{k=1}^{B_{\mathrm{pair}}}
                    \Delta_{i_k j_k} f(\cS^{(b)}_1)$
          \EndIf
        \EndFor
        \State $\widehat{\bbE\Gamma(f)} \gets \frac{1}{B_S}\sum_{b=1}^{B_S}\widehat\Gamma_b$,\quad
               $\widehat{\Var}(f) \gets \EmpVar \big(\{f_b\}_{b\le B_S}\big)$
        \State $\widehat{\lambda^\star} \gets \max \left\{\,\mathrm{gap}_{n,n_1},\ \tfrac{\widehat{\bbE\Gamma(f)} }{ \widehat{\Var}(f)} \right\}$
        \State $\widehat \Bstar \gets
            \frac{
            \left[
            B_S^{-1}\sum_{b=1}^{B_S}
            \bigl(\widehat{Lf}_b+\widehat{\lambda^\star} f_b\bigr)^2
            \right]^{1/2}
            }{\widehat{\lambda^\star}}$
        \State \Return $\big(\widehat{\bbE\Gamma(f)}, \, \widehat{\lambda^\star}, \, \widehat \Bstar\big)$
    \end{algorithmic}
\end{algorithm}

\subsubsection{Computational Complexity and Monte Carlo Accuracy}\label{sec:MC_discussion}

For OLS--RA, all swap effects are evaluated via the paired deletion--insertion telescope and the branch-aware row-update identities, so no OLS model re-fitting is required. 
If \(r_{S_\alpha}=\mathrm{rank}(X_{S_\alpha})\), the one-time arm-wise preprocessing cost is 
\(
    O \left(\sum_{\alpha\in\{0,1\}}\bigl(|S_\alpha|\,p\,r_{S_\alpha}+r_{S_\alpha}^3\bigr)\right).
\)

Thereafter, the dominant cost is the number of sampled swap evaluations. 
In the sampled-\(J\) implementation, Algorithm~\ref{alg:mc-varrange} and Algorithm~\ref{alg:mc-bias} respectively perform
\[
    N^{\mathrm{var}}_{\Delta}
    =
    \sum_{t=1}^{n_1} |\cI_t|\,B_{\mathrm{cond}}\,B_J^{\mathrm{eff}},
    \qquad
    N^{\mathrm{bias}}_{\Delta}
    =
    B_S \min(B_{\mathrm{pair}},n_1n_0)
\]
swap evaluations, where \(B_J^{\mathrm{eff}}=n_0\) if \(B_J=0\) and \(B_J^{\mathrm{eff}}=B_J\) otherwise. 
Thus \(B_i\), \(B_{\mathrm{cond}}\), and \(B_J\) control the resolution and MC noise of \((\widehat V^\star,\widehat R^\star)\), while \(B_S\) and \(B_{\mathrm{pair}}\) control the stability of \(\widehat B^\star\). 
The reported runs use moderate budgets to balance computational cost and Monte Carlo variability.

All in all, the total arithmetic count is
\[
    O \left(\sum_{\arm\in\{0,1\}}\big(|S_\arm|\,p\,r_{S_\arm}+r_{S_\arm}^3\big)\right)
        +
        O \Bigl((N^{\mathrm{var}}_{\Delta}+N^{\mathrm{bias}}_{\Delta})\cdot c_{\mathrm{swap}}\Bigr),
\]
where $c_{\mathrm{swap}}=O(r_{S_\arm})$ for deletions and $c_{\mathrm{swap}}\in\{O(r_{S_\arm}p),\,O(r_{S_\arm}|S_\arm|)\}$ for insertions, depending on whether the cross-terms are precomputed or not. 
The loops over $t$, $i\in\cI_t$, and $(i,j)$ pairs are parallelizable. 
With pre-computation, the memory requirement for the cached factors is $O(|S_\arm|r_{S_\arm})$.

\subsection{Simulation Studies}\label{sec:experiments_main}

\subsubsection{Setup}\label{sec:exp_setup}
\paragraph{Data Generation Process.}
We choose \( n \in \ZZp \) and let \(\rho = 0.3\); \(n_1= \textrm{round}(\rho n) \); \(n_0=n-n_1\). 
For \(\gamma\in\{0,0.25,0.5,0.75,1.0,1.25,1.5\}\), we set \(p=\lceil n^\gamma\rceil\) and form \(X\in\RR^{n\times p}\) in a two-stage process as follows.
First of all, we generate a single ``master'' matrix \(\widetilde X\in\RR^{n\times \lceil n^{1.5} \rceil}\) with i.i.d.\ entries from the standard Gaussian distribution $\cN(0,1)$, and hold it fixed for controlled comparison of covariate dimensionality \(p\) across all conditions.
Thereafter, for each \(p\), we create \(X\) by subsampling the first \(p\) columns of \(\widetilde X\); then column‑center and rescale so that \(\bone^\top X=0\) and each column of \(X\) has norm \(\sqrt{n}\). 

With $X$ instantiated, we generate potential outcomes \(y^{(1)} = X\beta_1^\ast + \varepsilon^{(1)}\) and \(y^{(0)} = X\beta_0^\ast + \varepsilon^{(0)}\).  
In the classical regime where \(\Sigma_{\arm}\) is invertible, arm-wise OLS with intercept is invariant to adding a common linear signal \(X\beta\) to both arms: the fitted slope in each arm shifts by \(\beta\), while the fitted intercept, and hence \(\hat\tau-\tau\), remains unchanged (cf.\ \eqref{eq:beta-sample-minus-pop} in Section~\ref{sec:reg_adjustment}). 
In overparameterized regimes, one may analogously use the minimum-\(\ell_2\)-norm OLS fit, replacing \(\Sigma_{\arm}^{-1}\) by \(\Sigma_{\arm}^{\dagger}\) \citep{shen2023algebraic}; we do not rely on an exact invariance statement there though. 
Thus, in our experiments, we set \(\beta_1^\ast=\beta_0^\ast=0\) for simplicity and take \(\varepsilon^{(1)}\) and \(\varepsilon^{(0)}\) to have i.i.d.\ \(\cN(0,1)\) entries. 
Preliminary runs with a common signal gave similar qualitative conclusions; see Remark~\ref{remark:model_linear} below.

We keep $(X, y^{(1)}, y^{(0)})$ fixed while sampling treatment assignments (complete randomization) and compare the OLS--RA estimator to the DiM baseline and empirical measures for diagnostics. 
For each realized instance \((X,y^{(0)},y^{(1)})\), we draw \(N\) treatment assignments under complete randomization and compute all statistics assignment‑wise; we repeat this \(R\) times with independent seeds. 
Unless otherwise noted in each experiment’s caption, we use \(R=20\) and \(N=500\).

\begin{remark}\label{remark:model_linear}
    We also tested models with a common linear signal by taking a random unit vector \(\beta^*\) and setting \(\beta^*_1=\beta^*_0=\theta\,\beta^*\) with \(\theta\ge0\). 
    Preliminary runs with stronger common linear signal (i.e., $y^{(\arm)}=\theta\,X\beta^\star+\varepsilon^{(\arm)}$ with $\theta>0$) gave qualitatively similar results and did not materially change our conclusions to follow, hence they are omitted for brevity. 
    A more extensive investigation is deferred to future work.
\end{remark}

\paragraph{Oracle Quantities \(\Vstar, \Rstar, \Bstar\) and Monte Carlo Estimation.} 
We consider two ATE estimators: 
\begin{enumerate}[label=(\roman*)]
    \item 
    the difference-in-means estimator \(\hat\tau_{\DIM}\), cf. \eqref{eqn:DiM}, and
    \item 
    the OLS-adjusted estimator \(\hat\tau_{\ols}\), cf. \eqref{eqn:RA_OLS}. 
\end{enumerate}
For each estimator, we estimate the concentration parameters \(\Vstar, \Rstar\), cf. \eqref{eqn:oracle_concentration_param}, and the bias parameter \(\Bstar\), cf. \eqref{eqn:oracle_bias_param}. 
For DiM, the remaining-pool oracle quantities \(V^\star\) and \(R^\star\) admit closed forms, and \(B^\star=0\). 
For OLS--RA, we estimate \((V^\star,R^\star,B^\star)\) by Monte Carlo, yielding \((\widehat{\Vstar},\widehat{\Rstar},\widehat{\Bstar})\); see Appendix~\ref{sec:experiments_monte_carlo} for details.

We recall that with the assignment‑exposure filtration \((\cF_t)\) and martingale \(M_t=\bbE[f(\Sone)\mid \cF_t]\), the increment sequence is \(D_t \coloneqq M_t-M_{t-1}\). 
Thus the realized predictable quadratic variation (PQV) and the local range along a reveal order \(\Pi\) are
\[
    V(\cS_1,\Pi)  =  \sum_{t=1}^{n_1} \Var \big(D_t\mid\cF_{t-1}\big),
    \qquad
    R(\cS_1,\Pi)  =  \max_{1\le t\le n_1} \,\esssup \bigl(\,|D_t| \,\bigm|\, \cF_{t-1}\bigr),
\]
cf.\ \eqref{eqn:oracle_conc_components}. 
By the reveal–swap identity \eqref{eq:rev-swap-avg}, \(D_t = -\,\alpha_t\,\zeta_t(I)\) with
\(\alpha_t = \frac{n_0}{\,n-t+1\,}\) and \(\zeta_t(i)=\bbE_{J,\cT}[\Delta_{iJ} f(\Sprox{t-1}(i,\cT))\mid\cF_{t-1}]\) (cf.\ \eqref{eqn:zeta_t}). 
Consequently,
\[
    \Vstar=\sum_{t=1}^{n_1}\alpha_t^2\,\Var \big(\zeta_t(I)\big),
    \qquad
    \Rstar=\max_{1\le t\le n_1}\alpha_t\max_{i\in\cR_{t-1}}|\zeta_t(i)|,
\]
which Algorithm~\ref{alg:mc-varrange} (presented in Appendix~\ref{sec:experiments_monte_carlo}) approximates via Rao–Blackwellization over controls \(J\).

Also, recall that the bias parameter is 
\[
    \Bstar
    =
    \frac{\sqrt{\bbE\big(Lf+\lambda^\star f\big)^2}}{\lambda^\star},
    \qquad
    \lambda^\star
    =
    \frac{\bbE\Gamma(f)}{\Var(f)},
\]
where
\[
    \Gamma(f)(\cS_1)
    =
    \frac{1}{2n_1 n_0}\sum_{i\in\cS_1}\sum_{j\in\cS_0}
    \big(\Delta_{ij} f(\cS_1)\big)^2,
    \qquad
    (Lf)(\cS_1)
    =
    \frac{1}{n_1n_0}\sum_{i\in\cS_1}\sum_{j\in\cS_0}
    \Delta_{ij} f(\cS_1).
\]
Algorithm~\ref{alg:mc-bias} (see Appendix~\ref{sec:experiments_monte_carlo}) estimates \(\bbE\Gamma(f)\), \(\lambda^\star\), and the residual norm \(\bbE(Lf+\lambda^\star f)^2\), yielding the Monte Carlo estimate \(\widehat{\Bstar}\).

Once the required oracle quantities are available—exactly in the DiM case, and via \((\widehat{\Vstar},\widehat{\Rstar},\widehat{\Bstar})\) in the OLS--RA case—for any given \(\delta \in (0,1)\), for each instantiated treatment set \(\cS_1^{(k)}\), we compute the oracle finite-sample CI for $\tau$ of desired coverage \(1-\delta\) using \eqref{eqn:main_oracle} as
\begin{equation}\label{eqn:CI_radius}
    \mathrm{CI}_{\delta}(\cS_1^{(k)}) \coloneqq \left[ \, \hat{\tau} - \widehat r_\delta(\cS_1^{(k)}),~ \hat{\tau} + \widehat r_\delta(\cS_1^{(k)}) \,\right]
    \qquad\text{where}\qquad
    \widehat r_\delta(\cS_1^{(k)}) \coloneqq \sqrt{2 \widehat{\Vstar} \, \log\frac{2}{\delta} }+\frac{ \widehat{\Rstar}}{3} \, \log\frac{2}{\delta} + \widehat{\Bstar}.
\end{equation}

\begin{remark}
    Here, the concentration parameter estimates \(\widehat{\Vstar} = \widehat{\Vstar}(\cS_1^{(k)})\) and \(\widehat{\Rstar} = \widehat{\Rstar}(\cS_1^{(k)})\) depend on the particular treatment assignment \(\cS_1^{(k)}\), whereas the bias parameter estimate \(\widehat{\Bstar}\) does not.
\end{remark}

\paragraph{Evaluation and Diagnostics.} 
For each instance \( k \in [N] \), we compute the ATE estimator \( \hat\tau \in \{ \hat\tau_{\DIM}, \hat\tau_{\ols} \} \). 
Let $\{\hat\tau_k\}_{k=1}^N$ be the $N$ estimates of a given method. 
At \(\delta = 0.05\), we report the following performance metrics:
\begin{itemize}
    \item 
    \emph{Coverage rate:} 
    \[
        \CR_\delta= \frac{1}{N}\sum_{k=1}^N \Ind \bigl\{\,|\hat\tau_k-\tau|\le \widehat r_\delta(\cS_1^{(k)})\,\bigr\}.
    \]
    \item 
    \emph{Aggregate CI radii:} 
    the mean, median, and variance of the CI radius \(\widehat r_\delta\) across $N$ assignments.
    \item 
    \emph{Inter‑quantile benchmark:} 
    \(q_{0.975}-q_{0.025}\), the empirical 97.5-th percentile of \(\hat\tau-\tau\) minus 2.5-th percentile of \(\hat\tau-\tau\) across the \(N\) assignments.
\end{itemize}
These metrics are used to assess how tight or loose the constructed confidence intervals are.

Aggregation of results for reporting is nested and experiment-specific as follows.  
For Experiment~1, coverage is averaged across outer draws, while CI widths and inter-percentile widths are summarized by the median of the within-draw means; the corresponding parenthetical entries report the matching within-draw or across-draw variance summaries. 
For Experiment~2, each outer draw is first summarized over its \(N\) assignments, and the reported entries are medians across the \(R\) outer draws, with variances across outer draws in parentheses. 
Throughout, we report \emph{CI widths} (twice the radius), not the radius.

In addition to coverage and width summaries, we report diagnostic quantities that help interpret the Monte Carlo estimates. 
For DiM, we compare the exact remaining-pool oracle quantities with terminal-assignment diagnostics computed after conditioning on the realized final treatment/control split. 
For OLS--RA, we report three algorithm-tied diagnostics---\(V_{\mathrm{PQV}}\), \(R_{\mathrm{swap}}\), and \(B_{\mathrm{emp}}\)---all computed from the same sampled candidate units, completion sets, and admissible controls that produce \((\widehat{\Vstar},\widehat{\Rstar},\widehat{\Bstar})\). 
Here PQV stands for \emph{predictable quadratic variation}. 
These diagnostics are not alternative oracle parameters and should not be interpreted as upper or lower bounds for \((\Vstar,\Rstar,\Bstar)\); rather, they are internal reference scales that help identify where slack enters in the Monte Carlo procedure.
In particular, the current implementation forms the direct MC estimate \(\widehat{\Rstar}\) from the averaged stepwise quantities and records \(R_{\mathrm{swap}}\) from the corresponding raw sampled swap effects.

\paragraph{Computing Environment.}
All experiments were run on a single Apple Mac mini (M4 Pro, Apple silicon) under macOS using Python (NumPy/SciPy; CPU only).

\subsubsection{Experiment 1: Finite-sample versus Asymptotic CIs for Difference-in-Means}\label{sec:exp1_validity_dim}
The difference-in-means (DiM) estimator is unbiased under complete randomization and admits asymptotic normality with the classical Neyman variance formula. 
It is also useful here because, for a realized reveal order \(\Pi\), the remaining-pool oracle quantities \(V^\star\) and \(R^\star\) admit explicit closed forms.

To state this precisely, we define
\[
    a_i \coloneqq \frac{y_i^{(1)}}{n_1}+\frac{y_i^{(0)}}{n_0},
    \qquad
    m_t \coloneqq |\cR_{t-1}| = n-t+1,
    \qquad
    \bar a_{t-1} \coloneqq \frac{1}{m_t}\sum_{u\in\cR_{t-1}} a_u .
\]
Then, for each \(i\in\cR_{t-1}\),
\[
    \zeta_t(i)
    =
    \frac{\sum_{u\in\cR_{t-1}} a_u - m_t a_i}{m_t-1}
    =
    -\,\frac{m_t}{m_t-1}\bigl(a_i-\bar a_{t-1}\bigr).
\]
Since \(\alpha_t = n_0/(n-t+1)=n_0/m_t\), it follows that
\begin{align}
    V^\star
    &=
    \sum_{t=1}^{n_1}\alpha_t^2
    \Var_{i\sim\Unif(\cR_{t-1})} \bigl(\zeta_t(i)\bigr)
    =
    \sum_{t=1}^{n_1}
    \frac{n_0^2}{(m_t-1)^2}
    \cdot
    \frac{1}{m_t}
    \sum_{i\in\cR_{t-1}}
    \bigl(a_i-\bar a_{t-1}\bigr)^2, \label{eqn:Vstar_rp}\\
    R^\star
    &=
    \max_{1\le t\le n_1}\alpha_t \max_{i\in\cR_{t-1}} |\zeta_t(i)|
    =
    \max_{1\le t\le n_1}
    \frac{n_0}{m_t-1}
    \max_{i\in\cR_{t-1}}
    \bigl|a_i-\bar a_{t-1}\bigr|.   \label{eqn:Rstar_rp}
\end{align}
Thus, for the DiM estimator, \(V^\star\) and \(R^\star\) are computable exactly from the remaining-pool means, variances, and maximal deviations of the score sequence \(\{a_i\}_{i=1}^n\) along the reveal path. 

In Experiment~1, we use the DiM estimator setting as a sanity check for the finite-sample (FS) construction, asking: 
(i) whether the proposed FS CIs achieve at least nominal coverage; 
(ii) how conservative their widths are relative to Wald intervals and empirical randomization dispersion; and 
(iii) how the remaining-pool oracle quantities scale with \(n\).

\paragraph{Setup and Evaluation.} 
We follow the general data-generation protocol in Section~\ref{sec:exp_setup}. 
DiM does not depend on covariates, so we vary only $n\in\{10,20,40,80,160,320,640\}$ and set $\gamma=0$ (hence $p=1$). 
Here we set \(\theta=0\), so \(y^{(\arm)}=\varepsilon^{(\arm)}\) and the finite‑population estimand is \(\tau=n^{-1}\sum_{i=1}^n(y^{(1)}_i-y^{(0)}_i)\), which is centered at \(0\) and fluctuates at scale \(O_P(n^{-1/2})\).

For each instantiated finite population of potential outcomes and each complete-randomization assignment, we compute:
\begin{itemize}
    \item 
    \emph{Finite-sample (FS) CI for DiM} with radius
    \[
        r_{\DIM}(\delta) 
            = \sqrt{2\,\Vstar_{\mathrm{rp}} \log \frac{2}{\delta}}
                +\frac{\Rstar_{\mathrm{rp}} }{3}\log \frac{2}{\delta},
    \]
    where \((\Vstar_{\mathrm{rp}},\Rstar_{\mathrm{rp}})\) are the exact remaining-pool oracle quantities for DiM under the reveal law of Proposition~\ref{prop:reveal-swap}; we set \(\Bstar=0\) because DiM is design-unbiased.
    
    \item 
    \emph{Asymptotic Wald (Neyman) CI} with radius
    \[
        r_{\mathrm{Wald}}(\delta) 
            = z_{1-\delta/2}\,    \sqrt{\widehat{\Var}_{\mathrm{Ney}}},\qquad
            \widehat{\Var}_{\mathrm{Ney}}=\frac{s_1^2}{n_1}+\frac{s_0^2}{n_0},
    \]
    where $s_{\arm}^2$ is the sample variance of observed outcomes in arm $\arm$ under the current assignment; $z_{1-\delta/2}$ is the standard normal quantile.

    \item 
    \emph{Empirical inter-percentile width} of $\hat\tau_{\DIM}-\tau$, \( q_{1 - \delta/2} - q_{\delta/2}\). 
    For instance, when $\delta = 0.05$, we compute the empirical 2.5--97.5\% widths of $\hat\tau_{\DIM}-\tau$.
\end{itemize}

\paragraph{Empirical Diagnostics for DiM.} 
For DiM, the reveal law of Proposition~\ref{prop:reveal-swap} yields exact remaining-pool oracle quantities \((V^\star,R^\star)\). 
Separately, we report terminal-assignment diagnostics \((V_{\mathrm{term}},R_{\mathrm{term}})\), obtained by conditioning on the realized final treatment/control split. 
Concretely, we define
\[
    V_{\mathrm{term}}
    =
    \sum_{t=1}^{n_1}\alpha_t^2
    \Var \bigl(\mu_i:\, i\in S_1\setminus\{\Pi(1),\ldots,\Pi(t-1)\}\bigr),
    \qquad\text{and}\qquad
    R_{\mathrm{term}}
    =
    \max_{i\in S_1}\alpha_{\mathrm{pos}(i)}|\mu_i|,
\]
where \(
    \mu_i
    =
    \frac{\overline{y^{(1)}_{S_0}}-y_i^{(1)}}{n_1}
    -
    \frac{y_i^{(0)}-\overline{y^{(0)}_{S_0}}}{n_0},
\) and \(\mathrm{pos}(i)\) is the reveal position of \(i\) in \(\Pi\). 
These terminal diagnostics are inexpensive empirical surrogates for assessing conservatism, but they are not the oracle quantities used in the FS CI. 
They condition on the realized final treatment/control split, whereas the oracle quantities follow the remaining-pool reveal law and therefore account for the randomness of the eventual control set.

\begin{table}[t]
    \centering
    \caption{
    Coverage and width summaries for DiM at \(\delta=0.05\) (Experiment~1). 
    For each \(n\), we report FS and Wald coverage, CI width (i.e., \(2\times\)radius), and the empirical \(2.5\%\)--\(97.5\%\) width of \(\hat\tau_{\DIM}-\tau\), computed from \(R=20\) outer replicates and \(N=500\) assignments per replicate. 
    For coverage and percentile widths, the parenthetical entries are variances across the \(R\) outer replicates of the replicate-level summaries. 
    For CI widths, the parenthetical entries are the medians across the \(R\) outer replicates of the within-replicate width variances.
    }
    \label{tab:exp1_validity}
    \begin{tabular}{l   c c     c c     c}
        \toprule
        $n$     & \multicolumn{2}{c}{Finite-sample CI}  & \multicolumn{2}{c}{Asymptotic CI}  & Inter-percentile range \\
        \cmidrule(lr){2-3} \cmidrule(lr){4-5}
                &  Coverage  & CI width     & Coverage  & CI width     & $q_{0.975}-q_{0.025}$ \\
        \midrule
        10  & 1.000 (0.000) & 4.413 (0.108) & 0.945 (0.002) & 2.504 (0.601) & 1.903 (0.236) \\
        20  & 1.000 (0.000) & 2.933 (0.023) & 0.976 (0.000) & 1.933 (0.093) & 1.426 (0.075) \\
        40  & 1.000 (0.000) & 2.068 (0.006) & 0.983 (0.000) & 1.423 (0.027) & 1.087 (0.023) \\
        80  & 1.000 (0.000) & 1.256 (0.001) & 0.987 (0.000) & 0.948 (0.006) & 0.697 (0.004) \\
        160 & 0.999 (0.000) & 0.885 (0.000) & 0.988 (0.000) & 0.691 (0.002) & 0.532 (0.001) \\
        320 & 0.998 (0.000) & 0.593 (0.000) & 0.987 (0.000) & 0.481 (0.000) & 0.364 (0.000) \\
        640 & 0.999 (0.000) & 0.404 (0.000) & 0.990 (0.000) & 0.338 (0.000) & 0.258 (0.000) \\
        \bottomrule
    \end{tabular}
\end{table}

\begin{table}[t]
    \centering
    \caption{
    Summary of remaining-pool oracle quantities and terminal-assignment diagnostics for DiM (Experiment~1). 
    For each \(n\), we report medians across outer draws, with variances across outer draws in parentheses. 
    Here, the subscript \(\mathrm{rp}\) in \(\Vstar_{\mathrm{rp}}\) and \(\Rstar_{\mathrm{rp}}\) indicates they are computed using the remaining-pool identities \eqref{eqn:Vstar_rp}, \eqref{eqn:Rstar_rp}.
    }
    \label{tab:exp1_params}
    \begin{tabular}{l   c c     c c }
        \toprule
        $n$     & \multicolumn{2}{c}{Variance} & \multicolumn{2}{c}{Range}  \\
        \cmidrule(lr){2-3} \cmidrule(lr){4-5}
        &   \(\Vstar_{\mathrm{rp}}\) & \(V_{\mathrm{term}}\) & \(\Rstar_{\mathrm{rp}}\) & \(R_{\mathrm{term}}\) \\
        \midrule
        10  & 0.257 (0.022) & 0.055 (0.001) & 0.657 (0.043) & 0.386 (0.009) \\
        20  & 0.143 (0.003) & 0.063 (0.001) & 0.385 (0.005) & 0.267 (0.002) \\
        40  & 0.082 (0.000) & 0.049 (0.000) & 0.211 (0.002) & 0.153 (0.000) \\
        80  & 0.033 (0.000) & 0.025 (0.000) & 0.111 (0.000) & 0.082 (0.000) \\
        160 & 0.018 (0.000) & 0.016 (0.000) & 0.065 (0.000) & 0.048 (0.000) \\
        320 & 0.009 (0.000) & 0.008 (0.000) & 0.036 (0.000) & 0.027 (0.000) \\
        640 & 0.004 (0.000) & 0.004 (0.000) & 0.019 (0.000) & 0.014 (0.000) \\
        \bottomrule
    \end{tabular}
\end{table}

\paragraph{Results.}
Table~\ref{tab:exp1_validity} summarizes coverage and CI widths, and Table~\ref{tab:exp1_params} compares the exact remaining-pool oracle quantities with terminal-assignment diagnostics.

The FS CIs achieve at least nominal coverage (which is 0.95 here) across all reported \(n\), but they are conservative. 
The FS widths exceed the Wald widths throughout; for example, at \(n=10\), the FS width is \(4.413\) versus Wald width \(2.504\), while the empirical \(2.5\%\)--\(97.5\%\) width of \(\hat\tau_{\DIM}-\tau\) is \(1.903\). 
Thus Experiment~1 should be read primarily as a validity and conservatism sanity check for the general finite-sample bound. 

Table~\ref{tab:exp1_params} helps localize, though not fully identify, the source of this conservatism. 
The exact remaining-pool oracle quantities \(V^\star\) and \(R^\star\) are noticeably larger than the terminal-assignment diagnostics \(V_{\mathrm{term}}\) and \(R_{\mathrm{term}}\), especially at small \(n\); for example, at \(n=10\), \(V^\star=0.257\) versus \(V_{\mathrm{term}}=0.055\), and \(R^\star=0.657\) versus \(R_{\mathrm{term}}=0.386\). 
Because \(V^\star\) and \(R^\star\) are computed exactly here, this gap is not a Monte Carlo artifact. 
At the same time, \(V_{\mathrm{term}}\) and \(R_{\mathrm{term}}\) are only diagnostic surrogates---not oracle targets and not quantities known to yield valid confidence intervals---so these comparisons should not be interpreted as calibration identities. 
Rather, they suggest that a substantial part of the conservatism may already enter through the remaining-pool reveal construction, and perhaps more fundamentally through the martingale/Freedman concentration step. 
Determining whether this looseness is intrinsic to this line of argument, or can be reduced by a sharper finite-population concentration method, remains an interesting direction for future work.

\subsubsection{Experiment 2: Finite-sample CIs for OLS--RA}\label{sec:exp2_ra_ci}

\paragraph{Objective.} 
Having completed basic sanity checks with the DiM estimator, here we assess finite-sample (FS) confidence intervals (CIs) for the OLS regression-adjusted estimator (OLS--RA). 
Our goals in this experiment are:
(i) to ensure at least nominal coverage across various \( (n,\gamma) \); and
(ii) to quantify the tightness or conservatism of the FS-CI widths relative to an asymptotic DiM baseline and the empirical \(2.5\%\)--\(97.5\%\) inter-percentile range of $\hat\tau_{\ols}-\tau$.

\paragraph{Setup and Evaluation.}
We follow the general data-generation protocol in Section~\ref{sec:exp_setup}. 
We consider \(n\in\{25,50\}\) and \(\gamma\in\{0,0.25,0.5,0.75,1.0,1.25,1.5\}\), so that \(p=\lceil n^\gamma\rceil\). 
For the reported run, we use \(R=20\), \(N=500\), and MC budgets \((B_S,B_{\mathrm{pair}},B_i,B_{\mathrm{cond}},B_J)=(30,30,10,10,10)\).
The bias parameter \(\widehat{\Bstar}\) is estimated once per finite-population instance using Algorithm~\ref{alg:mc-bias}, while \((\widehat{\Vstar},\widehat{\Rstar})\) are estimated assignment-wise using Algorithm~\ref{alg:mc-varrange}. 

\paragraph{Empirical Diagnostics for OLS--RA.}  
Recall \(\widehat\zeta_t(i)\) and the within-\(i\) completion variance \(\widehat s_t^2(i)\) from Section~\ref{sec:MC_concentration}. 
We report three diagnostics:
\begin{align}
    V_{\mathrm{PQV}}
        &\coloneqq
        \sum_{t=1}^{n_1}
        \alpha_t^2\,
        \widehat{\Var}_{i\in\cI_t} \big(\widehat\zeta_t(i)\big),   \label{eqn:V_PQV}\\
    R_{\mathrm{swap}}
        &\coloneqq
        \max_{1\le t\le n_1}
        \alpha_t
        \max_{i\in\cI_t}
        \max_{1\le b\le B_{\mathrm{cond}}}
        \max_{J\in\cJ_{b,i}}
        \bigl|
        \Delta_{iJ}f\big(\Sprox{t-1}(i,\cT_{b,i})\big)
        \bigr|,     \label{eqn:R_swap}\\
    B_{\mathrm{emp}}
        &\coloneqq
        \left|
        \frac{1}{N}\sum_{k=1}^N \hat\tau^{(k)}-\tau
        \right|.    \label{eqn:B_emp}
\end{align}

These quantities play different roles. 
\begin{itemize}
    \item 
    \(V_{\mathrm{PQV}}\) is the \emph{predictable quadratic variation} (PQV) before the within-\(i\) Monte Carlo noise correction. 
    In the current implementation,
    \[
        \widehat{\Vstar}
        =
        \sum_{t=1}^{n_1}\alpha_t^2
        \left[
            \widehat{\Var}_{i\in\cI_t} \big(\widehat\zeta_t(i)\big)
            -
            \frac{1}{|\cI_t|}\sum_{i\in\cI_t}\frac{\widehat s_t^2(i)}{B_{\mathrm{cond}}}
        \right]_+ ,
    \]
    where \([z]_+ = \max\{z, 0\}\). 
    Thus \(V_{\mathrm{PQV}}\) is not a separate oracle target; it is the pre-de-noising version of the variance aggregate built from the same sampled \(\widehat\zeta_t(i)\) values as \(\widehat{\Vstar}\). 
    Hence \(V_{\mathrm{PQV}}-\widehat{\Vstar}\) measures the size of the within-\(i\) Monte Carlo noise correction, rather than conservatism relative to the oracle \(V^\star\). 

    \item 
    \(R_{\mathrm{swap}}\) records the sampled raw-swap scale before conditional averaging, whereas \(\widehat{\Rstar}\) is built from the averaged quantities \(\widehat\zeta_t(i)\). 
    Thus \(R_{\mathrm{swap}}-\widehat{\Rstar}\) measures the smoothing induced by Rao--Blackwellization and completion averaging. 
    When \(B_J>0\), \(R_{\mathrm{swap}}\) is not an exact all-control envelope, but a diagnostic computed from the same sampled objects as \(\widehat{\Rstar}\).

    \item 
    \(B_{\mathrm{emp}}\) provides an empirical reference scale for the bias proxy \(\widehat{\Bstar}\). 
\end{itemize}

These diagnostics are useful because they are computed from the same sampled objects as \((\widehat{\Vstar},\widehat{\Rstar},\widehat{\Bstar})\), so they isolate distinct components of the implemented Monte Carlo pipeline. 
Unlike the DiM case, we do not report terminal-assignment diagnostics for OLS--RA, because there is no comparably simple closed-form terminal analogue.

\begin{table}[ht!]
    \centering
    \caption{
    Coverage and width summaries for OLS--RA at \(\delta=0.05\) (Experiment~2). 
    For each \((n,\gamma)\), we report FS and Wald coverage, CI width (i.e., \(2\times\)radius), and the empirical \(2.5\%\)--\(97.5\%\) width of \(\hat\tau_{\ols}-\tau\), obtained by first summarizing each outer replicate over its \(N=500\) assignments and then taking medians across the \(R=20\) outer replicates. 
    For coverage, the parenthetical entries are variances across outer replicates of the replicate-level coverage rates. 
    For CI widths, the parenthetical entries are the medians across outer replicates of the within-replicate width variances. 
    For the empirical inter-percentile widths, the parenthetical entries are variances across outer replicates of the replicate-level IPRs.
    }
    \label{tab:exp2_finite_CI_RA}
    \begin{tabular}{l c c   c c     c c     c}
        \toprule
        \(n\)   & \(\gamma\)    & \(p\)     & \multicolumn{2}{c}{Finite-sample CI (RA)}  & \multicolumn{2}{c}{Asymptotic CI (DiM)}  & Inter-percentile range \\
        \cmidrule(lr){4-5} \cmidrule(lr){6-7}
            & & &  Coverage  & CI width    & Coverage  & CI width     & $q_{0.975}-q_{0.025}$ \\
        \midrule
        25 & 0.00 & 1   & 1.000 (0.000) & 2.982 (0.212) & 0.953 (0.000) & 2.413 (0.083) & 1.316 (0.036) \\
        25 & 0.25 & 3   & 1.000 (0.000) & 4.100 (0.879) & 0.960 (0.000) & 2.389 (0.104) & 1.755 (0.039) \\
        25 & 0.50 & 5   & 1.000 (0.000) & 6.380 (14.157) & 0.948 (0.000) & 2.310 (0.123) & 2.453 (0.065) \\
        25 & 0.75 & 12  & 1.000 (0.000) & 5.874 (1.993) & 0.961 (0.000) & 2.327 (0.092) & 2.313 (0.052) \\
        25 & 1.00 & 25  & 1.000 (0.000) & 4.450 (0.385) & 0.958 (0.000) & 2.359 (0.111) & 1.849 (0.076) \\
        25 & 1.25 & 56  & 1.000 (0.000) & 4.922 (0.195) & 0.952 (0.000) & 2.540 (0.132) & 2.103 (0.078) \\
        25 & 1.50 & 125 & 1.000 (0.000) & 4.313 (0.112) & 0.960 (0.000) & 2.240 (0.097) & 1.887 (0.058) \\
        \midrule
        50 & 0.00 & 1   & 1.000 (0.000) & 2.194 (0.031) & 0.966 (0.000) & 1.788 (0.030) & 0.965 (0.013) \\
        50 & 0.25 & 3   & 1.000 (0.000) & 2.242 (0.082) & 0.961 (0.000) & 1.660 (0.027) & 0.970 (0.020) \\
        50 & 0.50 & 8   & 1.000 (0.000) & 3.284 (0.535) & 0.967 (0.000) & 1.707 (0.031) & 1.358 (0.042) \\
        50 & 0.75 & 19  & 1.000 (0.000) & 4.811 (1.766) & 0.967 (0.000) & 1.704 (0.029) & 1.862 (0.060) \\
        50 & 1.00 & 50  & 1.000 (0.000) & 3.399 (0.121) & 0.962 (0.000) & 1.727 (0.034) & 1.419 (0.014) \\
        50 & 1.25 & 133 & 1.000 (0.000) & 3.228 (0.053) & 0.970 (0.000) & 1.683 (0.033) & 1.434 (0.049) \\
        50 & 1.50 & 354 & 1.000 (0.000) & 3.254 (0.033) & 0.964 (0.000) & 1.704 (0.032) & 1.454 (0.024) \\
        \bottomrule
    \end{tabular}
\end{table}

\begin{table}[ht!]
    \centering
    \caption{
    Summary of Monte Carlo estimates and empirical diagnostics for OLS--RA (Experiment~2). 
    For each \((n,\gamma)\), we report medians across the \(R\) outer replicates, with variances across outer replicates in parentheses. 
    We do not report terminal-assignment diagnostics here because, unlike the DiM case, OLS--RA has no comparably simple closed-form terminal analogue; instead we report the algorithm-tied diagnostics \(V_{\mathrm{PQV}}\), \(R_{\mathrm{swap}}\), and \(B_{\mathrm{emp}}\).
    }
    \label{tab:exp2_diagnostics}

    \begin{adjustbox}{max width=\linewidth}
    \begin{tabular}{l c c   c c     c c     c c }
        \toprule
        \(n\)   & \(\gamma\)    & \(p\)     & \multicolumn{2}{c}{Variance (assignment-level)}  & \multicolumn{2}{c}{Range}     & \multicolumn{2}{c}{Bias}  \\
        \cmidrule(lr){4-5} \cmidrule(lr){6-7} \cmidrule(lr){8-9} 
           & &  &   \(\widehat{\Vstar}\) & \(V_{\mathrm{PQV}}\) & \(\widehat{\Rstar}\) & \(R_{\mathrm{swap}}\)   & \(\widehat{\Bstar}\) & \(B_{\mathrm{emp}}\) \\
        \midrule
        25 & 0.00 & 1   & 0.117 (0.001) & 0.119 (0.001) & 0.311 (0.003) & 0.891 (0.020) & 0.196 (0.004) & 0.021 (0.000) \\
        25 & 0.25 & 3   & 0.175 (0.002) & 0.186 (0.002) & 0.423 (0.002) & 1.913 (0.030) & 0.253 (0.024) & 0.035 (0.001) \\
        25 & 0.50 & 5   & 0.295 (0.003) & 0.380 (0.005) & 0.703 (0.006) & 7.280 (0.862) & 0.381 (0.059) & 0.043 (0.003) \\
        25 & 0.75 & 12  & 0.338 (0.006) & 0.372 (0.007) & 0.674 (0.004) & 2.323 (0.065) & 0.359 (0.007) & 0.062 (0.002) \\
        25 & 1.00 & 25  & 0.239 (0.005) & 0.248 (0.005) & 0.470 (0.005) & 1.096 (0.022) & 0.306 (0.008) & 0.030 (0.001) \\
        25 & 1.25 & 56  & 0.320 (0.005) & 0.324 (0.006) & 0.477 (0.008) & 0.954 (0.029) & 0.326 (0.003) & 0.035 (0.000) \\
        25 & 1.50 & 125 & 0.255 (0.004) & 0.257 (0.004) & 0.444 (0.011) & 0.782 (0.017) & 0.254 (0.003) & 0.024 (0.000) \\
        \midrule
        50 & 0.00 & 1   & 0.062 (0.000) & 0.062 (0.000) & 0.179 (0.001) & 0.390 (0.002) & 0.198 (0.002) & 0.011 (0.000) \\
        50 & 0.25 & 3   & 0.065 (0.000) & 0.066 (0.000) & 0.202 (0.001) & 0.535 (0.006) & 0.198 (0.004) & 0.009 (0.000) \\
        50 & 0.50 & 8   & 0.112 (0.001) & 0.126 (0.001) & 0.326 (0.002) & 1.366 (0.031) & 0.280 (0.015) & 0.026 (0.001) \\
        50 & 0.75 & 19  & 0.186 (0.002) & 0.221 (0.003) & 0.506 (0.004) & 2.434 (0.108) & 0.389 (0.011) & 0.026 (0.001) \\
        50 & 1.00 & 50  & 0.134 (0.000) & 0.139 (0.001) & 0.294 (0.001) & 0.683 (0.004) & 0.332 (0.005) & 0.026 (0.000) \\
        50 & 1.25 & 133 & 0.142 (0.002) & 0.144 (0.002) & 0.267 (0.002) & 0.535 (0.007) & 0.269 (0.004) & 0.011 (0.000) \\
        50 & 1.50 & 354 & 0.149 (0.001) & 0.150 (0.001) & 0.265 (0.001) & 0.505 (0.004) & 0.280 (0.004) & 0.017 (0.000) \\
        \bottomrule
    \end{tabular}
    \end{adjustbox}
\end{table}

\paragraph{Results.}
Table~\ref{tab:exp2_finite_CI_RA} reports coverage and widths for OLS--RA FS CIs and DiM Wald CIs, together with the empirical \(2.5\%\)--\(97.5\%\) width of \(\hat\tau_{\ols}-\tau\). 
Table~\ref{tab:exp2_diagnostics} reports the corresponding MC estimates and diagnostics.

The OLS--RA FS intervals attain empirical coverage \(1.000\) in all reported cells, which are always greater than the nominal coverage 0.95 but are conservative. 
The FS widths exceed the empirical inter-percentile widths throughout; in this run, the width/IPR ratio ranges from about \(2.24\) to \(2.60\). 
For example, at \(n=25,\gamma=0.50\), the FS width is \(6.380\) whereas the empirical IPR is \(2.453\), and at \(n=50,\gamma=0.75\), the corresponding values are \(4.811\) and \(1.862\).

The diagnostics are best interpreted structurally rather than as direct lower or upper bounds on the oracle quantities. 
By construction, \(V_{\mathrm{PQV}}\) is the pre-de-noising variance aggregate built from the same sampled \(\widehat\zeta_t(i)\) values used in \(\widehat{\Vstar}\); therefore the gap \(V_{\mathrm{PQV}}-\widehat{\Vstar}\) measures the size of the within-\(i\) noise correction, not conservatism relative to the true \(V^\star\). 
Likewise, \(R_{\mathrm{swap}}\) records the raw sampled swap scale before conditional averaging, whereas \(\widehat{\Rstar}\) is built from the averaged stepwise quantities \(\widehat\zeta_t(i)\); the gap between them quantifies the smoothing induced by Rao--Blackwellization. 
The fact that \(\widehat{\Bstar}\) exceeds \(B_{\mathrm{emp}}\) in every reported cell shows that the bias proxy is also conservative in these experiments.

Finally, the empirical IPR is largest in moderate-dimensional regimes—near \(\gamma=0.50\) when \(n=25\), and near \(\gamma=0.75\) when \(n=50\). 
This is consistent with the \(Q\)-geometry interpretation in Section~\ref{sec:main_OLS_RA}: covariates can reduce residual dispersion, while deletion and insertion sensitivities become largest near the arm-wise interpolation threshold. 
Larger-budget reruns should reduce MC noise, but are unlikely to change these qualitative conclusions.

\end{document}